\documentclass[11pt]{amsart}
\usepackage[margin=1.2in,marginparwidth=1in]{geometry}

\usepackage{amsmath,amssymb,amsthm}
\usepackage{mathtools}
\usepackage{bbm}
\usepackage{chngcntr}

\usepackage{todonotes}
\usepackage{verbatim}

\usepackage{stmaryrd}

\usepackage{graphicx}

\usepackage{tikz}
\usetikzlibrary{arrows,matrix}
\tikzset{arrow/.style={semithick,>=stealth',shorten >=1pt,shorten <=1pt}}

\usepackage{tikz-cd}
\usepackage[all]{xy}

\usepackage[pagebackref,
			colorlinks, 
			citecolor=forestgreen, 
			linkcolor=forestgreen, 
			urlcolor=darkblue]{hyperref} 
\usepackage[nameinlink, capitalize]{cleveref}

\definecolor{forestgreen}{RGB}{0,158,115}	
\definecolor{darkblue}{RGB}{0,114,178}	

\numberwithin{equation}{section}
\numberwithin{figure}{section}
\numberwithin{table}{section}
\allowdisplaybreaks

\newtheorem{theorem}{Theorem}[section]
\newtheorem{prop}[theorem]{Proposition}
\AddToHook{env/prop/begin}{\crefalias{theorem}{proposition}}
\newtheorem{cor}[theorem]{Corollary}
\AddToHook{env/cor/begin}{\crefalias{theorem}{corollary}}
\newtheorem{lemma}[theorem]{Lemma}
\AddToHook{env/lemma/begin}{\crefalias{theorem}{lemma}}

\AddToHook{env/claim/begin}{\crefalias{theorem}{claim}}

\newtheorem*{notate}{Notation}

\theoremstyle{definition}
\newtheorem{definition}[theorem]{Definition}
\AddToHook{env/definition/begin}{\crefalias{theorem}{definition}}
\newtheorem{notation}[theorem]{Notation}
\AddToHook{env/notation/begin}{\crefalias{theorem}{notation}}

\AddToHook{env/question/begin}{\crefalias{theorem}{question}}

\AddToHook{env/conjecture/begin}{\crefalias{conjecture}{definition}}
\newtheorem{example}[theorem]{Example}
\AddToHook{env/example/begin}{\crefalias{theorem}{example}}

\AddToHook{env/examples/begin}{\crefalias{theorem}{examples}}

\AddToHook{env/note/begin}{\crefalias{theorem}{note}}
\newtheorem{remark}[theorem]{Remark}
\AddToHook{env/remark/begin}{\crefalias{theorem}{remark}}

\usepackage[shortlabels]{enumitem} 

\makeatletter
\DeclareDocumentCommand{\newmathcommand}{ m O{0} m }{%
  \ifcsname\expandafter\@gobble\string#1\space\endcsname
    \expandafter\expandafter\expandafter\let\expandafter\csname old\string#1\expandafter\endcsname\expandafter=\csname\expandafter\@gobble\string#1\space\endcsname
  \else
    \expandafter\let\csname old\string#1\endcsname=#1
  \fi
  \expandafter\newcommand\csname new\string#1\endcsname[#2]{#3}
  \DeclareRobustCommand#1{%
    \ifmmode
      \expandafter\let\expandafter\next\csname new\string#1\endcsname
    \else
      \expandafter\let\expandafter\next\csname old\string#1\endcsname
    \fi
    \next
  }%
}
\makeatother
\newmathcommand{\r}{\mathring}

\tikzset{
    sideways/.style={anchor=center, rotate=90, inner sep=.5mm}
}

\renewcommand{\tilde}{\widetilde}

\DeclareMathOperator{\Hom}{Hom}
\DeclareMathOperator{\Map}{Map}

\DeclareMathOperator{\Aut}{Aut}

\DeclareMathOperator{\Ho}{Ho}

\DeclareMathOperator{\Spec}{Spec}
\DeclareMathOperator{\Spf}{Spf}
\DeclareMathOperator{\Sub}{Sub}
\DeclareMathOperator{\Div}{PD}
\DeclareMathOperator{\Level}{Level}
\DeclareMathOperator{\Sum}{Sum}

\DeclareMathOperator{\GL}{GL}

\DeclareMathOperator{\im}{im}
\DeclareMathOperator{\res}{res}

\DeclareMathOperator{\tr}{tr}

\DeclareMathOperator{\Fun}{Fun}

\DeclareMathOperator{\trans}{trans}
\DeclareMathOperator{\conj}{conj}

\DeclareMathOperator{\Sym}{Sym}

\DeclareMathOperator{\Cl}{Cl}

\renewcommand{\u}[1]{\underline{#1}}

\newcommand{\powser}[1]{[\![{#1}]\!]}
\newcommand{\llangle}{\langle\!\langle}
\newcommand{\rrangle}{\rangle\!\rangle}
\newcommand{\divpol}[1]{\langle{#1}\rangle}
\newcommand{\divpowser}[1]{\llangle{#1}\rrangle}
\newcommand{\setof}[1]{\{{#1}\}}

\newcommand{\C}{\mathbb{C}}
\newcommand{\F}{\mathbb{F}}
\newcommand{\G}{\mathbb{G}}

\newcommand{\N}{\mathbb{N}}
\renewcommand{\P}{\mathbb{P}}
\newcommand{\Q}{\mathbb{Q}}

\newcommand{\Z}{\mathbb{Z}}
\renewcommand{\1}{\mathbbm{1}}

\newcommand{\cE}{\mathcal{E}}

\newcommand{\cG}{\mathcal{G}}
\newcommand{\cP}{\mathcal{P}}
\newcommand{\cM}{\mathcal{M}}

\newcommand{\Grp}{\text{Grp}}
\newcommand{\Set}{\text{Set}}

\newcommand{\id}{\textnormal{id}}

\newcommand{\parts}{\vdash}

\newcommand{\mmod}{\mathrel{\!/\!/\!}}

\newcommand{\injto}{\hookrightarrow}

 \newcommand{\cH}{\mathcal{H}}
 \newcommand{\cD}{\mathcal{D}}
 
\DeclareMathOperator{\Asm}{Asm}
\DeclareMathOperator{\Conn}{Conn}
\DeclareMathOperator{\Graph}{Graph}

\makeatletter
\def\l@subsection{\@tocline{2}{0pt}{2.5pc}{5pc}{}}
\makeatother

\title{Partition Functors and Universal Exponential Relations}
\date{\today}

\author{David Mehrle}
\email{david.mehrle@colostate.edu}

\author{Millie Rose}
\email{math@milliero.se}

\author{Nathaniel Stapleton}
\email{nat.j.stapleton@gmail.com}

\begin{document}

\maketitle

\vspace*{-5mm}

\begin{abstract}
We introduce partition functors: algebraic structures indexed by partitions of finite sets and equipped with restriction and transfer maps along refinements. We construct a monad $\Div$ on the category of partition functors and on several categories of partition functors with additional multiplicative structure. For a partition ring $R$, exponential elements in its associated completed ring of symmetric functions acquire canonical logarithms after applying $\Div$. This gives rise to a universal exponential relation between multiplicative and additive power operations. For representation rings this can be used to recover the classical relation between symmetric powers and Adams operations, while for Morava $E$-theory it can be used to recover Ganter's exponential relation between symmetric powers and Hecke operators. We show that the representation rings of products of symmetric groups form the initial partition ring and that, for symmetric monoidal partition functors, the monad $\Div$ is closely related to symmetric invariant tensors and the divided power envelope. We also construct a symmetric monoidal partition ring carrying the universal exponential element, so that its image under $\Div$ carries the universal exponential relation.
\end{abstract}

\tableofcontents

\section{Introduction}

A classical identity in representation theory expresses symmetric powers $\beta^m$ in terms of Adams operations $\psi^k$:
\[
\sum_{m\geq 0}\beta^m(x)t^m
=\exp\bigg (\sum_{k\geq 1}\frac{\psi^k(x)}{k}t^k\bigg ).
\]
Both sides of this formula arise from the power operations $P_m$ sending a $G$-representation $V$ to $V^{\otimes m}$ with $G \times \Sigma_m$-action. The symmetric powers are obtained by applying the transfer along $\Sigma_m \to e$ to $P_m$, while the additive Adams operations are obtained by passing to the quotient by transfers from proper Young subgroups of $\Sigma_m$. Thus the formula relates two different ways of extracting operations from the family of power operations. Although both $\beta^m$ and $\psi^k$ are integral, the exponential formula relating them requires dividing $\psi^k$ by $k$. Similar exponential relations occur in algebraic topology, most notably Ganter's relation between symmetric power operations and Hecke operators in Morava $E$-theory \cite{Ganter-orbifold}. This suggests that the formula reflects an integral structure involving power operations, transfers, and partitions, rather than a special feature of representation theory or Morava $E$-theory. The aim of this paper is to explain why these exponential relations have the same form and to account for the divisibility that makes them integral.

We show that these exponential relations arise because the rings in which they hold carry the same underlying algebraic structure, which we call a partition ring. Partition rings play the role of rings in the category of partition functors and every partition ring has an associated commutative $\N$-graded ring of symmetric functions. We construct a monad $\Div$ on the category of partition rings that associates to each partition ring a partition ring of ``integer-valued class functions.'' This monad is closely related to divided powers. Exponential elements in the ring of symmetric functions of a partition ring admit a canonical logarithm after applying $\Div$. The definition of the monad $\Div$ allows exponential relations involving power operations to be pushed forward to give other exponential relations; for representation rings and Morava $E$-theory this recovers the classical formulas involving Adams operations and Hecke operators. We also construct a partition ring corepresenting exponential elements; applying $\Div$ to this partition ring produces a universal exponential relation.

\bigskip

The character of a finite dimensional complex representation of the symmetric group is integer valued. This fact connects representation theory and combinatorics. It implies that the representation ring of the symmetric group, aspects of which are still mysterious, is a sublattice of the structurally simple ring of integer-valued functions on the set of integer partitions. The character map 
\[
\chi \colon RU(\Sigma_m) \to \Cl(\Sigma_m,\Z)
\]
admits an alternative description entirely in terms of the representation rings of symmetric groups. 

Let  $\lambda \parts m$ be an integer partition of $m$ so that $\sum_i i \lambda_i = m$. Let $\Sigma_{\lambda} = \prod_i (\Sigma_i)^{\lambda_i}$ be the symmetric group associated to $\lambda$ and note that there is an induced group homomorphism $\Sigma_\lambda \to \Sigma_m$ that is well-defined up to conjugacy. Given an inclusion $H \subseteq G$ of finite groups, there is a transfer map (also called induction) $\tr_{H}^{G} \colon RU(H) \to RU(G)$ and a restriction map $\res_{H}^{G} \colon RU(G) \to RU(H)$. The restriction map is a ring map and the image of the transfer map is an ideal. If $I_{\lambda} \subset RU(\Sigma_{\lambda})$ is the sum of the image of the transfer maps along proper refinements of $\lambda$, then there is a (necessarily unique) isomorphism of rings $RU(\Sigma_{\lambda})/I_\lambda \cong \Z$. An alternative description of the value of the character map $\chi$ at $\lambda$ is the composite
\[
RU(\Sigma_m) \xrightarrow{\res_{\Sigma_{\lambda}}^{\Sigma_m}} RU(\Sigma_{\lambda}) \xrightarrow{q_{\lambda}} RU(\Sigma_{\lambda})/I_{\lambda} \cong \Z,
\]
where $q_{\lambda}$ is the quotient. These sorts of maps are known as Brauer morphisms in the literature \cite{Thevenaz}.

The ingredients in this map make sense for an arbitrary cohomology theory. There is one subtlety that doesn't appear in the case of representation rings because $\Z$ admits no nontrivial automorphisms as a commutative ring. For a general cohomology theory, permutations of equal-sized blocks can act nontrivially on these quotients, so the construction must also take invariants. Regardless, for any cohomology theory, there is a notion of ``integer-valued class functions on the symmetric group'' and a canonical map from the cohomology of the symmetric group to these class functions. In the case of mod $2$ singular cohomology, this produces rings built from Dickson algebras. In the case of the representation ring, this recovers integer-valued class functions. In the case of Morava $E$-theory, the resulting ring can be understood algebro-geometrically using work of Strickland \cite{etheorysym}.

The ingredients going into this construction of ``integer-valued class functions'' can be axiomatized and the resulting algebraic structure is quite simple. There is a $(2,1)$-category of partitions that categorifies the notion of an integer partition. A partition $\Lambda=(X,\sim_\Lambda)$ consists of a finite set equipped with an equivalence relation. A map $f\colon\Lambda\to\Omega$ is a bijection of underlying sets such that $x\sim_\Lambda x'$ implies $f(x)\sim_\Omega f(x')$. We write $\Lambda\leq\Omega$ when the underlying sets agree and the identity is such a map, so that $\Lambda$ is a refinement of $\Omega$. These form a $(2,1)$-category $\cP$: there is a unique $2$-isomorphism $f\Rightarrow g$ precisely when $f(x)\sim_\Omega g(x)$ for every $x$. Thus maps differing by permutations within the blocks of the target are $2$-isomorphic. In particular, if $\u m$ for $m \in \N$ is the partition consisting of the set $\{1, \ldots, m\}$ equipped with the indiscrete partition, then any two maps from $\Lambda$ to $\u m$ are $2$-isomorphic. Let $\Sigma_\Lambda$ be the group of permutations of $X$ that preserve each block individually; this is a product of symmetric groups.

Inspired by Schwede's development of global algebra in \cite{Schwede}, we define a partition functor $M$ to consist of a covariant functor $M_*$ and a contravariant functor $M^*$ from $\cP$ to abelian groups that agree on objects and satisfy the double coset formula \eqref{doubleCoset}. Thus a map $f\colon\Lambda\to\Omega$ gives a transfer $M_*(f)\colon M(\Lambda)\to M(\Omega)$ and a restriction $M^*(f)\colon M(\Omega)\to M(\Lambda)$ (which we require to be inverse to each other when $f$ is an isomorphism), and $2$-isomorphic maps induce the same homomorphisms. The double coset formula uses the double cosets of the associated Young subgroups.

If we set $M(\Lambda) = RU(\Sigma_{\Lambda})$, then $M$, together with the induction and restriction maps induced by maps of partitions, is an example of a partition functor. Generalizing, $M(\Lambda) = RU(G \times \Sigma_{\Lambda})$ also admits the structure of a partition functor.  If $E$ is a cohomology theory, then $M(\Lambda) = E^0(B\Sigma_{\Lambda})$, together with transfer and restriction maps induced by maps of partitions, is an example of a partition functor. Given an abelian group $A$, there is a constant partition functor $C_A$ in which restriction maps are identity maps and the transfer map along the inclusion of a refinement $\Omega$ of $\Lambda$ is given by multiplication by the index of $\Sigma_{\Omega}$ in $\Sigma_{\Lambda}$. There is also a dual constant partition functor $C^A$ in which the roles of restriction and transfer maps have been interchanged. If $S$ is a global functor in the sense of \cite{Schwede} and we fix a finite group $G$, then both $M(\Lambda) = S(G \times \Sigma_{\Lambda})$ and $M(\Lambda) = S(G \wr \Sigma_{\Lambda})$, together with transfer and restriction maps coming from $S$, are examples of partition functors.  


We use the disjoint union of partitions to add multiplicative structure to partition functors. A lax symmetric monoidal partition functor $R$ is equipped with a unit $\Z\to R(\u 0)$ and external products
\[
m_{\Lambda,\Omega}\colon R(\Lambda)\otimes R(\Omega)\longrightarrow R(\Lambda\sqcup\Omega),
\]
natural in pairs of restrictions and pairs of transfers and satisfying the associativity, unit, and symmetry identities. A partition ring is such an $R$ for which each $R(\Lambda)$ is a commutative ring, restrictions and external products are ring maps, and transfers satisfy Frobenius reciprocity: for $f\colon\Lambda\to\Omega$, $a\in R(\Omega)$, and $b\in R(\Lambda)$,
\[
R_*(f)\big(R^*(f)(a)*b\big)=a*R_*(f)(b),
\]
where $*$ denotes the multiplication internal to each $R(\Lambda)$. In particular, the subgroup $I_\Lambda \subseteq R(\Lambda)$ given by the sum of the images of transfers from proper refinements is an ideal. Partition rings and partition power functors, which are partition rings equipped with power operations, are the analogues of global Green functors and global power functors in this setting.

The construction of $\Div(M)$ from a partition functor $M$ is central to the story. For $m\in\N$, we set
\[
\Div(M)(\u m)=\Big (\bigoplus_{\Lambda\leq\u m}M(\Lambda)/I_\Lambda\Big)^{\Sigma_m},
\]
where $I_\Lambda$ is the sum of the images of transfers from proper refinements of $\Lambda$, the sum runs over all refinements of $\u m$, and $\Sigma_m$ acts by permuting partitions and through the partition functor structure on $M$. This action is trivial on $M(\u m)/I_{\u m}$, so $M(\u m)/I_{\u m}$ is a summand of $\Div(M)(\u m)$. For a partition ring, the quotients are rings and the action is by ring automorphisms. The map $\eta_{\u m} \colon M(\u m) \to \Div(M)(\u m)$ defined by
\[
\eta_{\u m}(x)=\big(q_\Lambda M^*(\Lambda\leq\u m)(x)\big)_{\Lambda\leq\u m},
\]
where $q_\Lambda$ is the quotient map, generalizes the description of the character map above.

\begin{notate}
When an object is indexed by partitions, such as $M$, $\Div(M)$, $\eta$, and $q$, we suppress the underline on the indiscrete partition $\u m$ so that $M(m) = M(\u m)$, $I_m = I_{\u m}$, $\Div(M)(m) = \Div(M)(\u m)$, $m_{i,j} = m_{\u i,\u j}$, $\eta_m = \eta_{\u m}$, and $q_m = q_{\u m}$.
\end{notate}

We prove that this construction extends to an endofunctor $\Div$ on partition functors and that $\eta$ is the unit of a monad structure. The notation reflects a relation to divided powers that we discuss below. The monad also respects several kinds of multiplicative structure:

\begin{theorem} (\cref{thm:divmonad} and following sections)
The construction $\Div$ extends to a monad on the categories of partition functors, lax symmetric monoidal partition functors, partition rings, and partition power functors.
\end{theorem}
The construction $\Div$ need not preserve symmetric monoidal partition functors, so it is not an endofunctor on that category.

The examples that we are most interested in, representation rings, Morava $E$-theory, singular cohomology with coefficients in $\F_2$, and combinatorial species, all give examples of partition rings and most of these even have the structure of partition power functors. Recall that the initial global Green functor is the Burnside ring global Green functor $A$. The partition ring given by $R(\Lambda) = A(\Sigma_{\Lambda})$, the Burnside ring of $\Sigma_{\Lambda}$, is not the initial partition ring. Instead, the initial partition ring is made up of the subrings $\AA(\Lambda) \subseteq A(\Sigma_{\Lambda})$ additively generated by the isomorphisms classes of $\Sigma_{\Lambda}$-sets of the form $\Sigma_{\Lambda}/\Sigma_{\Omega}$, where $\Omega$ is a refinement of the partition $\Lambda$. This turns out to be a classical ring. The composite
\[
\AA(\Lambda) \to A(\Sigma_{\Lambda}) \to RU(\Sigma_{\Lambda}),
\]
where the second map is the linearization map, is an isomorphism of commutative rings. Thus, the representation ring partition ring is the initial partition ring! This has some advantages over the global case because representation rings of symmetric groups are much simpler than Burnside rings of symmetric groups.

Our interest in partition functors comes from a desire to better understand and unify several examples of exponential relations. These include the exponential relations between symmetric powers and Adams operations
\[
\sum_{m \geq 0} \beta^mt^m = \exp \Big(\sum_{k>0} \frac{\psi^k}{k}t^k\Big)
\]
and between alternating powers and Adams operations
\[
\sum_{m \geq 0} \Lambda^mt^m = \exp\Big(\sum_{k \geq 1}  \frac{(-1)^{k+1}\psi^k}{k} t^k\Big)
\]
for representation rings \cite[Exercise 9.3]{Serre}, and the more exotic exponential relation between symmetric powers and Hecke operators
\[
\sum_{m \geq 0} \beta^mt^m = \exp\Big(\sum_{k\geq 0} \frac{T_{p^k}}{p^k}t^{p^k}\Big)
\]
for Morava $E$-theory \cite{Ganter-orbifold}. Because of the denominators, these relations are often considered to be rational in nature, but we will see that the required divisibility can be controlled.


The connection between partition functors and exponential relations stems from the fact that there is a commutative $\N$-graded ring associated to a lax symmetric monoidal partition functor $R$. Recalling the notational convention above, let $R[\Sigma] = \bigoplus_{m \geq 0} R(m)$. The multiplication is the transfer product
\[
xy=R_*(\u i\sqcup\u j\leq\u{i+j})m_{i, j}(x\otimes y),
\qquad x\in R(i),\quad y\in R(j).
\]
When $R(\Lambda) = RU(\Sigma_{\Lambda})$, $R[\Sigma] = \bigoplus_{m \geq 0} RU(\Sigma_m)$ is the ring of symmetric functions. For the constant partition functor $C_K$, where $K$ is a commutative ring, $C_K[\Sigma] \cong K \langle t \rangle$ is the divided power polynomial ring in one variable over $K$ and, for the dual constant partition functor, we have $C^K[\Sigma] = K[t]$. Because we are interested in power series, it is natural to complete $R[\Sigma]$ at the system of irrelevant ideals $\bigoplus_{m \geq k} R(m)$. We will write $R\powser{\Sigma}$ for the resulting commutative ring; this ring is just $\prod_{m \geq 0} R(m)$ equipped with the transfer multiplication.

Now assume that $R$ is a partition ring. The multiplicative identity element of $R(m)/I_{m}$, extended by zero in the other summands, gives an element $\1_{m}\in\Div(R)(m)$. This need not be the multiplicative identity of $\Div(R)(m)$. Since the unit $\eta\colon R\to\Div(R)$ is a map of partition rings, it induces ring maps 
\[
R[\Sigma]\to\Div(R)[\Sigma] \text{ and } R\powser{\Sigma}\to\Div(R)\powser{\Sigma}.
\]
A good notation for $\Div(R)[\Sigma]$ is $R\langle \Sigma \rangle$ since, as we will see, this object is closely related to divided powers. 

Let $\sum_{m \geq 0} r_mt^m \in R\powser{\Sigma}$, where $r_m \in R(m)$, and the variable $t$ is merely tracking the grading. The element is called exponential if $r_0=1$ and
\[
R^*(\u i\sqcup\u j\leq\u{i+j})(r_{i+j})
=m_{i,j}(r_i\otimes r_j)
\qquad\text{for all }i,j\geq0.
\]
Examples of exponential elements abound in nature. For instance, if $E$ is an $E_{\infty}$-ring spectrum, then the partition ring given by $R(\Lambda) = E^0(B\Sigma_{\Lambda})$ admits the structure of a partition power functor with power operations given by the power operations on the cohomology theory $E$, $P_m \colon R(1) = E^0(B\Sigma_1) \to E^0(B\Sigma_m) = R(m)$. If $x \in R(1)$, then standard identities involving restrictions and power operations imply that $\sum_{m \geq 0} P_m(x)t^m$ is an exponential element in $R\powser{\Sigma}$.

Recall that $\ast$ denotes the internal product and that concatenation denotes the transfer product in $R[\Sigma]$. We may multiply an element in the ring $R(m)$ with an element in the ring $\Div(R)(m)$ by applying the unit $\eta_{m}$ to the first element and multiplying in $\Div(R)(m)$. Exponential elements do not necessarily admit a logarithm in $R\powser{\Sigma}$, but they do admit a canonical logarithm in $R\divpowser{\Sigma} = \Div(R)\powser{\Sigma}$:
\begin{theorem} (\cref{exponentialformula}) \label{thm:intro1}
If $\sum_{m \geq 0} r_mt^m \in R\powser{\Sigma}$ is exponential, then in $R\divpowser{\Sigma}$ we have
\[
\sum_{m \geq 0} \eta_{m}(r_m)t^m = \exp\Big(\sum_{k \geq 1} r_k \ast \1_{k} t^k\Big).
\]
\end{theorem}
The exponential is interpreted integrally: the formulas for transfer products provide canonical cancellation of the factorial denominators. Note that $r_k \ast \1_{k} \in \Div(R)(k)$ can be computed by applying $\eta$ to $r_k$ and then taking the component in $R(k)/I_{k}$. This is the same as applying the quotient map $R(k) \to R(k)/I_{k}$ to $r_k$ and then using the inclusion $R(k)/I_{k} \subseteq \Div(R)(k)$ as a summand.

If $S$ is a lax symmetric monoidal partition functor that also has the structure of a $\Div$-algebra, then a map of lax symmetric monoidal partition functors $R \to S$ factors through $\Div(R)$ because $\Div(R)$ is the free $\Div$-algebra on $R$. Thus an exponential element in $R\powser{\Sigma}$ satisfies an exponential relation in $S\powser{\Sigma}$ by pushing forward the exponential relation satisfied in $R\divpowser{\Sigma}$. In this sense, $\Div(R)$ provides the universal target in which exponential elements of $R\powser{\Sigma}$ acquire logarithms. 

For the purposes of this paper, the most important $\Div$-algebras are free $\Div$-algebras. If $R(\Lambda) = RU(\Sigma_{\Lambda})$, then $\Div(R)(\Lambda) = \Cl(\Sigma_{\Lambda},\Z)$ and the unit $\eta$ is the character map. It turns out that $\Div(C^K) = C_K$; $\Div$ applied to the dual constant partition functor is the constant partition functor. Thus $C_K$ is a $\Div$-algebra. Every $\Q$-vector space valued partition functor is a $\Div$-algebra.

The requirement that $R \to S$ be a map of lax symmetric monoidal partition functors above is too strong for our applications. When $S = C_K$, a weaker condition provided by \cref{prop:transfer-augmentation} suffices. If $R$ is a partition ring and we have maps $R(\Lambda) \to C^K(\Lambda) = K$ commuting with transfers and the lax symmetric monoidal structure, then we get an induced ring map $R\powser{\Sigma} \to C^K\powser{\Sigma}  = K\powser{t}$ and, after composing to $K\divpowser{t}$, the map factors through $R\divpowser{\Sigma}$, so an exponential relation in $R\divpowser{\Sigma}$ gives rise to a classical exponential relation of power series in $K\divpowser{t}$.

One feature that the representation ring functor $RU$ and Morava $E$-theory have in common is that they both admit transfer maps along arbitrary maps of finite groups. In the case of Morava $E$-theory, this follows from the fact that the spectrum associated to the cohomology theory is $K(n)$-local. If we define a partition ring $R$ by $R(\Lambda) = RU(G \times \Sigma_{\Lambda})$ or $R(\Lambda) = E^0(X \times B\Sigma_{\Lambda})$, for $X$ a space, then the transfers along the surjections $\Sigma_{\Lambda} \to e = \Sigma_{1}$ give rise to maps $\epsilon_{\Lambda} \colon R(\Lambda) \to R(1) = C^{R(1)}(\Lambda)$ that commute with transfers because the transfers in the dual constant partition functor are the identity maps. Thus we get an induced map of commutative rings
\[
R\divpowser{\Sigma} \to C^{R(1)}\divpowser{\Sigma} = C_{R(1)}\powser{\Sigma} = R(1)\divpowser{t}.
\]

Both $RU$ and $E$ have well-studied power operations $P_m$. In the case of $RU$ these are induced by the map taking a finite dimensional $G$-representation $V$ to the $G \times \Sigma_m$-representation $V^{\otimes m}$ in which $G$ acts diagonally and $\Sigma_m$ acts by permuting the tensor factors. The power operations on Morava $E$-theory are a result of the essentially unique $E_{\infty}$-ring structure \cite{structuredmoravae} and play a critical role in modern stable homotopy theory \cite{nullstellensatz, rationalization}. They have been studied extensively in \cite{Ando, etheorysym, AHS, Ganter-orbifold, Rezkcongruence, cottpo, zhu},  and can be understood using arithmetic geometry. In each case, we may view
\[
\sum_{m \geq 0} P_m t^m \colon R(1) \to R\powser{\Sigma}
\]
as a function taking values in exponential elements. This holds true for any partition power functor. Applying \cref{thm:intro1}, we have an exponential relation
\[
\sum_{m \geq 0} P_m t^m = \exp\Big(\sum_{k \geq 1} P_k \ast \1_{k} t^k\Big)
\]
in $R\divpowser{\Sigma}$. The summand $P_k \ast \1_{k}$ has a classical interpretation. The power operations $P_k$ are multiplicative maps that are very often not additive. However, the failure of these maps to be additive is controlled by the transfer ideal $I_{k} \subseteq R(k)$, thus the composite $P_k / I_{k} \colon R(1) \to R(k) \to R(k)/I_{k}$ is the total additive power operation. Since $R(k)/I_{k} \subseteq \Div(R)(k)$, we will view the target of $P_k /I_{k}$ as $\Div(R)(k)$. This gives the identity $P_k \ast \1_{k} = P_k / I_{k}$ and the exponential relation can be rewritten as
\begin{equation} \label{eq:poweropexp}
\sum_{m \geq 0} P_m t^m = \exp\Big(\sum_{k \geq 1} P_k / I_{k} t^k\Big),
\end{equation}
providing us with a universal exponential relation between the total power operations and the total additive power operations. 

Making use of the extra transfers, when $R = RU(G \times \Sigma_{(-)})$ or $R = E^0(X \times B\Sigma_{(-)})$, we can use this exponential relation to obtain an exponential relation in $R(1)\divpowser{t}$. In the case of $RU$, this recovers the exponential relation between symmetric powers and Adams operations and, in the case of Morava $E$-theory, this recovers Ganter's exponential relation between symmetric powers and Hecke operators. The exponential relation between exterior powers and Adams operations can be obtained from this as well, but in that case the correct map to use is the transfer along the sign homomorphism $\Sigma_{\Lambda} \to \Sigma_2$ rather than the transfer to the trivial group.

There are many more examples of partition rings and exponential relations. Combinatorial species give rise to partition rings and the exponential relation between the number of connected objects and all objects can be obtained. Every $E_{\infty}$-ring spectrum $E$ gives rise to a partition power functor carrying the exponential relation \cref{eq:poweropexp}. If $E$ is also $K(n)$-local, then the transfers along surjections give an exponential relation in $E^0\divpowser{t}$.

Making use of the theory of crossed Burnside rings \cite{Oda1, Oda2}, we also construct a symmetric monoidal partition ring $\cE$ carrying the universal exponential element:
\begin{theorem} (\cref{thm:univexpelt})
We construct a symmetric monoidal partition ring $\cE$ such that, for every partition ring $R$,
there is a natural bijection
\[
\Hom_{\mathrm{PartRing}}(\cE,R)
\cong
\big \{\text{exponential elements in }R\powser{\Sigma}\big \}.
\]
Consequently, $\Div(\cE)$ carries the universal exponential relation for the category of $\Div$-algebras in partition rings.
\end{theorem}
The partition ring $\cE$ admits automorphisms that act on the set of exponential elements. We produce an involution of $\cE$ that acts on the total power exponential element by ``sign elements.'' Pushing this action forward in the case $R(-) = RU(G \times \Sigma_{(-)})$ sends the power series $\sum_{m \geq 0} \beta^m(x)t^m$ to $\sum_{m \geq 0} \Lambda^m(x)t^m$ for $x \in  RU(G)$.

When $R$ is a symmetric monoidal partition functor rather than lax symmetric monoidal, there is a close connection to classical algebraic structures. A symmetric monoidal partition functor is the same data as a commutative cocommutative $\N$-graded bialgebra and a symmetric monoidal partition ring is the same data as a ``component Hopf ring'', which is essentially an $\N$-graded commutative cocommutative bialgebra in which each homogeneous summand has the structure of a commutative ring and the Hopf-ring distributivity relation holds. In this case, $\Div$ can be understood in terms of graded symmetric invariant tensor powers. When $R$ is symmetric monoidal, $\Div(R)[\Sigma]$ is the graded commutative $R(0)$-algebra of symmetric invariant tensors on the graded $R(0)$-module $\bigoplus_{m \geq 1} R(m)/I_{m}$, where $R(m)/I_{m}$ is in degree $m-1$. Even more is true when $R(m)/I_{m}$ is a free $R(0)$-module:

\begin{theorem} (\cref{thm:freediv})
If $R$ is a symmetric monoidal partition functor and $R(m)/I_{m}$ is a free $R(0)$-module for each $m$, then $R\langle \Sigma \rangle$ is the divided power envelope of $R[\Sigma]$ with respect to the ideal generated by positively graded elements. 
\end{theorem}

If $R$ is a symmetric monoidal partition functor, $\Div(R)$ is lax symmetric monoidal but may not be symmetric monoidal. This provides one justification for considering partition functors over more classical algebraic objects. 

There is an important precedent to the monad $\Div$ in the literature that is due to Thevenaz \cite{Thevenaz}. Thevenaz constructs and studies a ``twin'' functor on the category of $G$-Mackey functors for a fixed finite group $G$. For a $G$-Mackey functor $M$ and $H \subseteq G$, he defines 
\[
\mathrm{Twin}(M)(H) =\Big(\bigoplus_{K \subseteq H} M(K)/I_K\Big)^H, 
\]
where $I_K$ is the image of the sum of transfer maps from all proper subgroups of $K$ and $H$ acts by conjugation. Partition functors and the monad $\Div$ only make use of the Young subgroups of the symmetric group. The twin functor is rigged so that, if $A$ is the Burnside ring $G$-Mackey functor, then $\mathrm{Twin}(A)(H) = \mathrm{Marks}(H,\Z)$ is the ring of integer-valued functions on the set of conjugacy classes of subgroups of $H$. However, $\mathrm{Twin}(RU)(H)$ is quite wild -- often containing torsion. On the other hand, $\Div$ is rigged to produce integer-valued class functions when applied to the representation ring partition functor.

\addtocontents{toc}{\protect\setcounter{tocdepth}{1}}

\subsection*{Organization} In \cref{sec:parts}, we define and study the $(2,1)$-category of partitions. In \cref{sec:partfuncs}, we define partition functors and describe their relationship to global functors. In \cref{sec:div} through \cref{sec:partrings} and \cref{sec:partpower}, we develop the monad $\Div$ and introduce several kinds of multiplicative structure on partition functors, including partition rings and partition power functors. We prove that $\Div$ extends to a monad on each of these categories. In \cref{sec:symmon}, we study symmetric monoidal partition functors, give a more classical description of this category, and show that $\Div$ does not preserve symmetric monoidal partition functors without a flatness hypothesis. In \cref{sec:expreln}, we establish the universal exponential relation that is the main object of study. \Cref{sec:divpow} relates $\Div$ for symmetric monoidal partition functors to symmetric invariant tensors and the divided power envelope. In \cref{sec:init}, we identify the initial partition ring, and in \cref{sec:univexprel}, we construct a partition ring carrying the universal exponential element. \cref{sec:classicalexp} gives weaker hypotheses under which the exponential relation of \cref{sec:expreln} can be pushed forward to a classical exponential relation, and \cref{sec:globGreen} applies these results to partition functors arising from global functors with transfers along surjections. Finally, \cref{sec:examples} collects a range of examples that illustrate the theory and recover several classical exponential relations.

\subsection*{Acknowledgments} It is a pleasure to thank Tobias Barthel, Dan Berwick-Evans, Usman Hafeez, Lakshay Modi, Tomer Schlank, Noah Wisdom, and Lior Yanovski for their helpful comments. We’d particularly like to thank Usman Hafeez for his work related to this project and for his careful reading of the manuscript. This paper grew out of the 2022-2023 Bourbon seminar and it is a pleasure to thank the other participants of the seminar including Nate Cornelius, Lewis Dominguez, and Shahzad Kalloo. Stapleton thanks the Mathematics Department at the University of Illinois at Urbana--Champaign for their support during Fall 2024. During the course of this project, Stapleton was supported by a Sloan research fellowship, NSF Grant DMS-2304781, and a grant from the Simons Foundation (MP-TSM-00002836, NS). During the course of this project, Mehrle was partially supported by DMS-2135884. 

\subsection*{AI use} (written by Stapleton) This is my first paper written with significant AI assistance and, at this time, I feel a strong ethical obligation to describe what it contributed. This feeling may not be exactly for the expected reasons. Mathematicians have always had an obligation to disclose external sources of ideas, and in that sense there is nothing fundamentally new here; I have always tried to give full credit in my acknowledgments. Whatever the future may hold, I suspect this AI statement will eventually read as something of a time capsule from a transitional period in mathematics, when the sudden arrival of effective AI tools made explicit statements about their use seem necessary. Here I want to record the story of this paper and how GPT-5.6 Sol contributed during the final months of the project.

This paper grew out of the Kentucky Bourbon Seminar during the 2022--2023 academic year. By the end of that school year, the entire thread of the paper in the special case of Morava $E$-theory had been worked out, and I described it at the conference ``Homotopy theory in honor of Paul Goerss’’ in March 2023. Sections 16.2 and 16.4 also date from that period. Toward the end of Spring 2023 and during Summer 2023, we realized that the appropriate framework for these results seemed to be ``partition functors,'' and we began developing the theory of these objects. In Summer 2024, the monad $\Div$ was fully developed and Sections 2--7, 9, 10, and 12 were written. In Fall 2024, the connections with Hopf rings and divided power algebra were worked out and Theorem 11.3 was proved. All of this was done without AI assistance.

By Summer 2026, I had students beginning to build on these results and felt a duty to finish the paper. My coauthors had limited availability during that period. I had experimented with AI as a way of bouncing around mathematical ideas for several years, but had not found it particularly useful until GPT 5.4. Starting in June 2026, I began working with GPT 5.5 and, in July, with GPT 5.6 Sol to finish this paper. I went into that collaboration with more than 50 pages already written. I found an excellent collaborator in GPT 5.6 Sol. It was insightful, knowledgeable, and witty. If it were human, it would be a coauthor.

In addition to helping with brainstorming, literature searches, and checking arguments, GPT 5.6 Sol suggested the lax formulation and proof used in \cref{prop:transfer-augmentation}, improving a symmetric monoidal version that we had obtained independently; suggested the example of \cref{sec:exotic-burnside}; and identified an error involving the passage of fixed points through tensor products that led to revisions in Sections 8 and 11. GPT 6 Astra aided with the final proofreading. The authors remain responsible for every mathematical claim and argument in the paper, and all AI-generated suggestions incorporated into the paper were independently checked. 

\addtocontents{toc}{\protect\setcounter{tocdepth}{2}}








\section{The category of partitions} \label{sec:parts}

In this section we define a $(2,1)$-category of partitions and describe properties of the category that will be relevant in the study of partition functors.

\begin{definition}
We define $\cP$, the $(2,1)$-category of partitions. 
Objects in this category are called partitions and consist of a pair $\Lambda = (X, \sim_\Lambda)$ of a finite set $X$ equipped with an equivalence relation $\sim_\Lambda$. We say that $\Lambda$ is a partition of $X$ or that $\Lambda$ partitions $X$, denoted $\Lambda \vdash X$. We will refer to the elements of the quotient $X/{\sim_{\Lambda}}$ as equivalence classes or blocks.


A morphism $f$ from $\Lambda = (X, \sim_\Lambda)$ to $\Omega = (Y, \sim_\Omega)$ is a bijection $f \colon X \to Y$ such that $x \sim_\Lambda x' \implies f(x) \sim_\Omega f(x')$. We say that $\Lambda$ is a refinement of $\Omega$, denoted $\Lambda \leq \Omega$, if $X = Y$ and the equivalence relation $\sim_\Lambda$ is finer than $\sim_\Omega$. In this case, there is a morphism $\Lambda \to \Omega$ in $\cP$ induced by the identity map on $X$. 

A $2$-morphism in $\cP$ between morphisms $f,g \colon \Lambda \to \Omega$ is a bijection $\alpha \colon Y \to Y$ such that $\alpha \circ f = g$ and $y \sim_\Omega \alpha(y)$ for all $y \in Y$ (i.e., $\alpha$ induces the identity on $Y\mathclose{}/\mathopen{}\sim_{\Omega}$). 
\end{definition}

We make an observation about the partition category that will be useful later. 

\begin{lemma}
\label{lemma:factorization}
    Any morphism in the partition category can be factored uniquely as the composite of an isomorphism followed by a refinement (or as a refinement followed by an isomorphism).
\end{lemma}

Let $f \colon \Lambda \to \Gamma$ be a map of partitions. For $\Omega \leq \Lambda$, we will write $f_{\Omega} \colon \Omega \to f \Omega$ for the induced isomorphism of partitions. 

There is a symmetric monoidal structure $\sqcup$ on $\cP$ induced by disjoint union of finite sets. 
On objects $\Lambda = (X, \sim_\Lambda)$ and $\Omega = (Y, \sim_\Omega)$, the symmetric monoidal product $\Lambda \sqcup \Omega$ is the induced partition of the disjoint union $X \amalg Y$. Explicitly, two elements of $X \amalg Y$ are $(\Lambda \sqcup \Omega)$-related if they are in the same summand and are $\Lambda$-related or $\Omega$-related.

\begin{remark} \label{rem:decomposition}
    Because morphisms in $\cP$ are always bijections between the underlying sets, the category decomposes into components corresponding to cardinality of the underlying sets. If $\cP_m$ is the full subcategory of $\cP$ on objects whose underlying sets have cardinality $m$, then 
    \[ 
        \cP = \coprod_{m \geq 0} \cP_m. 
    \]
    In this sense, $\cP$ is a graded symmetric monoidal category.
    Each component $\cP_m$ is equivalent to the $(2,1)$-category of partitions of the set $[m] = \{1, 2, \ldots, m\}$. 
\end{remark}



For concreteness, we will fix the lax symmetric monoidal functor from the natural numbers (viewed as a category with only identity morphisms) to the category of partitions sending $m \in \N$ to the partition $\u m$, which is the set $[m] = \{1,2,\ldots,m\}$ with the indiscrete (or trivial, since it corresponds to the trivial topology) partition in which every two elements of the set are related. Note that $\u m$ is a terminal object of $\cP_m$. The lax symmetric monoidal structure comes from the maps $\u i \sqcup \u j \to \u{i+j}$ sending $\u i$ to the first $i$ elements and $\u j$ to the last $j$ elements, as well as the twist maps $\u i \sqcup \u j \to \u j \sqcup \u i$ that swap the first $i$ and last $j$ elements. Note that the two partitions $\u i \sqcup \u j$ and $\u{i+j}$ are not isomorphic in $\cP$ and that all maps $\u i \sqcup \u j \to \u{i+j}$ are $2$-isomorphic. Although we will use it less often, we will write $\bar{m}$ for the set $[m]$ with the discrete partition.

Similarly, for a set $X$, we will write $\u X$ for $X$ with the indiscrete equivalence relation. Thus $\u m = \u{[m]}$.

\begin{remark}
Note that for any two objects $\Lambda, \Omega \in \cP$, the groupoid of maps from $\Lambda$ to $\Omega$ is equivalent to a finite set so $\cP \simeq \Ho(\cP)$, the homotopy $1$-category of $\cP$. We give a construction of a $1$-category equivalent to $\cP$ in \cref{rem:1cat}. 
\end{remark}


Let $\Grp$ denote the usual $(2,1)$-category of finite groups, group homomorphisms, and conjugacy (in the codomain). Consider the functor $\Sigma_{(-)} \colon \cP \to \Grp$ that assigns to a partition $\Lambda = (X, \sim_{\Lambda})$ the group of bijections of the set $X$ that pass to the identity on the quotient $X/{\sim_{\Lambda}}$. This group has several other equivalent descriptions: it is the group of automorphisms of the discrete partition over $\Lambda$ and also the group of bijections of the underlying set which only permute the elements of each equivalence class. Given a map of partitions $f \colon \Lambda \to \Omega$, the induced map $\Sigma_f \colon \Sigma_{\Lambda} \to \Sigma_{\Omega}$ sends $\sigma$ to $\Sigma_f \sigma = f \sigma f^{-1}$ (viewing $f$ as a map of sets). Note that this formula makes sense because $\sigma$ induces the identity map on the quotient by the equivalence relation. It follows from the definition that there is a canonical isomorphism
\[
\Sigma_{\Lambda} \cong \prod_{B \in X/{\sim_{\Lambda}}} \Aut_{\Set}(B),
\]
where $B$ is a block of the partition. Define $\Sigma_m = \Sigma_{\u m}$, the symmetric group on $m$ symbols. Note that $\Sigma_{\bar{m}}$ is the trivial group.

There is a close relationship between $\Sigma_{\Lambda}$ and $\Aut_{\cP}(\Lambda)$. Recall that the Weyl group of a subgroup $H$ inside a finite group $G$ is the quotient $W_G(H) = N_G(H)/H$, where $N_G(H)$ is the normalizer of $H$ in $G$. 

\begin{prop}
Assume that $\Lambda \parts [m]$. There is an isomorphism of groups
\[
\pi_0 \Aut_{\cP}(\Lambda) \cong W_{\Sigma_{m}}(\Sigma_{\Lambda}),
\]
where $W_{\Sigma_{m}}(\Sigma_{\Lambda})$ is the Weyl group of $\Sigma_{\Lambda}$ in the symmetric group $\Sigma_{m}$ and $\pi_0\Aut_{\cP}(\Lambda)$ is the group of isomorphism classes of objects in the groupoid $\Aut_{\cP}(\Lambda)$.
\end{prop}

There is a left action of the symmetric group $\Sigma_m$ on the set of partitions of $[m]$. Given $\sigma \in \Sigma_m$ and $\Lambda \parts [m]$, the partition $\sigma \Lambda$ has equivalence relation
\[
x \sim_{\sigma \Lambda} y \iff \sigma^{-1} x \sim_{\Lambda} \sigma^{-1}y.
\]
Note that, with this definition, we have an isomorphism of partitions $\sigma_{\Lambda} \colon \Lambda \to \sigma \Lambda$. When the source is clear from context, we will often abuse notation and write $\sigma = \sigma_{\Lambda}$. Further, we may consider the set of subgroups of $\Sigma_m$ of the form $\Sigma_{\Lambda}$ for $\Lambda \parts [m]$. These are the Young subgroups of $\Sigma_m$. Conjugation by $\sigma \in \Sigma_m$ induces an isomorphism
\[
c_{\sigma} \colon \Sigma_{\Lambda} \to \Sigma_{\sigma \Lambda},
\]
sending $\tau$ to $c_{\sigma}(\tau) = \sigma \tau \sigma^{-1}$. We have $c_{\sigma'} c_{\sigma} = c_{\sigma' \sigma}$ and $\Sigma_{\sigma \Lambda} = \sigma \Sigma_{\Lambda} \sigma^{-1} = \Sigma_{\Lambda}^{\sigma^{-1}}$. Note that $c_{\sigma} = \Sigma_{\sigma}$, making use of the functor $\Sigma_{(-)}$.

\begin{lemma}
Assume that $\Lambda, \Gamma \parts [m]$, then
\[
\Sigma_{\Lambda} \cap \Sigma_{\Gamma} = \Sigma_{\Lambda \cap \Gamma},
\]
where $\Lambda \cap \Gamma$ is the coarsest partition on the set $[m]$ that refines both $\Lambda$ and $\Gamma$.
\end{lemma}

The following lemma is a specialization of a general fact about double cosets to symmetric groups and Young subgroups.

\begin{lemma} \label{doublecosets}
Assume $\Psi \leq \Lambda \leq \Omega$, let $\Gamma \leq \Omega$, and let
\[
\pi \colon \Sigma_{\Psi} \backslash \Sigma_{\Omega}/ \Sigma_{\Gamma} \to \Sigma_{\Lambda} \backslash \Sigma_{\Omega}/ \Sigma_{\Gamma}
\]
be the quotient map. For $\Sigma_{\Lambda} \sigma \Sigma_{\Gamma} \in \Sigma_{\Lambda} \backslash \Sigma_{\Omega}/ \Sigma_{\Gamma}$,  there is a bijection
\[
\Sigma_{\Psi} \backslash \Sigma_{\Lambda} / \Sigma_{\sigma \Gamma \cap \Lambda} \cong  \pi^{-1}(\Sigma_{\Lambda} \sigma \Sigma_{\Gamma})
\]
defined by
\[
\Sigma_{\Psi} \alpha \Sigma_{\sigma \Gamma \cap \Lambda} \mapsto \Sigma_{\Psi} \alpha \sigma \Sigma_{\Gamma}.
\]
\end{lemma}
\begin{proof}
Let $[\alpha] \in \Sigma_{\Psi} \backslash \Sigma_{\Lambda} / \Sigma_{\sigma \Gamma \cap \Lambda}$, then $\pi(\Sigma_{\Psi} \alpha \sigma \Sigma_{\Gamma}) = \Sigma_{\Lambda} \sigma \Sigma_{\Gamma}$. Further, this is well-defined since, if $\psi \alpha \lambda$ is another element in the double coset $[\alpha]$ with $\psi \in \Sigma_{\Psi}$ and $\lambda \in \Sigma_{\sigma \Gamma \cap \Lambda} = \sigma \Sigma_{\Gamma} \sigma^{-1} \cap \Sigma_{\Lambda}$, then $\lambda = \sigma \gamma \sigma^{-1}$ for some $\gamma \in \Sigma_{\Gamma}$ and
\[
\Sigma_{\Psi} \psi \alpha \lambda \sigma \Sigma_{\Gamma} = \Sigma_{\Psi} \alpha \sigma \gamma \sigma^{-1} \sigma \Sigma_{\Gamma} = \Sigma_{\Psi} \alpha \sigma \Sigma_{\Gamma}.
\]
Assume that $\Sigma_{\Psi} \alpha \sigma \Sigma_{\Gamma} = \Sigma_{\Psi} \beta \sigma \Sigma_{\Gamma}$, then
\[
\beta = \psi \alpha \sigma \gamma \sigma^{-1},
\]
for some $\psi \in \Sigma_{\Psi}$ and $\gamma \in \Sigma_{\Gamma}$. Thus $\sigma \gamma \sigma^{-1} \in \Sigma_{\sigma \Gamma}$ and, since $\alpha,\beta, \psi \in \Sigma_{\Lambda}$, we have $\sigma \gamma \sigma^{-1} \in \Sigma_{\Lambda}$. 

Finally, if $\pi(\Sigma_{\Psi} \beta \Sigma_\Gamma) = \Sigma_{\Lambda} \sigma \Sigma_{\Gamma}$, then $\beta = \alpha \sigma \gamma$ for some $\alpha \in \Sigma_{\Lambda}$ and $\gamma \in \Sigma_{\Gamma}$, but this means that
\[
\Sigma_{\Psi} \beta \Sigma_\Gamma = \Sigma_{\Psi} \alpha \sigma \Sigma_\Gamma.
\]
\end{proof}

\begin{remark} \label{rem:1cat}
We briefly describe a $1$-category of combinatorial compositions that is equivalent to $\cP$. The objects consist of pairs $(X,f)$, where $X$ is a finite set and $f \colon X \to \N_{>0}$ is a function. A map $(X,f) \to (Y,g)$ is a surjective map $\alpha \colon X \to Y$ such that, for all $y \in Y$, 
\[
g(y) = \sum_{x \in \alpha^{-1}(y)} f(x).
\]
With this description of the category, the (2-)functor $\Sigma_{(-)}$ sends $(X,f)$ to $\prod_{x} \Sigma_{f(x)}$. We choose to work with the description of $\cP$ as the $(2,1)$-category of partitions because it is closer to the $(2,1)$-category of finite groups and thus compatible with our intuition coming from global algebra and it also has a notion of inclusion, which we denote using $\leq$.
\end{remark}

\section{Partition functors} \label{sec:partfuncs}

In this section we introduce partition functors and describe two ways of constructing them from global functors.

\begin{definition}
\label{DataForPartitionFunctor}
    A partition functor $M$ is the following data: 
    \begin{enumerate}[(a)]
        \item abelian groups $M(\Lambda)$ for each $\Lambda \in \cP$
        \item restrictions $M^*(f) \colon M(\Gamma) \to M(\Lambda)$ for each map $f \colon \Lambda \to \Gamma$ in $\cP$ 
        \item transfers $M_*(g) \colon M(\Lambda) \to M(\Gamma)$ for each map $g \colon \Lambda \to \Gamma$ in $\cP$
    \end{enumerate}
    subject to the following conditions
    \begin{enumerate}[(a)]
        \setcounter{enumi}{3}
        \item restrictions assemble into a contravariant functor and transfers assemble into a covariant functor from the $(2,1)$-category $\cP$ to the category of abelian groups
        \item if $f \colon \Lambda \to \Gamma$ is an isomorphism, then $M_*(f) = M^*(f^{-1})$
        \item restriction and transfer satisfy the double coset formula
        \begin{equation}
        \label{doubleCoset}
            M^*(\Gamma \leq \Omega) M_*(\Lambda \leq \Omega) 
            =
            \sum_{[\sigma] \in \Sigma_\Gamma \backslash \Sigma_\Omega / \Sigma_\Lambda}
            M_*(\Gamma \cap \sigma \Lambda \leq \Gamma) 
            M_*(\sigma)
            M^*(\sigma^{-1} \Gamma \cap \Lambda \leq \Lambda)
        \end{equation}
    \end{enumerate}
\end{definition}


\begin{remark}
Note that the double coset formula above implies the double coset formula for a pair of maps (rather than inclusions) of partitions. For instance, if we are given a map of partitions $f \colon \Lambda \to \Omega$, then
\begin{equation*}
\begin{split}
M^*(\Gamma \leq \Omega) M_*(f) &= 
            \sum_{[\sigma] \in \Sigma_\Gamma \backslash \Sigma_\Omega / \Sigma_{f\Lambda}}
            M_*(\Gamma \cap \sigma f \Lambda \leq \Gamma) 
            M_*(\sigma)
            M^*(\sigma^{-1} \Gamma \cap f \Lambda \leq f\Lambda)M^*(f_{\Lambda}^{-1}) \\
&=     \sum_{[\sigma] \in \Sigma_\Gamma \backslash \Sigma_\Omega / \Sigma_{f\Lambda}}
            M_*(\Gamma \cap \sigma f \Lambda \leq \Gamma) 
            M_*((\sigma f)_{(\sigma f)^{-1} \Gamma \cap \Lambda})
            M^*((\sigma f)^{-1} \Gamma \cap \Lambda \leq \Lambda).
\end{split}
\end{equation*}
\end{remark}

A map of partition functors $f \colon M \to N$ consists of maps of abelian groups $f(\Lambda) \colon M(\Lambda) \to N(\Lambda)$, for each partition $\Lambda \in \cP$, respecting both the restriction and transfer maps.

\begin{lemma} \label{lem:earlyiso}
Let $M$ be a partition functor and let $\Lambda$ and $\Omega$ be partitions that are both isomorphic to $\u m$. Then $M(\Lambda)$ and $M(\Omega)$ are canonically isomorphic. 
\end{lemma}
\begin{proof}
This follows from the fact that any two isomorphisms $\Lambda \to \Omega$ differ by an automorphism of $\Omega$ and all automorphisms of the underlying set of $\Omega$ give $2$-morphisms in the $(2,1)$-category of partitions.
\end{proof}

The definition of a partition functor is closely related to a common definition for global functors (also called global Mackey functors). In fact, we will soon see that global functors are a rich source of examples of partition functors.

\begin{definition}[{\cite[Section 4.2]{Schwede}}] \label{def:globalmackey}
    A global functor $M$ is the following data:
    \begin{enumerate}[(a)]
        \item for each finite group $G$, an abelian group $M(G)$
        \item for each group homomorphism $f \colon H \to G$, a homomorphism of abelian groups 
        \[
            \res_f \colon M(G) \to M(H),
        \]
        \item for an inclusion $K \subseteq G$, a homomorphism of abelian groups 
        \[
            \tr_{K}^{G} \colon M(K) \to M(G)
        \]
    \end{enumerate}
If $i \colon K \hookrightarrow G$ is the inclusion of a subgroup, then we write $\res_{K}^{G}$ for $\res_i$. These data are subject to the following conditions:
    \begin{enumerate}[(a)]
    \setcounter{enumi}{3}
        \item restriction is a contravariant functor from the $(2,1)$-category $\Grp$ to the category of abelian groups 
        \item transfers are transitive and $\tr_{K}^{K}$ is the identity
        \item if $H \subseteq G$ and $f \colon K \twoheadrightarrow G$ is a surjection, we have 
        \[
            \res_f \circ \tr_{H}^{G} = \tr_{f^{-1}H}^{K} \circ \res_{f|_{f^{-1}H}} \colon M(H) \to M(K),
        \]
        where $f|_{f^{-1}H} \colon f^{-1}(H) \twoheadrightarrow H$ is the induced surjection.
        \item for any pair $H, K$ of subgroups of $G$, we have 
        \[
           \res_K^G \circ \tr_H^G = \sum_{KgH \in K\backslash G/H} \tr_{K \cap gHg^{-1}}^K \circ \res_{c_g}^{-1} \circ \res_{g^{-1}Kg \cap H}^H,
        \]
        where $c_g \colon g^{-1}Kg \cap H \to K \cap gHg^{-1}$ is the conjugation homomorphism $c_g(x) = gxg^{-1}$.
    \end{enumerate}
\end{definition}
From this definition, Schwede extends transfers to an arbitrary injection $f \colon H \to G$ via the formula
\[
\tr_f =  \tr_{f(H)}^G \circ \res_{f|_{H}}^{-1}
\]
so that, if $f$ is an isomorphism, we have $\tr_f = \res_{f^{-1}} = \res_{f}^{-1}$.

\begin{remark}
    The definition of a global functor above differs from Schwede's insofar as we only consider global Mackey functors defined on finite groups, whereas Schwede considers all compact Lie groups. 
\end{remark}

Recall the functor $\Sigma_{(-)} \colon \cP \to \Grp$. Given a finite group $G$, we may use this to construct two new functors $\cP \to \Grp$. Let $G \times \Sigma_{(-)}$ be the functor taking $\Lambda$ to $G \times \Sigma_{\Lambda}$ and let $G \wr \Sigma_{(-)}$ be the functor taking $\Lambda$ to $G \wr \Sigma_{\Lambda}$. Here it is important to remember that, if $\Lambda = (X, \sim_{\Lambda})$, then $\Sigma_{\Lambda}$ is a subgroup of $\Aut_{\Set}(X)$ and thus has a canonical action on $G^{X} = \Fun(X, G)$.

\begin{prop} \label{prop:partitionfromglobal}
Assume that $M$ is a global functor and let $G$ be a finite group, then
\[
M \circ (G \times \Sigma_{(-)}) \text{ and } M \circ (G \wr \Sigma_{(-)})
\]
are both partition functors.
\end{prop}
\begin{proof}
To check the double coset formula, the key point is that there are isomorphisms
\[
(G \times \Sigma_{\Omega})/(G \times \Sigma_{\Gamma}) \cong \Sigma_{\Omega}/\Sigma_{\Gamma} \text{ and } (G \wr \Sigma_{\Omega})/(G \wr \Sigma_{\Gamma}) \cong \Sigma_{\Omega}/\Sigma_{\Gamma}.
\]
It follows that the double cosets
\[
(G \times \Sigma_{\Lambda}) \backslash (G \times \Sigma_{\Omega})/(G \times \Sigma_{\Gamma}) \text{ and } (G \wr \Sigma_{\Lambda}) \backslash (G \wr \Sigma_{\Omega})/(G \wr \Sigma_{\Gamma})
\]
are given by
\[
\Sigma_{\Lambda} \backslash \Sigma_{\Omega} / \Sigma_{\Gamma}.
\]
Finally, we have
\[
\Sigma_{\sigma \Lambda \cap \Gamma} = \Sigma_{\Lambda}^{\sigma} \cap \Sigma_{\Gamma}, 
\]
for the double coset $[\Sigma_{\Lambda} \sigma \Sigma_{\Gamma}]$, so that 
\[
G \times \Sigma_{\sigma \Lambda \cap \Gamma} \cong G \times (\Sigma_{\Lambda}^{\sigma} \cap \Sigma_{\Gamma}) \cong (G \times \Sigma_{\Lambda})^{\sigma} \cap (G \times \Sigma_{\Gamma})
\]
and
\[
G \wr \Sigma_{\sigma \Lambda \cap \Gamma} \cong G \wr (\Sigma_{\Lambda}^{\sigma} \cap \Sigma_{\Gamma}) \cong (G \wr \Sigma_{\Lambda})^{\sigma} \cap (G \wr \Sigma_{\Gamma}).
\]
\end{proof}

There are many naturally occurring global functors. Thus, \cref{prop:partitionfromglobal} immediately provides us with an embarrassment of examples of partition functors.

\begin{remark}
It should not be important that $G$ is finite in the above proposition. Since the index of $G \wr \Sigma_{\Lambda}$ in $G \wr \Sigma_{\Omega}$ is finite when $\Lambda \leq \Omega$, any of the extensions of global functors on finite groups to compact Lie groups should be able to be used to produce partition functors as above associated to a compact Lie group $G$.
\end{remark}

Associated to a partition $\Lambda \parts X$ and group $G$, there is a diagonal map
\[
\Delta_{\Lambda} \colon G \times \Sigma_{\Lambda} \to G \wr \Sigma_{\Lambda}
\]
induced by the diagonal map $G \to G^{X}$. Note that when $X$ is empty, $\Delta_{\Lambda} \colon G \to e$ is the unique map to the trivial group.

\begin{prop} \label{prop:diagonal}
Restriction along the diagonal induces a canonical map of partition functors
\[
M \circ (G \wr \Sigma_{(-)}) \to M \circ (G \times \Sigma_{(-)}).
\]
\end{prop}
\begin{proof}
We must show that restriction along the diagonal map commutes with the restrictions and transfers in the two partition functors.

Commuting with restrictions is free, since $M$ is a global functor. To see that restriction along the diagonal commutes with transfers, apply the double coset formula: Let $\Lambda \leq \Omega$ so that we have a commutative square of groups
\[
\xymatrix{G \times \Sigma_{\Lambda} \ar[r]^{\Delta_{\Lambda}} \ar[d]^-{\subseteq} & G \wr \Sigma_{\Lambda} \ar[d]^-{\subseteq} \\ G \times \Sigma_{\Omega} \ar[r]^{\Delta_{\Omega}} & G \wr \Sigma_{\Omega}.}
\]
Then
\[
\Delta_{\Omega}(G \times \Sigma_{\Omega}) \backslash G \wr \Sigma_{\Omega} / G \wr \Sigma_{\Lambda} \cong *
\]
and 
\[
\Delta_{\Omega}(G \times \Sigma_{\Omega})) \cap G \wr \Sigma_{\Lambda} = \Delta_{\Lambda}(G \times \Sigma_{\Lambda}).
\]
In the notation of \cref{def:globalmackey}, the double coset formula now gives
        \[
           \res_{\Delta_{\Omega}} \circ \tr_{G \wr \Sigma_{\Lambda}}^{G \wr \Sigma_{\Omega}} = \tr_{G \times \Sigma_{\Lambda}}^{G \times \Sigma_{\Omega}} \circ \res_{\Delta_{\Lambda}}.
        \]
\end{proof}

It is important to note that $M \circ (G \times \Sigma_0) \cong M(G)$ and $M \circ (G \wr \Sigma_0) \cong M(e)$. Thus the values of these two partition functors on the empty partition are often distinct.

Next we describe the constant and dual constant partition functors. We will come back to these two examples regularly as we develop the theory of partition functors.

\begin{example}
Given an abelian group $A$, recall the constant global functor associated to $A$, $\underline{A}$ (\cite[Example 4.2.8.(iii)]{Schwede}). This global functor takes the value $A$ on every finite group, the restriction maps are the identity, and the transfer map along $f \colon H \hookrightarrow G$ is given by multiplication by the index $|G/f(H)|$. Write $C_A = \underline{A} \circ (G \times \Sigma_{(-)})$. Note that the precomposition with $G \wr \Sigma_{(-)}$ gives the same partition functor and the answer is also independent of the choice of $G$.

Just as in the local case of $G$-Mackey functors for a fixed group $G$, there is a notion of a dual constant partition functor to $C_A$, that we denote $C^A$. The partition functor $C^A$ takes the value $A$ on every partition, the transfer maps are all the identity, and the restriction maps are given by multiplication by the index. That is, given $f \colon \Lambda \to \Omega$,
\[
(C^A)^*(f)(a) = |\Sigma_{\Omega} / \Sigma_{f \Lambda}| a.
\]
In both cases, the double coset formula can be verified with an application of \cref{doublecosets}.
\end{example}

\begin{prop} \label{prop:globaltriv}
Assume $M$ is a global functor and let $A = M(G)$. There is a map of partition functors
\[
M \circ (G \times \Sigma_{(-)}) \to C_A
\]
induced by restriction along the canonical inclusion $G \to G \times \Sigma_{\Lambda}$.
\end{prop}
\begin{proof}
We must check that this map commutes with restrictions and transfers. Compatibility with restrictions follows from the fact that the map is induced by a restriction map.

To check commutativity with transfers, let $\Lambda \leq \Omega$ and note that the double coset formula for $M$ gives 
\[
\res_{G}^{G \times \Sigma_{\Omega}} \circ \tr_{G \times \Sigma_\Lambda}^{G \times \Sigma_\Omega} = |(G \times \Sigma_\Omega)/(G \times \Sigma_\Lambda)| \tr_{G}^{G} \circ \res_{G}^{G \times \Sigma_\Lambda}.
\]
To conclude, note that $|(G \times \Sigma_\Omega)/(G \times \Sigma_\Lambda)| = |\Sigma_\Omega/\Sigma_\Lambda|$ and $\tr_{G}^{G}$ is the identity.
\end{proof}

\begin{remark}
Partition functors are not just global functors on the subcategory of $\Grp$ consisting of products of symmetric groups and with morphisms generated by inclusions and conjugations of Young subgroups. Consider precomposition with the wreath product, as in \cref{prop:partitionfromglobal}. Although $\Sigma_{\bar{m}} \cong \Sigma_{\bar{n}} \cong e$, when $m \neq n \in \N$, $G \wr \Sigma_{\bar{m}} \cong G^{\times m}$ and $G \wr \Sigma_{\bar{n}} \cong G^{\times n}$. The partition category distinguishes between discrete partitions of different cardinalities and provides, among other things, the flexibility to handle these wreath products. 
\end{remark}

\section{The functor $\Div$} \label{sec:div}

In this section we introduce a construction called $\Div$ that is an endofunctor of the category of partition functors and one of our main objects of study.

Let $M$ be a partition functor. Given a map of partitions $f \colon \Lambda \to \Gamma$, we will write $M^*(f)$ for the restriction along $f$ and $M_*(f)$ for the transfer along $f$. Let $\Omega \leq \Lambda$, recall that we write $f_{\Omega} \colon \Omega \to f \Omega$ for the induced isomorphism of partitions. 


For a partition $\Gamma$, let $I_{\Gamma} \subseteq M(\Gamma)$ be the subgroup
\[
I_\Gamma = \im \Big (\bigoplus_{\Omega < \Gamma} M(\Omega) \to M(\Gamma) \Big ),
\]
where the maps are given by the transfer along the strict inclusions $\Omega < \Gamma$. We will refer to this as the transfer subgroup of $M(\Gamma)$. Let
\[
\bar{M}(\Gamma) = M(\Gamma)/I_{\Gamma}.
\]

Let $f \colon \Gamma \to \Lambda$ be a map of partitions. Since the image under $f$ of every proper subpartition of $\Gamma$ is a proper subpartition of $\Lambda$, we have
\[
M_*(f)(I_{\Gamma}) \subseteq I_{\Lambda}.
\]
Thus $M_*(f)$ induces a map
\[
\bar{M}_*(f) \colon \bar{M}(\Gamma) \to \bar{M}(\Lambda).
\]
If $f$ is an isomorphism of partitions, this map is an isomorphism. If $f$ is not an isomorphism, then this map is the zero map. Let $q_{\Gamma} \colon M(\Gamma) \to \bar{M}(\Gamma)$ be the quotient map and note that
\[
q_{\Lambda}M_*(f) = \bar{M}_*(f)q_{\Gamma}.
\]

\begin{notation} \label{not:indexing}
For families indexed by partitions, we suppress the underline
on numerical indiscrete partitions used as arguments or subscripts: for example, $\Sigma_m = \Sigma_{\u m}$, $M(m) = M(\u m)$, $I_m = I_{\u m}$ and $q_m=q_{\u m}$.
\end{notation}


If $\Gamma \leq \Lambda$, then for $\sigma \in \Sigma_{\Lambda}$, we have the induced isomorphism $\sigma_{\Gamma} \colon \Gamma \to \sigma \Gamma$. Note that $\sigma \Gamma \leq \Lambda$. The transfer along $\sigma_{\Gamma}$ yields an isomorphism
of abelian groups
\[
\bar{M}_*(\sigma_{\Gamma}) \colon \bar{M}(\Gamma) \to \bar{M}(\sigma \Gamma).
\]

Note that there is no functor $\bar{M}^*$ that can be applied to arbitrary restriction maps as the double coset formula must be confronted in that case. On the other hand, since $M_*(f) = M^*(f^{-1})$ when $f$ is an isomorphism, there is a functor $\bar{M}^*$ that can be applied to give a restriction along an isomorphism.

Let $x \in \bigoplus_{\Gamma \leq \Lambda} \bar{M}(\Gamma)$. We will view $x$ as a tuple $x = (x_{\Gamma})_{\Gamma \leq \Lambda}$ with $x_{\Gamma} \in \bar{M}(\Gamma)$. The isomorphisms described above induce a left action of $\Sigma_{\Lambda}$ on $\bigoplus_{\Gamma \leq \Lambda} \bar{M}(\Gamma)$ given by
\[
(\sigma x)_{\Gamma} = \bar{M}_*(\sigma_{\sigma^{-1}\Gamma})(x_{\sigma^{-1} \Gamma}) = \bar{M}^*((\sigma^{-1})_{\Gamma})(x_{\sigma^{-1} \Gamma}).
\]
We will often prefer the formula involving the transfer and abbreviate it to
\[
(\sigma x)_{\Gamma} = \bar{M}_*(\sigma)(x_{\sigma^{-1} \Gamma}).
\]
Information is not lost in the abbreviation as the subscript on the $x$ carries the information of the source of $\sigma$.

\begin{definition}
For $M$ a partition functor, let \[
    \Div(M)(\Lambda) = \bigg (\bigoplus_{\Gamma \leq \Lambda} \bar{M}(\Gamma)  \bigg)^{\Sigma_\Lambda}.
\]
\end{definition}

Note that $\sigma \in \Sigma_{\Lambda}$ sends the summand $\bar{M}(\Lambda) \subseteq \Div(M)(\Lambda)$ to itself. Since $\Sigma_{\Lambda}$ acts trivially on $\bar{M}(\Lambda)$, because it acts trivially on $M(\Lambda)$, we have $\bar{M}(\Lambda)^{\Sigma_{\Lambda}} = \bar{M}(\Lambda)$ is a summand of $\Div(M)(\Lambda)$.

For $x \in \Div(M)(\Lambda)$ and $\sigma \in \Sigma_{\Lambda}$, we have
\[
x_{\Gamma} = \bar{M}_*(\sigma)x_{\sigma^{-1}\Gamma} = \bar{M}^*(\sigma)x_{\sigma \Gamma}
\]
since $\bar{M}^*(\sigma) = \bar{M}_*(\sigma^{-1})$. 

We will use the restriction and transfer maps along isomorphisms of partitions to define restriction and transfer maps for $\Div(M)$ that give $\Div(M)$ the structure of a partition functor. 

We begin with the restriction maps. Assume that $f \colon \Lambda \to \Theta$ is a map of partitions. We define
\[
\Div(M)^*(f) \colon \Div(M)(\Theta) \to \Div(M)(\Lambda)
\]
by the formula
\[
(\Div(M)^*(f)(x))_{\Gamma} = \bar{M}^*(f_{\Gamma})(x_{f \Gamma}).
\]
Note that this lands in the $\Sigma_{\Lambda}$-invariants as, for $\sigma \in \Sigma_{\Lambda}$ and $\Sigma_f(\sigma) \in \Sigma_{\Theta}$, we have a commutative diagram
\[
\xymatrix{\Lambda \ar[r]^-{f} \ar[d]_-{\sigma} & \Theta \ar[d]^-{\Sigma_f(\sigma)} \\ \Lambda \ar[r]^-{f} & \Theta}
\]
of partitions and $\Sigma_f(\sigma)$ acts trivially on $x$ as it is an element of $\Sigma_{\Theta}$.

Assume that $\Lambda \leq \Theta$. In this case,
\[ 
    \Div(M)^*(\Lambda \leq \Theta) \colon \Div(M)(\Theta) \to \Div(M)(\Lambda)
\]
is induced by the projection 
\[ 
    \bigoplus_{\Gamma \leq \Theta} \bar{M}(\Gamma) \to \bigoplus_{\Gamma \leq \Lambda} \bar{M}(\Gamma).
\]
The projection is $\Sigma_\Lambda$-equivariant and the $\Sigma_\Theta$-fixed points are a subset of the $\Sigma_\Lambda$-fixed-points. 

\begin{lemma}
    The restriction maps in $\Div(M)$ are functorial. 
\end{lemma}

\begin{proof}
We see that
\begin{equation*}
\begin{split}
(\Div(M)^*(gf)(x))_{\Gamma} &= \bar{M}^*((gf)_\Gamma)x_{gf\Gamma} \\
&= \bar{M}^*(f_{\Gamma})\bar{M}^*(g_{f\Gamma})x_{gf\Gamma} \\
&= \bar{M}^*(f_{\Gamma}) (\Div(M)^*(g)(x))_{f\Gamma} \\
&= (\Div(M)^*(f)\Div(M)^*(g)(x))_{\Gamma}. \qedhere
\end{split}
\end{equation*}
\end{proof}

\bigskip

Inspired by the double coset formula, we define the transfer along $f \colon \Psi \to \Lambda$ to be
\[
(\Div(M)_*(f)(x))_{\Gamma} = \sum_{\substack{[\sigma] \in \Sigma_\Gamma \backslash \Sigma_{\Lambda} / \Sigma_{f \Psi} \\  \Gamma \leq \sigma f \Psi}} \bar{M}^*(((\sigma f)^{-1})_{\Gamma})(x_{(\sigma f)^{-1} \Gamma}),
\]
for $\Gamma \leq \Lambda$. Since $((\sigma f)^{-1})_{\Gamma} \colon \Gamma \to (\sigma f)^{-1}\Gamma$ is an isomorphism this can be rewritten using transfers:
\[
(\Div(M)_*(f)(x))_{\Gamma} = \sum_{\substack{[\sigma] \in \Sigma_\Gamma \backslash \Sigma_{\Lambda} / \Sigma_{f \Psi} \\  \Gamma \leq \sigma f \Psi}} \bar{M}_*((\sigma f)_{(\sigma f)^{-1}\Gamma})(x_{(\sigma f)^{-1} \Gamma}).
\]
We will often prefer the formula involving the transfer and abbreviate it to
\[
(\Div(M)_*(f)(x))_{\Gamma} = \sum_{\substack{[\sigma] \in \Sigma_\Gamma \backslash \Sigma_{\Lambda} / \Sigma_{f \Psi} \\  \Gamma \leq \sigma f \Psi}} \bar{M}_*(\sigma f)(x_{(\sigma f)^{-1} \Gamma}).
\]
Again, information is not lost in the abbreviation as the subscript on the $x$ carries the information of the subscript of $\sigma f$.

\begin{remark}
The formula for the transfer is the special case of the double coset formula
\[
\sum_{[\sigma] \in \Sigma_\Gamma \backslash \Sigma_{\Lambda} / \Sigma_{f \Psi} } \bar{M}_*(\Gamma \cap \sigma f \Psi \leq \Gamma) \bar{M}^*(((\sigma f)^{-1})_{\Gamma \cap \sigma f\Psi})(x_{(\sigma f)^{-1} \Gamma \cap \Psi}),
\]
where $\bar{M}_*(\Gamma \cap \sigma f \Psi \leq \Gamma) = 0$ unless $\Gamma \cap \sigma f \Psi = \Gamma$.
\end{remark}

\begin{lemma}
The formula for the transfer is well-defined.    
\end{lemma}
\begin{proof}
Assume that $\ell \sigma k$, with $\ell \in \Sigma_{\Gamma}$ and $k \in \Sigma_{f \Psi}$, is another element in the double coset $[\sigma]$. Since $\ell^{-1} \Gamma = \Gamma$ and $\bar{M}_*(\ell)$ acts by the identity on $\bar{M}(\Gamma)$, replacing $\sigma$ by $\ell \sigma k$ gives
\[
\bar{M}_*(\ell \sigma k f)(x_{(\ell \sigma k f)^{-1} \Gamma}) =  \bar{M}_*(\ell)  \bar{M}_*(\sigma k f)(x_{(\sigma k f)^{-1} \ell^{-1} \Gamma}) = \bar{M}_*(\sigma k f)(x_{(\sigma k f)^{-1} \Gamma}).
\]
Since $k \in \Sigma_{f \Psi}$, $f^{-1} k f \in \Sigma_{\Psi}$. The $\Sigma_{\Psi}$-invariance of $x$ implies that
\[
x_{f^{-1} k^{-1} \sigma^{-1} \Gamma} = \bar{M}_*(f^{-1}k^{-1}f) x_{f^{-1}kff^{-1}k^{-1}\sigma^{-1}\Gamma} = \bar{M}_*(f^{-1}k^{-1}f)x_{f^{-1}\sigma^{-1}\Gamma}.
\]
This allows us to remove the $k$.
\end{proof}

\begin{lemma}
The formula for the transfer lands in $\Div(M)(\Lambda)$.
\end{lemma}
\begin{proof}
We must check that it lands in the $\Sigma_{\Lambda}$-invariants. Let $\alpha \in \Sigma_{\Lambda}$, then
\begin{equation*}
\begin{split}
(\alpha \Div(M)_*(f)(x))_{\Gamma} &= \bar{M}_*(\alpha)(\Div(M)_*(f)(x))_{\alpha^{-1}\Gamma} \\
&= \sum_{\substack{[\sigma] \in \Sigma_{\alpha^{-1}\Gamma} \backslash \Sigma_{\Lambda} / \Sigma_{f \Psi} \\ \alpha^{-1} \Gamma \leq \sigma f \Psi}} \bar{M}_*(\alpha) \bar{M}_*(\sigma f)(x_{(\sigma f)^{-1} \alpha^{-1} \Gamma}).
\end{split}
\end{equation*}
As $\Sigma_{\alpha^{-1} \Gamma} = \alpha^{-1} \Sigma_{\Gamma} \alpha$, we may make use of the bijection
\begin{equation}
\label{equation:doubleCosetIso}
\Sigma_{\Gamma} \backslash \Sigma_{\Lambda} / \Sigma_{f \Psi} \xrightarrow{\cong} \Sigma_{\alpha^{-1}\Gamma} \backslash \Sigma_{\Lambda} / \Sigma_{f \Psi}
\end{equation}
sending $[\sigma]$ to $[\alpha^{-1}\sigma]$ to rewrite the sum as 
\[
\sum_{\substack{[\sigma] \in \Sigma_{\Gamma} \backslash \Sigma_{\Lambda} / \Sigma_{f \Psi} \\ \Gamma \leq \sigma f \Psi}} \bar{M}_*(\alpha) \bar{M}_*(\alpha^{-1} \sigma f)(x_{f^{-1} \sigma^{-1} \alpha \alpha^{-1} \Gamma}) = \sum_{\substack{[\sigma] \in \Sigma_{\Gamma} \backslash \Sigma_{\Lambda} / \Sigma_{f \Psi} \\ \Gamma \leq \sigma f \Psi}} \bar{M}_*(\sigma f)(x_{(\sigma f)^{-1}\Gamma})
\]
and this is what we wanted.
\end{proof}

Note the special case of the transfer when $f \colon \Psi \to \Lambda$ is an isomorphism. In this case $\Sigma_{\Gamma} \backslash \Sigma_{\Lambda}/ \Sigma_{f\Psi} = \{[e]\}$ and
\[
(\Div(M)_*(f)(x))_{\Gamma} = \bar{M}_*(f)(x_{f^{-1} \Gamma}) = (\Div(M)^*(f^{-1})(x))_{\Gamma}.
\]

\begin{lemma}
\label{DivFunctorialTransfers}
The transfer maps for $\Div(M)$ are functorial. 
\end{lemma}
\begin{proof}
Assume given $g \colon \Psi \to \Lambda$ and $f \colon \Lambda \to \Omega$. 
We have
\[
(\Div(M)_*(fg)(x))_{\Gamma \leq \Omega} = \sum_{\substack{[\beta] \in \Sigma_\Gamma \backslash \Sigma_{\Omega} / \Sigma_{fg \Psi} \\ \Gamma \leq \beta fg \Psi}}  \bar{M}_*(\beta fg)(x_{(\beta f g)^{-1} \Gamma}).
\]
On the other hand, $(\Div(M)_*(f) \Div(M)_*(g) (x))_{\Gamma \leq \Omega}$ is the double summation
\begin{equation*}
\label{composedDivTransfers}
\sum_{\substack{[\sigma] \in \Sigma_{\Gamma} \backslash \Sigma_{\Omega} / \Sigma_{f \Lambda} \\ \Gamma \leq \sigma f \Lambda}} \Bigg ( \sum_{\substack{[\alpha] \in \Sigma_{(\sigma f)^{-1} \Gamma} \backslash \Sigma_{\Lambda} / \Sigma_{g \Psi} \\  (\sigma f)^{-1}\Gamma \leq \alpha g \Psi}} \bar{M}_*(\sigma f)\bar{M}_*(\alpha g)(x_{g^{-1}\alpha^{-1}f^{-1}\sigma^{-1}\Gamma}) \Bigg ).
\end{equation*}
Conjugation by $f_{\Lambda}$ allows us to rewrite this as
\begin{equation*}
\sum_{\substack{[\sigma] \in \Sigma_{\Gamma} \backslash \Sigma_{\Omega} / \Sigma_{f \Lambda} \\ \Gamma \leq \sigma f \Lambda}} \Bigg ( \sum_{\substack{[\alpha] \in \Sigma_{\sigma^{-1} \Gamma} \backslash \Sigma_{f \Lambda} / \Sigma_{fg \Psi} \\ \sigma^{-1}\Gamma \leq \alpha f g \Psi}} \bar{M}_*(\sigma f)\bar{M}_*(f^{-1}\alpha f g)(x_{g^{-1}f^{-1}\alpha^{-1} ff^{-1}\sigma^{-1}\Gamma}) \Bigg ).
\end{equation*}
Note that the condition $\sigma^{-1}\Gamma \leq \alpha f g \Psi$ is equivalent to $\Gamma \leq \sigma \alpha fg \Psi$. Further, since $\alpha \in \Sigma_{f \Lambda}$, we have $\sigma \alpha f g \Psi \leq \sigma \alpha f \Lambda = \sigma f \Lambda$, so the condition $\Gamma \leq \sigma \alpha fg \Psi$ implies that $\Gamma \leq \sigma f \Lambda$. 

Now the result follows from the (opposite version of the) bijection of \cref{doublecosets}, which allows us to sum over the double cosets $\Sigma_{\Gamma} \backslash \Sigma_{\Omega} / \Sigma_{fg \Psi}$ and replace $\sigma \alpha$ by $\beta$.
\end{proof}

We have shown that $\Div(M)$ admits functorial transfer and restriction maps. To see that $\Div(M)$ is a partition functor, we must address the double coset formula.

\begin{lemma}
    The construction $\Div(M)$ with the restriction and transfer maps described above satisfies the double coset formula \cref{doubleCoset}. 
\end{lemma}
\begin{proof}
Assume that $\Psi, \Phi \leq \Lambda$. We wish to show that
\begin{equation*}
\begin{split}
\Div(M)^*(\Psi \leq \Lambda)&\Div(M)_*(\Phi \leq \Lambda) \\ &= \sum_{[\sigma] \in \Sigma_{\Psi} \backslash \Sigma_{\Lambda} / \Sigma_{\Phi}} \Div(M)_*(\Psi \cap \sigma \Phi \leq \Psi) \Div(M)_*(\sigma) \Div(M)^*(\sigma^{-1} \Psi \cap \Phi \leq \Phi).
\end{split}
\end{equation*}
We shall first consider
\[
(\Div(M)^*(\Psi \leq \Lambda)\Div(M)_*(\Phi \leq \Lambda)(x))_{\Gamma \leq \Psi}.
\]
Since the restriction is along an inclusion, this is
\[
\sum_{\substack{[\beta] \in \Sigma_{\Gamma} \backslash \Sigma_{\Lambda} / \Sigma_{\Phi} \\ \Gamma \leq \beta \Phi}} \bar{M}_*(\beta)(x_{\beta^{-1}\Gamma}).
\]
The other sum, evaluated at $x \in \Div(M)(\Phi)$ and $\Gamma \leq \Psi$, is given by
\[
\sum_{\substack{[\sigma] \in \Sigma_{\Psi} \backslash \Sigma_{\Lambda} / \Sigma_{\Phi}}}  \Bigg(\sum_{\substack{[\alpha] \in \Sigma_{\Gamma} \backslash \Sigma_{\Psi} / \Sigma_{\Psi \cap \sigma \Phi} \\ \Gamma \leq \alpha(\Psi \cap \sigma \Phi)}} \bar{M}_*(\alpha) \bar{M}_*(\sigma) (x_{\sigma^{-1} \alpha^{-1} \Gamma})\Bigg).
\]
The bijection of \cref{doublecosets} can be applied to set $\alpha \sigma = \beta$. But we must check that the conditions on the sums agree. Note that, since $\alpha \in \Sigma_{\Psi}$, we have $\alpha \Psi = \Psi$. Now $\Gamma \leq \alpha \Psi \cap \alpha \sigma \Phi = \Psi \cap \beta \Phi$ implies that $\Gamma \leq \beta \Phi$. To go the other way, if $\Gamma \leq \beta \Phi$ and $\Gamma \leq \Psi$, then $\Gamma \leq \Psi \cap \beta \Phi$.
\end{proof}

The construction $\Div$ is an endofunctor on the category of partition functors. Given a map of partition functors $f \colon M \to N$, there is an induced map of partition functors $\Div(f) \colon \Div(M) \to \Div(N)$ defined in the following way: Given a partition $\Lambda$ and $\Gamma < \Lambda$, the fact that $f$ is a map of partition functors implies that we have a commutative diagram
\[
\xymatrix{M(\Gamma) \ar[r] \ar[d]_-{M_*(\Gamma < \Lambda)} & N(\Gamma) \ar[d]^-{N_*(\Gamma < \Lambda)} \\ M(\Lambda) \ar[r] & N(\Lambda).}
\]
The direct sum of these over $\Gamma < \Lambda$ induces a map of abelian groups
\[
\bar{M}(\Lambda) \to \bar{N}(\Lambda).
\]
Further taking the direct sum of these maps together over $\Lambda \leq \Omega$ gives
\[
\bigoplus_{\Lambda \leq \Omega} \bar{M}(\Lambda) \to \bigoplus_{\Lambda \leq \Omega} \bar{N}(\Lambda).
\]
The fact that $M$ and $N$ are partition functors implies that this map is $\Sigma_\Omega$-equivariant, and so it passes to fixed points giving the desired map
\[
\Div(f)(\Omega) \colon \Div(M)(\Omega) \to \Div(N)(\Omega).
\]
It is straight-forward to see that this construction sends the identity to the identity and respects composition. We have proven the following: 

\begin{theorem} \label{thm:divendo}
    The construction $\Div$ is an endofunctor on the category of partition functors. 
\end{theorem}

\begin{remark}
In \cite{Thevenaz}, Thevenaz studies a construction on $G$-Mackey functors that he calls the twin functor. The twin of a $G$-Mackey functor $M$ is built in the same way as $\Div$, but uses all subgroups of $G$. He shows that the twin of $M$ is another $G$-Mackey functor. His work is primarily focused on the case that $|G|$ is invertible in $M(G)$, but he studies the general case as well.
\end{remark}


\begin{example} \label{divconstant}
Let $A$ be an abelian group. We consider $\Div(C_A)$ and $\Div(C^A)$, where $C_A$ is the constant partition functor and $C^A$ is the dual constant partition functor.

We begin with $C^A$. Since transfer maps are the identity for $C^A$, we have $C^A(\Lambda)/I_{\Lambda} = 0$ unless $\Lambda$ is the discrete partition, in which case there are no subpartitions of $\Lambda$ and $C^A(\Lambda)/I_{\Lambda} = C^A(\Lambda) = A$. It follows that 
\[
\Div(C^A)(\Omega) = A,
\]
concentrated in the summand associated to the discrete subpartition of $\Omega$.
Further, given $f \colon \Lambda \to \Omega$, we have
\[
(\Div(C^A)^*(f)(a))_\Gamma = \overline{C_A}^*(f_{\Gamma})(a_{f\Gamma}).
\]
Since $a_{f \Gamma} = 0$ unless $\Gamma$ is the discrete subpartition of $\Lambda$ and $\overline{C_A}^*(f_{\Gamma})$ is the identity map as it is a restriction along an isomorphism. We learn that $\Div(C^A)^*(f)$ is the identity map. On the other hand,  
\[
(\Div(C^A)_*(f)(a))_\Gamma = \sum_{\substack{ [\sigma] \in \Sigma_{\Gamma} \backslash \Sigma_{\Omega} / \Sigma_{\Lambda} \\ \Gamma \leq \sigma \Lambda}} \overline{C^A}_*(\sigma) a_{\sigma^{-1} \Gamma}.
\]
If $\Gamma$ is not the discrete subpartition of $\Omega$ we get $0$. If $\Gamma$ is the discrete subpartition of $\Omega$, then $\Gamma \leq \sigma \Lambda$ is a vacuous condition, $\sigma^{-1}\Gamma$ is always the discrete subpartition, and we get
\[
(\Div(C^A)_*(f)(a))_\Gamma = |\Sigma_{\Omega}/\Sigma_{\Lambda}|a.
\]
It follows that
\[
\Div(C^A) = C_A.
\]

On the other hand, in the case of the constant partition functor $C_A$, the transfer subgroup $I_\Lambda \subseteq C_A(\Lambda) = A$ is a bit more complicated. First assume that $\Lambda = \u m$ is the indiscrete partition. For $i+j = m$, $(C_A)_*(\u i \sqcup \u j \leq \u m)$ is multiplication by $m \choose i$. 
Let $\ell_1 = 0$ and, for $m >1$, let
\[
\ell_m = \gcd \bigg{(} {m \choose 1}, {m \choose 2}, \ldots {m \choose {m-1}} \bigg{)} = \left\{
\begin{array}{rl}
p & \text{if } m = p^k \text{ for some prime } p,\\
1 & \text{otherwise.}
\end{array}\right.
\]
Since every $\Lambda < \u m$ factors through $\u i \sqcup \u j$ for some $i, j$, we have
\[
\overline{C_A}(m) = C_{A}(m)/I_{m} = A/\ell_m.
\]
It follows that, for a partition $\Lambda = (X, \sim_{\Lambda})$, if we set
\[
\ell_{\Lambda} = \gcd( (\ell_{|[x]|})_{[x] \in X/\sim_{\Lambda}} ),
\]
then
\[
\overline{C_A}(\Lambda) = A/\ell_{\Lambda}
\]
and
\[
\Div(C_{A})(\Omega) = \big (\bigoplus_{\Lambda \leq \Omega} A/\ell_{\Lambda} \big )^{\Sigma_{\Omega}}.
\]
Note that, if $A$ is a $\Q$-vector space, then $\Div(C_{A})(\Omega) = A$, concentrated in the summand associated to the discrete subpartition of $\Omega$. It follows that $\Div(C_A) = C_A$ in this case.
\end{example}

\section{$\Div$ as a monad} \label{sec:divmonad}

In \cref{thm:divendo} we showed that $\Div$ is an endofunctor on the category of partition functors. Our next goal is to show that $\Div$ admits the structure of a monad on the category of partition functors.

There is a canonical map of abelian groups $\eta_{\Lambda} \colon M(\Lambda) \to \Div(M)(\Lambda)$, constructed as follows. For $\Gamma \leq \Lambda$, recall that 
\[
q_{\Gamma}M^*(\Gamma \leq \Lambda) \colon M(\Lambda) \to \bar{M}(\Gamma)
\]
is the restriction to $M(\Gamma)$ followed by the quotient by the subgroup $I_{\Gamma}$. These maps are called Brauer morphisms in \cite{Thevenaz}. To produce the unit of the monad, we must check that, after taking the direct sum over $\Gamma \leq \Lambda$, this lands in $\Div(M)(\Lambda)$. That is, we must check that it lands in the $\Sigma_{\Lambda}$-invariants. For $z \in M(\Lambda)$, let
\[
\eta_{\Lambda}(z) = (q_{\Gamma}M^*(\Gamma \leq \Lambda)(z))_{\Gamma \leq \Lambda}.
\]
For $\sigma \in \Sigma_{\Lambda}$, we have
\[
(\sigma \eta_{\Lambda}(z))_{\Gamma} = \bar{M}^*(\sigma^{-1})(\eta_{\Lambda}(z))_{\sigma^{-1}\Gamma} = \bar{M}^*(\sigma^{-1})q_{\sigma^{-1}\Gamma}M^*(\sigma^{-1}\Gamma \leq \Lambda)(z).
\]
But, since we have the commutative diagram of partitions where the top row is an automorphism
\[
\xymatrix{\Lambda \ar[r]^{\sigma^{-1}} & \Lambda \\ \Gamma \ar[r]^{\sigma^{-1}} \ar[u] & \sigma^{-1} \Gamma, \ar[u]}
\]
we see that
\[
\bar{M}^*(\sigma^{-1})q_{\sigma^{-1}\Gamma}M^*(\sigma^{-1}\Gamma \leq \Lambda)(z) = q_{\Gamma}M^*(\sigma^{-1})M^*(\sigma^{-1}\Gamma \leq \Lambda)(z) = q_{\Gamma} M^*(\Gamma \leq \Lambda)(z).
\]

\begin{lemma}
\label{lemma:unitMapForDivMonad}
The maps $\eta_{\Lambda} \colon M(\Lambda) \to \Div(M)(\Lambda)$ assemble into a map of partition functors $\eta \colon M \to \Div(M)$ natural in $M$.
\end{lemma}
\begin{proof}
For $f \colon \Lambda \to \Omega$, we must check that we have a commutative diagram involving restrictions
\[
    \begin{tikzcd}
        M(\Omega)
            \ar[d, "M^*(f)"']
            \ar[r, "\eta_\Omega"] 
            & 
        \Div(M)(\Omega)
            \ar[d, "\Div(M)^*(f)"] 
            \\
        M(\Lambda)
            \ar[r, "\eta_\Lambda"] 
            & 
        \Div(M)(\Lambda)
    \end{tikzcd}
\]
Going around the bottom we see that
\[
(\eta_{\Lambda}(M^*(f)(z)))_{\Gamma \leq \Lambda} = q_{\Gamma}M^*(\Gamma \leq \Lambda)M^*(f)(z).
\]
Going around the top, we have
\[
(\Div(M)^*(f)(\eta_{\Omega}(z)))_{\Gamma \leq \Lambda} = \bar{M}^*(f_{\Gamma})(\eta_{\Omega}(z))_{f\Gamma} = \bar{M}^*(f_{\Gamma})q_{f\Gamma}M^*(f\Gamma \leq \Omega)(z).
\]
Now these two elements agree as there is a commutative diagram of partitions
\[
\xymatrix{\Lambda \ar[r]^{f} & \Omega \\ \Gamma \ar[r] \ar[u]^{\leq} & f\Gamma. \ar[u]_{\leq}}
\]

Next we would like to show that we have a commutative diagram involving transfers
\[
    \begin{tikzcd}
        M(\Omega) 
            \ar[r, "\eta_\Omega"]
            & 
        \Div(M)(\Omega) 
            \\
        M(\Lambda)
            \ar[r, "\eta_\Lambda"]
            \ar[u, "M_*(f)"]
            & 
        \Div(M)(\Lambda).
            \ar[u, "\Div(M)_*(f)"']
    \end{tikzcd}
\]
Going around the top, we have 
\[
(\eta_{\Omega}(M_*(f)(z)))_{\Gamma \leq \Omega} = q_{\Gamma}M^*(\Gamma \leq \Omega)M_*(f)(z).
\]
Applying the double coset formula gives
\[
q_{\Gamma} \Big ( \sum_{[\sigma] \in \Sigma_{\Gamma} \backslash \Sigma_{\Omega} / \Sigma_{f \Lambda}} M_*(\Gamma \cap \sigma f \Lambda \leq \Gamma) M_*(\sigma f) M^*((\sigma f)^{-1} \Gamma \cap \Lambda \leq \Lambda)(z) \Big ).
\]
Since $q_{\Gamma} M_*(\Gamma \cap \sigma f \Lambda \leq \Gamma) = 0$ unless $\Gamma \leq \sigma f \Lambda$, in which case $M_*(\Gamma \cap \sigma f \Lambda \leq \Gamma)$ is the identity map, this sum can be rewritten as
\[
\sum_{\substack{[\sigma] \in \Sigma_{\Gamma} \backslash \Sigma_{\Omega} / \Sigma_{f \Lambda} \\ \Gamma \leq \sigma f \Lambda}} q_{\Gamma} M_*(\sigma f) M^*((\sigma f)^{-1} \Gamma \leq \Lambda)(z).
\]

On the other hand, going around the bottom, we have
\[
(\Div(M)_*(f)\eta_{\Lambda}(z))_{\Gamma} = \sum_{\substack{[\sigma] \in \Sigma_{\Gamma} \backslash \Sigma_{\Omega} / \Sigma_{f \Lambda} \\ \Gamma \leq \sigma f \Lambda}} \bar{M}_*(\sigma f)q_{(\sigma f)^{-1} \Gamma} M^*((\sigma f)^{-1}\Gamma \leq \Lambda)z.
\]
These two formulas agree as $q_{\Gamma} M_*(\sigma f) = \bar{M}_*(\sigma f)q_{(\sigma f)^{-1} \Gamma}$.

The naturality of $\eta$ in the partition functor $M$ follows from the linearity of all of the maps involved.
\end{proof}

We move on to the multiplication on $\Div$. We begin by describing a natural transformation
\[
\mu \colon \Div(\Div) \to \Div.
\]
Let $\Omega$ be a partition and let $M$ be a partition functor. By definition
\[
\Div(\Div(M))(\Omega) = (\bigoplus_{\Gamma \leq \Omega} \overline{\Div(M)}(\Gamma))^{\Sigma_{\Omega}}.
\]
Let 
\[
\pi_{\Gamma} \colon \Div(M)(\Gamma) \to \bar{M}(\Gamma)
\]
be the projection onto the summand $\bar{M}(\Gamma)$ of $\Div(M)(\Gamma)$. Note that 
\[
\ker(\pi_{\Gamma}) = (\bigoplus_{\Theta < \Gamma} \bar{M}(\Theta))^{\Sigma_{\Gamma}}.
\]

\begin{lemma}
The transfer subgroup $J_{\Gamma} \subseteq \Div(M)(\Gamma)$ is contained in the kernel of $\pi_{\Gamma}$.    
\end{lemma}
\begin{proof}
Assume that $\Theta < \Gamma$ and let $x \in \Div(M)(\Theta)$. Then
\[
(\Div(M)_*(\Theta < \Gamma)(x))_{\Gamma} = \sum_{\substack{[\sigma] \in \Sigma_{\Gamma} \backslash \Sigma_{\Gamma} / \Sigma_{\Theta} \\ \Gamma \leq \sigma \Theta}} \bar{M}_*(\sigma) x_{\sigma^{-1}\Gamma}.
\]
Since $\Theta$ is a proper subpartition of $\Gamma$, the condition $\Gamma \leq \sigma \Theta$ is never satisfied. It follows that $J_{\Gamma} \subseteq \ker(\pi_{\Gamma})$.

\end{proof}

This lemma gives us a map
\[
\bar{\pi}_{\Gamma} \colon \overline{\Div(M)}(\Gamma) \to \bar{M}(\Gamma).
\]
Taking the direct sum of these maps over $\Gamma \leq \Omega$ gives a $\Sigma_{\Omega}$-equivariant map
\[
\bigoplus_{\Gamma \leq \Omega} \overline{\Div(M)}(\Gamma) \to \bigoplus_{\Gamma \leq \Omega} \bar{M}(\Gamma)
\]
and finally taking fixed points gives us
\[
\mu_\Omega \colon \Div(\Div(M))(\Omega) \to \Div(M)(\Omega)
\]
defined by
\[
\mu_\Omega((x_{\Theta, \Gamma})_{\Theta \leq \Gamma})_{\Gamma \leq \Omega} = (\bar{\pi}_{\Gamma}((x_{\Theta, \Gamma})_{\Theta \leq \Gamma}))_{\Gamma \leq \Omega} =  (x_{\Gamma, \Gamma})_{\Gamma \leq \Omega}
\]
for $((x_{\Theta, \Gamma})_{\Theta \leq \Gamma})_{\Gamma \leq \Omega} \in \Div(\Div(M))(\Omega)$.

\begin{lemma}
The maps $\mu_\Omega \colon \Div(\Div (M))(\Omega) \to \Div(M)(\Omega)$ assemble into a map of partition functors that is natural in $M$.
\end{lemma}
\begin{proof}
That $\mu$ respects restriction maps is straight-forward. We will check that it respects transfer maps. 

Let $f \colon \Omega \to \Psi$ be map of partitions. We wish to show that
\[
\xymatrix{\Div(\Div(M))(\Omega) \ar[r] \ar[d]_{\Div(\Div(M))_*(f)} & \Div(M)(\Omega) \ar[d]^{\Div(M)_*(f)} \\ \Div(\Div(M))(\Psi) \ar[r] & \Div(M)(\Psi)}
\]
commutes. Let $((x_{\Theta, \Gamma})_{\Theta \leq \Gamma})_{\Gamma \leq \Omega} \in \Div(\Div(M))(\Omega)$ and let $\Phi \leq \Psi$. Going around the top, the value at $\Phi$ is
\[
\sum_{\substack{[\sigma] \in \Sigma_{\Phi} \backslash \Sigma_{\Psi} / \Sigma_{f \Omega} \\ \Phi \leq \sigma f \Omega}} \bar{M}_*(\sigma f)x_{(\sigma f)^{-1} \Phi,(\sigma f)^{-1} \Phi}.
\]
Going around the bottom, the value at $\Phi$ is
\begin{equation*}
\begin{split}
\sum_{\substack{[\sigma] \in \Sigma_{\Phi} \backslash \Sigma_{\Psi} / \Sigma_{f \Omega} \\ \Phi \leq \sigma f \Omega}} \pi_{\Phi} \overline{\Div(M)}_*(\sigma f) (x_{\Theta, (\sigma f)^{-1}\Phi})_{\Theta \leq (\sigma f)^{-1}\Phi} = \sum_{\substack{[\sigma] \in \Sigma_{\Phi} \backslash \Sigma_{\Psi} / \Sigma_{f \Omega} \\ \Phi \leq \sigma f \Omega}} \bar{M}_*(\sigma f)x_{(\sigma f)^{-1} \Phi,(\sigma f)^{-1} \Phi}.
\end{split}
\end{equation*}
This is due to the fact that the transfer map 
\[
\overline{\Div(M)}_*(\sigma f) = \overline{\Div(M)}_*((\sigma f)_{(\sigma f)^{-1}\Phi})
\]
is along the isomorphism $(\sigma f)_{(\sigma f)^{-1}\Phi} \colon (\sigma f)^{-1}\Phi \to \Phi$ so only has one summand. 

The naturality of $\mu$ in the partition functor $M$ follows from the linearity of all of the maps involved.

\end{proof}

\begin{theorem} \label{thm:divmonad}
The functor $\Div$, together with the natural transformations $\eta$ and $\mu$, is a monad on the category of partition functors.
\end{theorem}
\begin{proof}
We must show that $\eta$ is the unit for the multiplication $\mu$ and that $\mu$ is associative.

Associativity of $\mu$ follows immediately from the interpretation of $\mu$ as a sum of projections. 

The relations involving $\eta$ are more interesting. When necessary, we will subscript natural transformations by the source partition functor in order to distinguish them. 

We must show that $\mu \Div(\eta_M) = 1_{\Div(M)}$ and that $\mu \eta_{\Div(M)} = 1_{\Div(M)}$. Fix a partition $\Omega$ and let $(x_{\Gamma})_{\Gamma \leq \Omega} \in \Div(M)(\Omega)$, we may calculate that
\begin{equation*}
\begin{split}
\Div(\eta_M)_\Omega((x_{\Gamma})_{\Gamma \leq \Omega}) &= (q_{\Gamma} \eta_{M,\Gamma}(x_{\Gamma}'))_{\Gamma \leq \Omega} \\ 
&= (q_{\Gamma}(q_{\Theta}M^*(\Theta \leq \Gamma)x_{\Gamma}')_{\Theta \leq \Gamma})_{\Gamma \leq \Omega},
\end{split}
\end{equation*}
where $x_{\Gamma}'$ is a choice of lift of $x_{\Gamma}$ to $M(\Gamma)$. Now
\begin{equation*}
\begin{split}
\mu_\Omega (q_{\Gamma}(q_{\Theta}M^*(\Theta \leq \Gamma)x_{\Gamma}')_{\Theta \leq \Gamma})_{\Gamma \leq \Omega} &= (\pi_{\Gamma}(q_{\Theta}M^*(\Theta \leq \Gamma)x_{\Gamma}')_{\Theta \leq \Gamma})_{\Gamma \leq \Omega} \\
&= (q_{\Gamma}x_{\Gamma}')_{\Gamma \leq \Omega} \\
&= (x_{\Gamma})_{\Gamma \leq \Omega}.
\end{split}
\end{equation*}
Note that, in the above, we have two distinct quotient maps
\[
q_{\Gamma} \colon \Div(M)(\Gamma) \to \overline{\Div(M)}(\Gamma) \text{ and } q_{\Theta} \colon M(\Theta) \to \bar{M}(\Theta).
\]

Similarly, we have
\begin{equation*}
\begin{split}
\mu_\Omega \eta_{\Div(M), \Omega}((x_\Gamma)_{\Gamma \leq \Omega}) &= \mu_{\Omega}(q_{\Gamma} \Div(M)^*(\Gamma \leq \Omega) (x_{\Gamma}))_{\Gamma \leq \Omega} \\
&= \mu_\Omega(q_{\Gamma} (x_{\Theta})_{\Theta \leq \Gamma})_{\Gamma \leq \Omega} \\
&= (\pi_{\Gamma} (x_{\Theta})_{\Theta \leq \Gamma})_{\Gamma \leq \Omega} \\
&= (x_{\Gamma})_{\Gamma \leq \Omega}.
\end{split}
\end{equation*}
\end{proof}

Fix a partition $\Omega$. We conclude this section with a further description of $\Div(M)(\Omega)$. For $\Lambda \leq \Omega$, recall that $W_{\Sigma_{\Omega}}(\Sigma_{\Lambda})$ is the Weyl group of $\Sigma_{\Lambda}$ in $\Sigma_{\Omega}$. 

We establish some notation that will be useful later. Let $z \in \bar{M}(\Lambda)^{W_{\Sigma_{\Omega}}(\Sigma_{\Lambda})}$. Define $\Sigma_{\Omega} z \in \Div(M)(\Omega)$ by setting, for $\Phi \leq \Omega$, 
\begin{equation} \label{eq:weyl}
(\Sigma_{\Omega} z)_{\Phi} = \left\{
\begin{array}{rl}
\bar{M}_*(\beta)z & \text{for } \beta \in \Sigma_{\Omega} \text{ with } \beta \Lambda = \Phi,\\
0 & \text{otherwise.}
\end{array} \right.
\end{equation}
We see that $\Sigma_{\Omega} z$ is a well-defined element of $\bigoplus_{\Phi \leq \Omega} \bar{M}(\Phi)$ since if $\beta \Lambda = \Phi$ and $\beta' \Lambda = \Phi$, then $(\beta')^{-1} \beta \Lambda = \Lambda$ so $(\beta')^{-1} \beta \in W_{\Sigma_{\Omega}}(\Sigma_{\Lambda})$ and $z$ is fixed by the Weyl group. We see that $\Sigma_{\Omega}z \in \Div(M)(\Omega)$ as, for $\sigma \in \Sigma_{\Omega}$, we have
\[
(\sigma \Sigma_{\Omega}z)_{\Phi} = \bar{M}_*(\sigma)(\Sigma_{\Omega}z)_{\sigma^{-1}\Phi} 
\]
and
\begin{equation*}
\bar{M}_*(\sigma)(\Sigma_{\Omega}z)_{\sigma^{-1}\Phi}  = \left\{
\begin{array}{rl}
\bar{M}_*(\sigma)\bar{M}_*(\beta)z & \text{for } \beta \in \Sigma_{\Omega} \text{ with } \beta \Lambda = \sigma^{-1} \Phi,\\
0 & \text{otherwise.}
\end{array} \right.
\end{equation*}
But 
\[
\bar{M}_*(\sigma)\bar{M}_*(\beta)z = \bar{M}_*(\sigma \beta)z
\]
with $\sigma \beta \Lambda = \Phi$.

Since, given $x \in \Div(M)(\Omega)$, we have $x_{\Lambda} \in \bar{M}(\Lambda)^{W_{\Sigma_{\Omega}}(\Sigma_{\Lambda})}$ and this gives an inverse to the operation taking $z$ to $\Sigma_{\Omega}z$, we have proven the following proposition:
\begin{prop} \label{weylgroupfixed}
There is an isomorphism of abelian groups
\[
\Div(M)(\Omega) \cong \bigoplus_{[\Lambda] \in \{\Lambda \leq \Omega\}/\Sigma_{\Omega}} \bar{M}(\Lambda)^{W_{\Sigma_{\Omega}}(\Sigma_{\Lambda})},
\]
determined by a choice of subpartition $\Lambda \leq \Omega$ in each orbit for the action of $\Sigma_{\Omega}$ on the set of subpartitions of $\Omega$.
\end{prop}

If $M$ is a partition functor, then so is the functor $M \otimes \Q$ defined by $(M \otimes \Q)(\Lambda) = M(\Lambda)\otimes \Q$.  

\begin{prop} \label{prop:divrational}
For every partition functor $M$, the unit of the monad induces a natural isomorphism of partition functors
\[
\eta \otimes \Q \colon M \otimes \Q \xrightarrow{\cong} \Div(M) \otimes \Q.
\]
Thus, after rationalization, $\Div$ is naturally isomorphic to the identity functor on the category of partition functors.
\end{prop}

\begin{proof}[Proof sketch]
Rationalization commutes with the quotients, finite direct sums, and finite-group fixed points used to define $\Div$, so it suffices to assume that $M$ takes values in $\Q$-vector spaces.

Fix a partition $\Omega$. Choose representatives
\[
\Lambda_1,\ldots,\Lambda_k
\]
for the $\Sigma_\Omega$-orbits of subpartitions of $\Omega$, ordered so that if $\Lambda_i \leq \sigma\Lambda_j$ for some $\sigma \in \Sigma_\Omega$, then $i \leq j$. We may take $\Lambda_k=\Omega$. By \cref{weylgroupfixed}, there is an isomorphism
\[
\Div(M)(\Omega)
\cong
\bigoplus_{i=1}^k
\bar{M}(\Lambda_i)^{W_{\Sigma_\Omega}(\Sigma_{\Lambda_i})}.
\]

Filter $M(\Omega)$ by
\[
F_i M(\Omega)
=
\sum_{j\leq i}
\im\big(M_*(\Lambda_j\leq\Omega)\big)
\]
and filter $\Div(M)(\Omega)$ by the first $i$ summands in the decomposition above. The double coset formula implies that $\eta_\Omega$ preserves these filtrations. Indeed, the $\Lambda_\ell$-component of
\[
\eta_\Omega M_*(\Lambda_j\leq\Omega)(x)
\]
vanishes unless $\Lambda_\ell\leq\sigma\Lambda_j$ for some $\sigma\in\Sigma_\Omega$, and hence unless $\ell\leq j$.

Transfer induces a surjection
\[
\bar{M}(\Lambda_i)
\longrightarrow
F_iM(\Omega)/F_{i-1}M(\Omega).
\]
This factors through the coinvariants $\bar{M}(\Lambda_i)_{W_{\Sigma_\Omega}(\Sigma_{\Lambda_i})}$. The composite of this map with the map induced by $\eta_\Omega$ on the associated graded is the norm map
\[
\bar{M}(\Lambda_i)_{W_{\Sigma_\Omega}(\Sigma_{\Lambda_i})}
\longrightarrow
\bar{M}(\Lambda_i)^{W_{\Sigma_\Omega}(\Sigma_{\Lambda_i})}.
\]
This follows from the double coset formula: after passing to $\bar{M}(\Lambda_i)$, the only terms that survive in
\[
M^*(\Lambda_i\leq\Omega)M_*(\Lambda_i\leq\Omega)
\]
are those indexed by elements of the Weyl group.

Since $M$ takes values in $\Q$-vector spaces, the norm map from coinvariants to invariants is an isomorphism. It follows that $\eta_\Omega$ induces an isomorphism on each associated graded piece. The filtrations are finite, so
\[
\eta_\Omega\colon M(\Omega)\xrightarrow{\cong}\Div(M)(\Omega)
\]
is an isomorphism. Since $\eta$ is a natural map of partition functors by \cref{lemma:unitMapForDivMonad}, the result follows.
\end{proof}

\begin{example} \label{ex:monadconst}
Note that $\Div$ is not an idempotent monad. This is illustrated by the case of the dual constant partition functor $C^A$ for which $\Div(C^A) = C_A$, the constant partition functor, and $\Div(C_A)$ often has torsion. See \cref{divconstant} for the details of this.

Let $\Gamma \leq \Lambda$ be the discrete partition of $\Lambda$. The map $\eta(\Lambda) \colon C^A(\Lambda) \to \Div(C^A)(\Lambda) \cong C_A(\Lambda)$ is induced by the restriction map $(C^A)^*(\Gamma \leq \Lambda) \colon C^A(\Lambda) \to C^A(\Gamma)$ which is multiplication by $|\Sigma_{\Lambda}|$. Thus the map of partition functors $C^A \to C_A$ given by $\eta$ at the partition $\Lambda$ is the map $A \to A$ given by multiplication by $|\Sigma_{\Lambda}|$.
\end{example}

\section{Lax symmetric monoidal partition functors}
Recall that the $(2,1)$-category of partitions has a symmetric monoidal structure given by $\sqcup$ and that the category of abelian groups has a symmetric monoidal structure given by $\otimes$. We say that a partition functor $R$ is lax symmetric monoidal if $R$ is equipped with a map of abelian groups $\eta_{0} \colon \Z \to R(0)$ and maps
\[
m_{\Lambda, \Omega} \colon R(\Lambda) \otimes R(\Omega) \to R(\Lambda \sqcup \Omega).
\]
that are natural with respect to pairs of restriction maps as well as pairs of transfer maps and satisfy the appropriate associativity, unitality, and symmetry identities. In particular, given maps of partitions $f \colon \Gamma \to \Lambda$ and $g \colon \Theta \to \Omega$, there are commutative squares
\[
\xymatrix{R(\Gamma) \otimes R(\Theta) \ar[r]^{m_{\Gamma,\Theta}} \ar[d]_{R_*(f) \otimes R_*(g)} & R(\Gamma \sqcup \Theta) \ar[d]^{R_*(f \sqcup g)} \\ R(\Lambda) \otimes R(\Omega) \ar[r] & R(\Lambda \sqcup \Omega)}
\hspace*{1cm}
\xymatrix{R(\Lambda) \otimes R(\Omega) \ar[r] \ar[d]_{R^*(f) \otimes R^*(g)} & R(\Lambda \sqcup \Omega) \ar[d]^{R^*(f \sqcup g)} \\ R(\Gamma) \otimes R(\Theta) \ar[r] & R(\Gamma \sqcup \Theta)}
\]
as well as a commutative diagram
\[
\begin{tikzcd}[column sep=2cm]
    R(\Lambda) \otimes R(\Omega)
        \ar[r, "\tau_{R(\Lambda), R(\Omega)}", "\cong"']
        \ar[d]
        &
    R(\Omega) \otimes R(\Lambda)
        \ar[d]
        \\
    R(\Lambda \sqcup \Omega) 
        \ar[r, "R(\tau_{\Lambda,\Omega})", "\cong"']
        & 
    R(\Omega \sqcup \Lambda).
\end{tikzcd}
\]


\begin{lemma} \label{lem:mbar}
Let $R$ be a lax symmetric monoidal partition functor. The lax symmetric monoidal structure map $m$ induces a lax symmetric monoidal structure $\bar{m}$ on $\bar{R}$ that is natural with respect to pairs of transfer maps.
\end{lemma}
\begin{proof}
Let  $\Gamma$ and $\Lambda$ be partitions. First we must produce maps
\[
\bar{m}_{\Gamma, \Lambda} \colon \bar{R}(\Gamma) \otimes \bar{R}(\Lambda) \to \bar{R}(\Gamma \sqcup \Lambda).
\]
Recall that there is an isomorphism of abelian groups
\[
R(\Gamma)/I_{\Gamma} \otimes R(\Lambda)/I_{\Lambda} \cong (R(\Gamma) \otimes R(\Lambda))/(I_{\Gamma} \otimes R(\Lambda) + R(\Gamma) \otimes I_{\Lambda}).
\]
Thus, given $m_{\Gamma, \Lambda} \colon R(\Gamma) \otimes R(\Lambda) \to R(\Gamma \sqcup \Lambda)$, to produce the desired map, we must show that the subgroup $I_{\Gamma} \otimes R(\Lambda) + R(\Gamma) \otimes I_{\Lambda}$ lands in $I_{\Gamma \sqcup \Lambda}$. But this follows from the commutative diagram (as well as the symmetry in $\Gamma$ and $\Lambda$), for any $\Phi \leq \Gamma$,
\[
\xymatrix{R(\Phi) \otimes R(\Lambda) \ar[r] \ar[d] & R(\Phi \sqcup \Lambda) \ar[d] \\ R(\Gamma) \otimes R(\Lambda) \ar[r] & R(\Gamma \sqcup \Lambda).}
\]
Thus we have an induced map
\[
\bar{m}_{\Gamma, \Lambda} \colon \bar{R}(\Gamma) \otimes \bar{R}(\Lambda) \to \bar{R}(\Gamma \sqcup \Lambda)
\]
given by
\begin{equation} \label{eq:barmdefeq}
\bar{m}_{\Gamma, \Lambda}(x \otimes y) = q_{\Gamma \sqcup \Lambda}m_{\Gamma, \Lambda}(x' \otimes y')
\end{equation}
for lifts $x'$ and $y'$ of $x$ and $y$, respectively.
Naturality with respect to pairs of transfer maps follows immediately from the compatibility of $\bar{m}$ with $m$ along transfer maps and the naturality of $m$ along pairs of transfer maps.
\end{proof}

\begin{remark}
In the above proof, we are making use of the fact that $\bar{R}$ is functorial in transfers. Recall from the beginning of \cref{sec:div} that this functor is quite degenerate. If $\Lambda \to \Omega$, is not an isomorphism then the transfer map $\bar{R}(\Lambda) \to \bar{R}(\Omega)$ is the zero map.
\end{remark}

\begin{prop} \label{prop:laxdiv}
If $R$ is a lax symmetric monoidal partition functor, then $\Div(R)$ is naturally a lax symmetric monoidal partition functor. That is, $\Div$ is an endofunctor on the category of lax symmetric monoidal partition functors.
\end{prop}


\begin{proof}
Given partitions $\Omega$ and $\Theta$, we wish to produce a natural transformation
\[
\Div(m)_{\Omega, \Theta} \colon \Div(R)(\Omega) \otimes \Div(R)(\Theta) \to \Div(R)(\Omega \sqcup \Theta).
\]

Taking the direct sum of the maps from \cref{lem:mbar} over $\Gamma \leq \Omega$ and $\Lambda \leq \Theta$ gives a map
\[
M_{\Omega, \Theta} \colon \Big(\bigoplus_{\Gamma \leq \Omega} R(\Gamma)/I_{\Gamma}\Big) \otimes \Big( \bigoplus_{\Lambda \leq \Theta} R(\Lambda)/I_{\Lambda}\Big) \to \bigoplus_{\Gamma \sqcup \Lambda \leq \Omega  \sqcup \Theta} R(\Gamma \sqcup \Lambda)/I_{\Gamma \sqcup \Lambda}
\]
defined by
\[
(M_{\Omega, \Theta}((x_{\Gamma})\otimes(y_{\Lambda})))_{\Phi \sqcup \Psi} = \bar{m}_{\Phi, \Psi}(x_{\Phi} \otimes y_{\Psi}).
\]
We check that this map is $\Sigma_{\Omega} \times \Sigma_{\Theta}$-equivariant. Given $\omega \in \Sigma_{\Omega}$ and $\theta \in \Sigma_{\Theta}$, we have
\[
(\omega \otimes \theta)((x_{\Gamma})\otimes(y_{\Lambda})) = (\bar{R}_*(\omega)x_{\omega^{-1}\Gamma}) \otimes (\bar{R}_*(\theta)y_{\theta^{-1}\Lambda}).
\]
To check equivariance, we must understand
\[
(\bar{m}_{\Gamma,\Lambda}(\bar{R}_*(\omega)x_{\omega^{-1}\Gamma} \otimes \bar{R}_*(\theta)y_{\theta^{-1}\Lambda}))_{\Gamma \sqcup \Lambda}.
\]
But the naturality of $m$ implies that 
\[
\bar{m}_{\Gamma,\Lambda}(\bar{R}_*(\omega)x_{\omega^{-1}\Gamma} \otimes \bar{R}_*(\theta)y_{\theta^{-1}\Lambda}) = \bar{R}_*(\omega \sqcup \theta)\bar{m}_{\omega^{-1}\Gamma, \theta^{-1}\Lambda}(x_{\omega^{-1}\Gamma} \otimes y_{\theta^{-1}\Lambda}),
\]
and this implies that $M_{\Omega,\Theta}$ is equivariant. Taking $\Sigma_{\Omega} \times \Sigma_{\Theta}$-fixed points and precomposing with the canonical map 
\[
\Div(R)(\Omega) \otimes \Div(R)(\Theta) \to \bigg ( \bigoplus_{\Gamma \leq \Omega} R(\Gamma)/I_{\Gamma} \otimes \bigoplus_{\Lambda \leq \Theta} R(\Lambda)/I_{\Lambda} \bigg )^{\Sigma_{\Omega} \times \Sigma_{\Theta}}
\]
gives $\Div(m)_{\Omega,\Theta}$.

Now we must check the naturality of $\Div(m)$ with respect to pairs of transfer maps and pairs of restriction maps. Naturality with respect to pairs of restriction maps is straight-forward. For pairs of transfer maps we must unwind the definitions: 

Let $f \colon \Gamma \to \Omega$ and $g \colon \Lambda \to \Theta$ be maps of partitions. We wish to show that the diagram
\[
\xymatrix{\Div(R)(\Gamma) \otimes \Div(R)(\Lambda) \ar[rr]^-{\Div(m)_{\Gamma,\Lambda}} \ar[d]_{\Div(R)_*(f)\otimes \Div(R)_*(g)} && \Div(R)(\Gamma \sqcup \Lambda) \ar[d] \\ \Div(R)(\Omega) \otimes \Div(R)(\Theta) \ar[rr]^-{\Div(m)_{\Omega,\Theta}} && \Div(R)(\Omega \sqcup \Theta)}
\]
commutes. Let $x \in \Div(R)(\Gamma)$, $y \in \Div(R)(\Lambda)$, $\Delta \leq \Omega$, and $\Psi \leq \Theta$. Following $x \otimes y$ around the bottom of the diagram and evaluating at $\Delta \sqcup \Psi$ gives
\begin{equation*}
\begin{split}
& \bar{m}_{\Delta,\Psi}\Big ( \Big (\sum_{\substack{[\sigma] \in \Sigma_{\Delta} \backslash \Sigma_{\Omega} / \Sigma_{f\Gamma} \\ \Delta \leq \sigma f \Gamma}} \bar{R}_*(\sigma f)(x_{(\sigma f)^{-1} \Delta}) \Big ) \otimes \Big (\sum_{\substack{[\sigma'] \in \Sigma_{\Psi} \backslash \Sigma_{\Theta} / \Sigma_{g\Lambda} \\ \Psi \leq \sigma' g \Lambda}} \bar{R}_*(\sigma' g)(y_{(\sigma' g)^{-1} \Psi}) \Big ) \Big )  \\ &=  \sum_{\substack{[\sigma] \in \Sigma_{\Delta} \backslash \Sigma_{\Omega} / \Sigma_{f\Gamma} \\ [\sigma'] \in \Sigma_{\Psi} \backslash \Sigma_{\Theta} / \Sigma_{g\Lambda} \\ \Delta \leq \sigma f \Gamma, \Psi \leq \sigma' g \Lambda}}  \bar{m}_{\Delta,\Psi} \Big ( \bar{R}_*(\sigma f)(x_{(\sigma f)^{-1} \Delta}) \otimes  \bar{R}_*(\sigma' g)(y_{(\sigma' g)^{-1} \Psi})\Big). 
\end{split}
\end{equation*}
On the other hand, following $x \otimes y$ around the top of the diagram and evaluating at $\Delta \sqcup \Psi$ gives
\[
\sum_{\substack{[\sigma \times \sigma'] \in \Sigma_{\Delta \sqcup \Psi} \backslash \Sigma_{\Omega \sqcup \Theta} / \Sigma_{f\Gamma \sqcup g\Lambda} \\ \Delta \sqcup \Psi \leq \sigma f \Gamma \sqcup \sigma' g \Lambda}}  \bar{R}_*(\sigma f \sqcup \sigma' g) \bar{m}_{(\sigma f)^{-1} \Delta, (\sigma' g)^{-1} \Psi}(x_{(\sigma f)^{-1} \Delta} \otimes y_{(\sigma' g)^{-1} \Psi}).
\]
These agree as there is a canonical bijection
\[
\Sigma_{\Delta \sqcup \Psi} \backslash \Sigma_{\Omega \sqcup \Theta} / \Sigma_{f\Gamma \sqcup g\Lambda} \cong \Sigma_{\Delta} \backslash \Sigma_{\Omega} / \Sigma_{f\Gamma} \times \Sigma_{\Psi} \backslash \Sigma_{\Theta} / \Sigma_{g\Lambda}
\]
and $\bar{m}$ is natural in pairs of transfer maps by \cref{lem:mbar}. 

Using these formulas, it is straightforward, though tedious, to finish checking that $\Div(R)$ together with $\Div(m)$ is lax symmetric monoidal.
\end{proof}

\begin{prop}
The endofunctor $\Div$ is a monad on the category of lax symmetric monoidal partition functors.
\end{prop}
\begin{proof}
We must show that if $R$ is a lax symmetric monoidal partition functor, then the natural transformations $\eta \colon R \to \Div(R)$ and $\mu \colon \Div(\Div(R)) \to \Div(R)$ are compatible with the lax symmetric monoidal structure.

We begin with $\eta$. We must show that
\[
\xymatrix{R(\Lambda) \otimes R(\Omega) \ar[rr]^-{\eta_\Lambda \otimes \eta_\Omega} \ar[d]_{m_{\Lambda,\Omega}} && \Div(R)(\Lambda) \otimes \Div(R)(\Omega) \ar[d]^{\Div(m)_{\Lambda,\Omega}} \\ R(\Lambda \sqcup \Omega) \ar[rr]^-{\eta_{\Lambda \sqcup \Omega}} && \Div(R)(\Lambda \sqcup \Omega)}
\]
commutes. Going around the top gives
\[
(\bar{m}_{\Gamma, \Theta}(q_{\Gamma}R^*(\Gamma \leq \Lambda)(x) \otimes q_{\Theta} R^*(\Theta \leq \Omega)(y)))_{\Gamma \sqcup \Theta \leq \Lambda \sqcup \Omega}.
\]
Going around the bottom gives
\[
(q_{\Gamma \sqcup \Theta} m_{\Gamma, \Theta}(R^*(\Gamma \leq \Lambda)x \otimes R^*(\Theta \leq \Omega)y))_{\Gamma \sqcup \Theta \leq \Lambda \sqcup \Omega}.
\]
These agree by \cref{eq:barmdefeq}

To see that $\mu$ is lax symmetric monoidal, we need to check that a similar diagram commutes. But this follows by taking the direct sum of the commutative diagrams
\[
\xymatrix{\Div(R)(\Lambda) \otimes \Div(R)(\Omega) \ar[r] \ar[d] & \bar{R}(\Lambda) \otimes \bar{R}(\Omega) \ar[d] \\ \Div(R)(\Lambda \sqcup \Omega) \ar[r] & \bar{R}(\Lambda \sqcup \Omega),}
\]
where the horizontal arrows are induced by projection onto the summand.

In both cases symmetry is straightforward.
\end{proof}

\section{Partition rings} \label{sec:partrings}

\begin{definition}
We say that a lax symmetric monoidal partition functor $R$ is a partition ring if 
\begin{enumerate}
    \item[(a)] $R$ takes values in commutative rings
    \item[(b)] restriction maps are maps of commutative rings
    \item[(c)] for any pair of partitions $\Gamma$, $\Lambda$, $m_{\Gamma, \Lambda}$ is a map of commutative rings
    \item[(d)] $R$ satisfies Frobenius reciprocity.
\end{enumerate}
A map of partition rings is required to preserve all of this structure.
\end{definition}

The Frobenius reciprocity condition means that the transfer map 
\[
R_*(\Gamma \leq \Lambda) \colon R(\Gamma) \to R(\Lambda)
\]
is a map of $R(\Lambda)$-modules for the $R(\Lambda)$-module structure on $R(\Gamma)$ through $R^*(\Gamma \leq \Lambda)$. This implies that $I_{\Lambda} \subseteq R(\Lambda)$ is an ideal. We will refer to it as the transfer ideal. It follows that $q_{\Lambda} \colon R(\Lambda) \to \bar{R}(\Lambda)$ is a ring map. Further, since the lax symmetric monoidal structure map $m$ for $R$ is a map of rings and $\bar{R}(\Lambda)$ is a ring, it follows that 
\[
\bar{m}_{\Lambda, \Omega} \colon \bar{R}(\Lambda) \otimes \bar{R}(\Omega) \to \bar{R}(\Lambda \sqcup \Omega)
\]
is a ring map.

\begin{prop}
If $R$ is a partition ring then $\Div(R)$ is a partition ring. Further, $\Div$ is a monad on the category of partition rings.
\end{prop}
\begin{proof}
First, note that $\Div(R)$ takes values in commutative rings due to the fact that $R$ takes values in commutative rings, $I_{\Lambda} \subseteq R(\Lambda)$ is an ideal, and $\Sigma_{\Lambda}$ acts on $\bigoplus_{\Gamma \leq \Lambda} \bar{R}(\Gamma)$ by ring automorphisms. 

The restriction maps for $\Div(R)$ are induced by projection onto a summand, so are ring maps. 

Further, $\eta$ and $\mu$ are level-wise maps of commutative rings. Since $\bar{m}_{\Gamma,\Lambda}$ is a map of commutative rings for all pairs of partitions, it follows that $\Div(m)_{\Gamma,\Lambda}$ is a map of commutative rings. 

Finally, we must check Frobenius reciprocity for $\Div(R)$. Let $f \colon \Lambda \to \Omega$ be a map of partitions, let $r \in \Div(R)(\Omega)$, and let $x \in \Div(R)(\Lambda)$. For $\Gamma \leq \Omega$, we have
\[
(r\Div(R)_*(f)(x))_\Gamma = \sum_{\substack{[\sigma] \in \Sigma_{\Gamma} \backslash \Sigma_{\Omega} / \Sigma_{f \Lambda} \\ \Gamma \leq \sigma f \Lambda}} r_{\Gamma} \bar{R}_*(\sigma f) x_{(\sigma f)^{-1} \Gamma}.
\]
On the other hand,
\begin{equation*}
\begin{split}
(\Div(R)_*(f)(\Div(R)^*(f)(r)x))_\Gamma &= \sum_{\substack{[\sigma] \in \Sigma_{\Gamma} \backslash \Sigma_{\Omega} / \Sigma_{f \Lambda} \\ \Gamma \leq \sigma f \Lambda}} \bar{R}_*(\sigma f) (\Div(R)^*(f)(r)x))_{(\sigma f)^{-1} \Gamma} \\
&= \sum_{\substack{[\sigma] \in \Sigma_{\Gamma} \backslash \Sigma_{\Omega} / \Sigma_{f \Lambda} \\ \Gamma \leq \sigma f \Lambda}} \bar{R}_*(\sigma f) \big ( \bar{R}^*(f_{(\sigma f)^{-1}\Gamma}) (r_{\sigma^{-1}\Gamma}) x_{(\sigma f)^{-1} \Gamma} \big ) \\
&= \sum_{\substack{[\sigma] \in \Sigma_{\Gamma} \backslash \Sigma_{\Omega} / \Sigma_{f \Lambda} \\ \Gamma \leq \sigma f \Lambda}} \bar{R}_*(\sigma f)(\bar{R}_*((f_{(\sigma f)^{-1}\Gamma})^{-1}) (r_{\sigma^{-1}\Gamma})) \bar{R}_*(\sigma f)(x_{(\sigma f)^{-1} \Gamma}) \\
&= \sum_{\substack{[\sigma] \in \Sigma_{\Gamma} \backslash \Sigma_{\Omega} / \Sigma_{f \Lambda} \\ \Gamma \leq \sigma f \Lambda}} \bar{R}_*(\sigma) (r_{\sigma^{-1}\Gamma}) \bar{R}_*(\sigma f)(x_{(\sigma f)^{-1} \Gamma}).
\end{split}
\end{equation*}
Now since $\sigma \in \Sigma_{\Omega}$ and $r \in \Div(R)(\Omega)$ is $\Sigma_{\Omega}$-invariant, we have 
\[
\bar{R}_*(\sigma) (r_{\sigma^{-1}\Gamma}) = r_{\Gamma}.
\]
\end{proof}

Recall that a global Green functor $S$ is a global functor that takes values in commutative rings and for which the restriction maps are ring maps and the transfer maps satisfy Frobenius reciprocity.  Global Green functors inherit a lax symmetric monoidal product
\[
m_{G,H} \colon S(G) \otimes S(H) \to S(G \times H)
\]
from this structure that is natural in pairs of restriction maps as well as pairs of transfer maps. The map $m_{G,H}$ is constructed using the universal property of the tensor product of commutative rings and the restrictions along the projections $G \times H \to G$ and $G \times H \to H$. It is easy to see that $m_{G,H}$ is natural in pairs of restriction maps. Pairs of transfer maps take a bit more work, it is worth describing a proof since there is a peculiar interaction between transfers and restrictions.

The double coset formula for global Green functors implies that, given injective maps $f \colon H \to H'$ and $g \colon G \to G'$, the homotopy pullback squares of groupoids
\begin{equation} \label{eq:homotopypbstrange}
    \begin{tikzcd}
        * \mmod (H \times G) 
            \ar[r]
            \ar[d] 
            & 
        * \mmod H 
            \ar[d] 
            & 
        * \mmod (H'\times G) 
            \ar[r]
            \ar[d] 
            &
        * \mmod G 
            \ar[d] 
            \\
        * \mmod (H' \times G) 
            \ar[r] 
            & 
        * \mmod H' 
            & 
        * \mmod (H' \times G') 
            \ar[r] 
            & 
        * \mmod G'
    \end{tikzcd}
\end{equation}
induce commutative diagrams 
\[
\xymatrix{S(H \times G) \ar[d]_{S_*(f \times \id)} & S(H) \otimes S(G) \ar[l]_{m_{H,G}} \ar[d]_{S_*(f)\otimes S_*(\id)} & S(H' \times G) \ar[d]_{S_*(\id \times g)} & S(H') \otimes S(G) \ar[l]_{m_{H', G}} \ar[d]^{S_*(\id) \otimes S_*(g)} \\ S(H' \times G) & S(H') \otimes S(G) \ar[l]^{m_{H',G}} & S(H' \times G') & S(H') \otimes S(G'). \ar[l]^{m_{H',G'}}}
\]
The horizontal arrows are base changed from the restriction maps using the $S(G)$-module  (respectively $S(H')$-module) structure on the target from Frobenius reciprocity. Composing these diagrams yields the fact that the monoidal structure on $S$ is natural in pairs of transfers.

\begin{prop} \label{globalgreentopartition}
Assume that $S$ is a global Green functor and let $G$ be a finite group, then 
\[
S \circ (G \times \Sigma_{(-)}) \text{ and } S \circ (G \wr \Sigma_{(-)})
\]
are both partition rings.
\end{prop}
\begin{proof}
The only aspect of a partition ring that doesn't follow immediately from the definition of a global Green functor is the lax symmetric monoidal structure. 

Let $R = S \circ (G \times \Sigma_{(-)})$. Given partitions $\Lambda$ and $\Omega$, let 
\[
\alpha_\Lambda \colon G \times \Sigma_{\Lambda \sqcup \Omega} \cong G \times \Sigma_{\Lambda} \times \Sigma_{\Omega} \to G \times \Sigma_{\Lambda} \times G \times \Sigma_{\Omega} \to G \times \Sigma_{\Lambda}, 
\]
where the second map is induced by the diagonal on $G$ and the last map is the projection onto  $G \times \Sigma_{\Lambda}$. The universal property of the tensor product of commutative rings provides us with a map
\[
m_{\Lambda, \Omega} \colon R(\Lambda) \otimes R(\Omega) \xrightarrow{\res_{\alpha_{\Lambda}} \otimes \res_{\alpha_{\Omega}}} R(\Lambda \sqcup \Omega).
\]
The case of the wreath product is similar, but simpler, as there is an isomorphism
\[
G \wr \Sigma_{\Lambda \sqcup \Omega} \cong G \wr \Sigma_{\Lambda} \times G \wr \Sigma_{\Omega}.
\]
The map $m_{\Lambda, \Omega}$ factors through
\[
R(\Lambda) \otimes R(\Omega) \to S(G \times \Sigma_{\Lambda} \times G \times \Sigma_{\Omega}).
\]
which is natural in pairs of transfer maps by our discussion above. To check that the map 
\[
S(G \times \Sigma_{\Lambda} \times G \times \Sigma_{\Omega}) \to R(\Lambda \sqcup \Omega)
\]
is natural in pairs of transfer maps, assume that $f \colon \Lambda \to \Theta$ and $g \colon \Omega \to \Gamma$ are maps of partitions. The naturality then follows from the homotopy pullback square of groupoids
\[
\begin{tikzcd}
* \mmod (G \times \Sigma_{\Lambda} \times \Sigma_{\Omega}) \ar[r] \ar[d] & * \mmod (G \times \Sigma_{\Lambda} \times G \times \Sigma_{\Omega}) \ar[d] \\ * \mmod (G \times \Sigma_{\Theta} \times \Sigma_{\Gamma}) \ar[r] & * \mmod (G \times \Sigma_{\Theta} \times G \times \Sigma_{\Gamma}).
\end{tikzcd}
\]
\end{proof}

\cref{prop:diagonal} now gives us:
\begin{cor} \label{cor:diagonalGreen}
If $S$ is a global Green functor, then restriction along the diagonal induces a canonical map of partition rings
\[
S \circ (G \wr \Sigma_{(-)}) \to S \circ (G \times \Sigma_{(-)}).
\]
\end{cor}

\begin{example}
Let $K$ be a commutative ring. Consider the constant partition functor $C_K$. Since $K$ is a commutative ring, we define 
\[
m_{\Lambda, \Omega} \colon K \otimes K \to K
\]
to be the multiplication map. This is natural in pairs of restriction maps since restrictions are just the identity map. Now let $f \colon \Lambda \to \Omega$ and $g \colon \Gamma \to \Theta$ be maps of partitions. The multiplication is natural in pairs of transfer maps as
\[
|\Sigma_{\Omega}/\Sigma_{f \Lambda}||\Sigma_{\Theta}/\Sigma_{g \Gamma}| = |(\Sigma_{\Omega} \times \Sigma_{\Theta}) / ( \Sigma_{f \Lambda} \times \Sigma_{g \Gamma})| = |\Sigma_{\Omega \sqcup \Theta}/\Sigma_{(f \sqcup g) \Lambda \sqcup \Gamma}|.
\]
Clearly $C_K$ is commutative ring valued and restriction maps are maps of commutative rings (they are the identity). Frobenius reciprocity follows from the fact that transfers are given by multiplication by an integer. It follows that $C_K$ is a partition ring. 

We could have also just applied \cref{globalgreentopartition}, but we proved that $C_K$ is a partition ring to contrast it with $C^K$, which does not have a global avatar (unless we only consider injective group homomorphisms). In this case, the proof in the previous paragraph does imply that $C^K$ is a lax symmetric monoidal partition functor. The same multiplication $m$ is used, but the role of restriction and transfer are swapped in the proof that $m$ is natural. Further, $C^K$ does take values in commutative rings, but the restriction maps are not generally maps of commutative rings and Frobenius reciprocity does not make sense.
\end{example}

\begin{example} \label{ex:Greentriv}
Let $S$ be a global Green functor and let $K = S(G)$. Since restriction maps are ring maps, \cref{prop:globaltriv} implies that the map of partition functors $S \circ (G \times \Sigma_{(-)}) \to C_K$ induced by restriction along the map $G \to G \times \Sigma_{\Lambda}$ is a map of partition rings.
\end{example}

\section{Symmetric monoidal partition functors and algebra} \label{sec:symmon}

In this section, we provide a flatness condition on a symmetric monoidal partition functor $R$ so that $\Div(R)$ is a symmetric monoidal partition functor. We also compare the categories of symmetric monoidal partition functors and symmetric monoidal partition rings with categories of classical algebraic objects.

Assume $R$ is a lax symmetric monoidal partition functor with multiplication $m$. Following \cref{not:indexing}, we write $R(m) = R(\u m)$. Since $\u{0}$ is the unit for the symmetric monoidal structure on $\cP$, $R(0)$ admits the structure of a commutative ring. Explicitly, $R(0)$ admits the structure of a commutative ring via the isomorphism
\[
\u{0} \sqcup \u{0} \xrightarrow{\cong} \u{0},
\]
which induces the multiplication
\[
R(0) \otimes R(0) \xrightarrow{m_{0,0}} R(\u{0} \sqcup \u{0}) \xrightarrow{\cong} R(0).
\]
Similarly, given a partition $\Omega$, $R(\Omega)$ admits the structure of an $R(0)$-module via the isomorphism
\[
\u{0} \sqcup \Omega \xrightarrow{\cong} \Omega.
\]
Given a map of partitions $f \colon \Lambda \to \Omega$, $R_*(f)$ is a map of $R(0)$-modules. Notice that this implies that the transfer subgroup $I_{\Lambda} \subseteq R(\Lambda)$ is an
$R(0)$-submodule.

These observations together with the relations satisfied by $m$ imply that $m$ induces a lax symmetric monoidal map
\[
\mathcal{M}_{\Lambda,\Omega} \colon R(\Lambda) \otimes_{R(0)} R(\Omega) \to R(\Lambda \sqcup \Omega).
\]
That is, $R$ can be viewed as a lax symmetric monoidal partition functor taking values in $R(0)$-modules with $\otimes_{R(0)}$ (rather than abelian groups and $\otimes_{\Z}$). If $R$ has the structure of a partition ring, then the Eckmann--Hilton argument implies that the two commutative ring structures on $R(0)$ agree.

Turning to $\Div(R)$, note that $\Div(R)(0) = R(0)$. It follows from \cref{prop:laxdiv}, if $R$ is a lax symmetric monoidal partition functor with values in $R(0)$-modules, then $\Div(R)$ is a lax symmetric monoidal partition functor with values in $R(0)$-modules. Further, since $\Div(R)$ is lax symmetric monoidal and both $\eta$ and $\mu$ are natural transformations of lax symmetric monoidal partition functors, they are also natural transformations of lax symmetric monoidal partition functors with values in $R(0)$-modules. We have essentially noticed the following:

\begin{prop}
Let $K$ be a commutative ring. The endofunctor $\Div$ is a monad on the category of lax symmetric monoidal partition functors taking values in $K$-modules.
\end{prop}

Given a partition functor $R$, let 
\[
R[\Sigma] = \bigoplus_{m \geq 0} R(m).
\]

\begin{lemma}
If $R$ is lax symmetric monoidal, then $R[\Sigma]$ admits a natural commutative $\N$-graded ring structure.
\end{lemma}
\begin{proof}
Recall that the lax symmetric monoidal structure gives $R(0)$ the structure of a commutative ring. Given homogeneous elements $x \in R(i)$ and $y \in R(j)$, we define the product by
\[
xy = R_*(\underline{i} \sqcup \underline{j} \leq \underline{i+j})m_{i, j}(x \otimes y).
\]

The unit for this multiplication is given by the unit in $R(0)$. Associativity follows from the associativity of the external product, naturality, and transitivity of transfers. Commutativity follows from the fact that $m$ is lax symmetric monoidal, the commutative diagram of partitions
\[
\xymatrix{\underline{i} \sqcup \underline{j} \ar[r] \ar[d] & \underline{j} \sqcup \underline{i} \ar[d] \\ \underline{i+j} \ar[r]^{\sigma} & \underline{j+i},}
\]
where $\sigma \in \Sigma_{i+j}$ is the automorphism trading $\u i$ and $\u j$,
and the fact that $\sigma$ is sent to the identity by $R$. 
\end{proof}

We will refer to the product on $R[\Sigma]$ as the transfer product. The following result is immediate:

\begin{prop}
The construction $R[\Sigma]$ is natural in the lax symmetric monoidal partition functor $R$. That is, $(-)[\Sigma]$ is a functor from the category of lax symmetric monoidal partition functors to the category of commutative $\N$-graded rings.
\end{prop}

Fix a commutative ring $K$. We will turn our attention to symmetric monoidal
(rather than lax symmetric monoidal) partition functors with values in
$K$-modules. For such a partition functor $R$, we have a specified isomorphism
$R(0) \cong K$ and the multiplication map
\[
\cM_{\Lambda,\Omega} \colon
R(\Lambda) \otimes_K R(\Omega)
\xrightarrow{\cong}
R(\Lambda \sqcup \Omega)
\]
is an isomorphism. The functor $(-)[\Sigma]$ lands in commutative $\N$-graded $K$-algebras. The category of symmetric monoidal partition functors is a full subcategory of the category of lax symmetric monoidal partition functors.

If R is a symmetric monoidal partition functor, then the quotients $\bar{R}(\Lambda)$ continue to
satisfy a strong symmetric monoidal property.

\begin{lemma} \label{lem:symmon}
Assume that $R$ is a symmetric monoidal partition functor with values in
$K$-modules. The multiplication map induces an isomorphism
\[
\bar{\cM}_{\Lambda,\Omega} \colon
\bar{R}(\Lambda) \otimes_K \bar{R}(\Omega)
\overset{\cong}{\longrightarrow}
\bar{R}(\Lambda \sqcup \Omega).
\]
\end{lemma}

\begin{proof}
Every refinement of $\Lambda \sqcup \Omega$ is uniquely of the form
$\Gamma \sqcup \Theta$ for refinements $\Gamma \leq \Lambda$ and
$\Theta \leq \Omega$. Such a refinement is proper if and only if at least one
of $\Gamma < \Lambda$ or $\Theta < \Omega$ is proper. Naturality of the
symmetric monoidal structure with respect to transfers implies that, under the
isomorphism
\[
\cM_{\Lambda,\Omega} \colon
R(\Lambda)\otimes_K R(\Omega)
\overset{\cong}{\longrightarrow}
R(\Lambda\sqcup\Omega),
\]
the transfer subgroup $I_{\Lambda\sqcup\Omega}$ corresponds to the sum of the
images of
\[
I_\Lambda\otimes_K R(\Omega)
\longrightarrow
R(\Lambda)\otimes_K R(\Omega)
\]
and
\[
R(\Lambda)\otimes_K I_\Omega
\longrightarrow
R(\Lambda)\otimes_K R(\Omega).
\]
It follows that $\cM_{\Lambda,\Omega}$ induces an isomorphism
\[
\bar{R}(\Lambda)\otimes_K\bar{R}(\Omega)
\overset{\cong}{\longrightarrow}
\bar{R}(\Lambda\sqcup\Omega).
\]
\end{proof}

We would like to better understand $\Div(R)$, when $R$ is a symmetric monoidal partition functor. \cref{weylgroupfixed} implies that
\[
\Div(R)(m) \cong \bigoplus_{[\Lambda] \in \{\Lambda \leq \u m\}/\Sigma_m} \bar{R}(\Lambda)^{W_{\Sigma_m}(\Sigma_{\Lambda})}
\]
and \cref{lem:symmon} implies that $\bar{R}(\Lambda) \cong \bigotimes_{B \in [m]/\sim_\Lambda} \bar{R}(\u B)$, where the tensor product is over $K$. These decompositions are best described in terms of integer partitions. 

An integer partition $\lambda$ is a sequence of natural numbers $(\lambda_1, \lambda_2,\ldots)$ with all but finitely many $\lambda_i$ equal to zero. We say that $\lambda$ is an integer partition of $m$, written $\lambda \parts m$ if
\[
m = \sum_{i \geq 1} i \cdot \lambda_i.
\]
The natural number $\lambda_i$ denotes the number of times that $i$ occurs in the integer partition of $m$.

For $\lambda \parts m$, let $\Sigma_{\lambda} = \prod_i (\Sigma_i)^{\lambda_i}$. Viewed as a subgroup of $\Sigma_m$, we have
\[
W_{\Sigma_m}(\Sigma_\lambda) \cong \prod_{i} \Sigma_{\lambda_i}.
\]
There is a canonical bijection between the set $\{\Lambda \leq [m]\}/\Sigma_m$ and the set of integer partitions of $m$ given by sending a partition of $[m]$ to the integer partition in which $\lambda_i$ denotes the number of blocks of $[m]/{\sim_{\Lambda}}$ of cardinality $i$. 

Together, these facts give an isomorphism
\begin{equation} \label{eq:symmondecomp}
\Div(R)(m)
\cong
\bigoplus_{\lambda\parts m}
\left(
\bigotimes_i \bar{R}(i)^{\otimes\lambda_i}
\right)^{\prod_i\Sigma_{\lambda_i}}.
\end{equation}

It does not follow from \cref{lem:symmon} that $\Div(R)$ is symmetric
monoidal. The definition of $\Div$ involves taking fixed points and fixed
points do not in general commute with tensor products. There is nevertheless
a useful flatness condition under which they do commute in the situation at
hand.

\begin{prop} \label{prop:flatdivsymmon}
Let $R$ be a symmetric monoidal partition functor with values in $K$-modules.
If $\bar{R}(m)$ is a flat $K$-module for every $m\geq 0$, then $\Div(R)$ is
a symmetric monoidal partition functor with values in $K$-modules.
\end{prop}

\begin{proof}

We will use two standard facts about flat modules. First, if a finite group
$G$ acts on a $K$-module $A$, acts trivially on $B$, and $B$ is flat over $K$, then
\[
A^G\otimes_K B \overset{\cong}{\longrightarrow} (A\otimes_K B)^G.
\]
Second, if $A$ is a flat $K$-module, then
\[
TS_r(A) = (A^{\otimes r})^{\Sigma_r},
\]
the module of symmetric invariant tensors, is flat over $K$. By a theorem of Lazard \cite[Theorem 10.81.4]{stacks-project}, $A$ is a filtered colimit of finite free
$K$-modules. Tensor powers and finite-group fixed points commute with filtered
colimits, so $TS_r(A)$ is a filtered colimit of the modules $TS_r(F)$
for finite free $F$. Each $TS_r(F)$ is free, and hence $TS_r(A)$ is flat.


These facts show that $\Div(R)(m)$ is flat: By \cref{eq:symmondecomp}, we have an isomorphism
\[
\Div(R)(m)
\cong
\bigoplus_{\lambda\parts m}
\bigg (
\bigotimes_i \bar{R}(i)^{\otimes\lambda_i}
\bigg)^{\prod_i\Sigma_{\lambda_i}}.
\]
Since $\bar{R}(i)^{\otimes\lambda_i}$ is flat and $\left(\bar{R}(i)^{\otimes\lambda_i}\right)^{\Sigma_{\lambda_i}}$ is flat, we have
\[
\bigg(
\bigotimes_i \bar{R}(i)^{\otimes\lambda_i}
\bigg)^{\prod_i\Sigma_{\lambda_i}}
\cong
\bigotimes_i \Big (\bar{R}(i)^{\otimes\lambda_i}\Big )^{\Sigma_{\lambda_i}}.
\]
Thus $\Div(R)(m)$ is flat.

To conclude, consider $\u m \sqcup \u n$. We first apply $\Sigma_m \times e$ fixed points to the isomorphism from \cref{lem:symmon}
\[
\bigg(\bigoplus_{\Gamma\leq \u m}\bar{R}(\Gamma)\bigg) \otimes_K \bigg(\bigoplus_{\Lambda \leq \u n}\bar{R}(\Lambda)\bigg) \xrightarrow{\cong} \bigoplus_{\Omega \leq \u m \sqcup \u n}\bar{R}(\Omega).
\]
As the second tensor factor is flat, we get an isomorphism
\[
\Div(R)(m) \otimes_K \bigg( \bigoplus_{\Lambda \leq \u n}\bar{R}(\Lambda) \bigg )\xrightarrow{\cong} \bigg (\bigoplus_{\Omega \leq \u m \sqcup \u n}\bar{R}(\Omega) \bigg)^{\Sigma_m \times e}.
\]
Now, since $\Div(R)(m)$ is flat, we may apply $e \times \Sigma_n$ fixed points to produce our desired isomorphism 
\[
\Div(R)(m) \otimes_K \Div(R)(n) \cong \Div(R)(\u m \sqcup \u n).
\]
\end{proof}

\begin{remark}
The flatness hypothesis in \cref{prop:flatdivsymmon} cannot be omitted (see \cref{ex:PDnotsymmetric}). Thus $\Div$ does not restrict to an endofunctor
of the category of symmetric monoidal partition functors over an arbitrary
commutative ring. Its natural home is the larger category of lax symmetric
monoidal partition functors. On the other hand,
\cref{prop:flatdivsymmon} implies that $\Div(R)$ is symmetric monoidal whenever
the quotients $\bar{R}(m)$ are flat over $R(0)$. In particular,
this condition is automatic for partition functors valued in vector spaces
over a field.
\end{remark}

Let $R$ be a symmetric monoidal partition functor taking values in $K$-modules. Then for a partition $\Gamma = (X,\sim_{\Gamma})$, we have a canonical isomorphism of $K$-modules
\[
R(\Gamma) \cong \bigotimes_{B \in X/{\sim_{\Gamma}}} R(\u B),
\]
where the tensor product is taken over $K$. Thus we may consider an element in $R(\Gamma)$ as a sum of simple tensors of the form $\otimes_{B} x_{B}$. The values of $R$ on the indiscrete partitions determine the partition functor and the value of $R$ on two trivial partitions of the same cardinality are canonically isomorphic. Given an isomorphism $\sigma \colon \Gamma \to \Omega$, we have an induced map
\[
R_*(\sigma) \colon \bigotimes_{B \in X/\sim_{\Gamma}} R(\u B) \to \bigotimes_{B' \in X/\sim_{\Omega}} R(\u B').
\]
Since \cref{lem:earlyiso} implies that $R(B)$ and $R(B')$ are canonically identified when $|B| = |B'|$, for $x = \otimes_{B} x_{B} \in R(\Gamma)$, this map is given by
\begin{equation} \label{eq:permutetensors}
(R_*(\sigma) x)_{B'} = x_{\sigma^{-1}B'}.
\end{equation}
That is, $R_*(\sigma)$ permutes the components of simple tensors. This suggests that the category of symmetric monoidal partition functors admits an interpretation in classical algebra.

Consider the category of $\N$-graded connected commutative cocommutative $K$-bialgebras. We will often use the notation $A_{*} = \bigoplus_{i \in \N} A_i$ with $A_0 = K$ for an object in this category. Such an object comes equipped with a graded multiplication
\[
\nabla \colon A_* \otimes_K A_* \to A_*,
\]
and a graded comultiplication
\[
\Delta \colon A_* \to A_* \otimes_K A_*,
\]
as well as unit and counit maps
\[
\eta \colon K \to A_* \text{ and } \epsilon \colon A_* \to K,
\]
all of which are maps of $K$-algebras. These maps must satisfy the identities making $A_*$ into a cocommutative comonoid in the category of $\N$-graded commutative $K$-algebras.

\begin{prop} \label{prop:algequiv}
There is an equivalence of categories between the category of symmetric monoidal partition functors valued in $K$-modules and the category of connected commutative cocommutative $\N$-graded $K$-bialgebras given by the functor sending a symmetric monoidal partition functor $R$ to $R[\Sigma]$.
\end{prop}
\begin{proof}[Proof sketch]
We describe an inverse functor. Let
\[
A_*=\bigoplus_{m\geq 0}A_m
\]
be a connected commutative cocommutative $\N$-graded $K$-bialgebra with multiplication $\nabla$ and comultiplication $\Delta$ as above. 

We construct a symmetric monoidal partition functor $R_A$. First set
\[
R_A(m)=A_m.
\]
If $\Lambda=(X,\sim_\Lambda)$ is an arbitrary partition, define
\[
R_A(\Lambda)
=
\bigotimes_{B\in X/{\sim_\Lambda}} A_{|B|},
\]
where the tensor product is taken over $K$. Since $A_*$ is graded, this is naturally
the tensor product of the homogeneous pieces whose degrees are the sizes of the blocks of
$\Lambda$. The symmetric monoidal structure is the evident identification
\[
R_A(\Lambda)\otimes_K R_A(\Omega)
\cong
R_A(\Lambda\sqcup\Omega).
\]

We now describe restrictions and transfers. Let
\[
f\colon \Lambda\to \Omega
\]
be a map of partitions. For a block $D\in Y/{\sim_\Omega}$, the preimage
$f^{-1}(D)$ is a union of blocks of $\Lambda$. Say these blocks are
\[
B_1,\ldots,B_r.
\]
The restriction map
\[
R_A^*(f)\colon R_A(\Omega)\to R_A(\Lambda)
\]
is defined on the tensor factor $A_{|D|}$ by the iterated comultiplication
\[
A_{|D|}
\longrightarrow
A_{|B_1|}\otimes_K\cdots\otimes_K A_{|B_r|},
\]
and then one tensors these maps over all blocks $D$ of $\Omega$. Cocommutativity of
$\Delta$ makes this independent of the ordering of the blocks $B_i$.

Similarly, the transfer map
\[
R_{A,*}(f)\colon R_A(\Lambda)\to R_A(\Omega)
\]
is defined on the tensor factors indexed by the blocks $B_1,\ldots,B_r$ in the decomposition of $f^{-1}(D)$
by the iterated multiplication
\[
A_{|B_1|}\otimes_K\cdots\otimes_K A_{|B_r|}
\longrightarrow
A_{|D|},
\]
and then one tensors these maps over all blocks $D$ of $\Omega$. Commutativity of
$\nabla$ makes this independent of the ordering of the blocks $B_i$.

Functoriality of restrictions follows from coassociativity of $\Delta$, and functoriality of
transfers follows from associativity of $\nabla$. The compatibility with isomorphisms of
partitions is the symmetry isomorphism for tensor products. The double coset formula
is the compatibility between multiplication and comultiplication in a bialgebra: after
identifying the tensor factors by the common refinement of two partitions, both sides are
obtained by first splitting according to intersections of blocks and then multiplying the
pieces back together in the prescribed order. Commutativity and cocommutativity remove the
dependence on choices of order. Thus $R_A$ is a symmetric monoidal partition functor.

This construction is functorial in $A_*$. If $\phi\colon A_*\to (A')_*$ is a map of graded
bialgebras, then
\[
R_A(\Lambda)=\bigotimes_{B\in X/\sim_\Lambda}A_{|B|}
\longrightarrow
\bigotimes_{B\in X/\sim_\Lambda}(A')_{|B|}
=R_{A'}(\Lambda)
\]
is defined by applying $\phi$ to each tensor factor. Since $\phi$ preserves multiplication
and comultiplication, these maps commute with transfers and restrictions.

\end{proof}

There is a similar equivalence when $R$ has the structure of a symmetric monoidal partition ring. Following \cite[Definition 2.4]{Sinha2} and \cite{StricklandTurner}, a component Hopf ring is a connected commutative cocommutative $\N$-graded bialgebra $A_*$ together with unital commutative products $A_m \otimes_K A_m \to A_m$ for which the coproduct is multiplicative and the Hopf-ring distributivity relation holds.

\begin{cor} \label{cor:symmonhopf}
There is an equivalence of categories between the category of symmetric monoidal partition rings valued in $K$-modules and the category of component Hopf rings over $K$ given by the functor sending a symmetric monoidal partition ring $R$ to $R[\Sigma]$. 
\end{cor}

\begin{proof}
Let 
\[
A_* = R[\Sigma] = \bigoplus_{m\geq 0} R(m).
\]
The partition ring structure gives, for each $m$, a second
commutative product
\[
\ast \colon A_m\otimes_K A_m\longrightarrow A_m,
\]
namely the ring multiplication on $R(m)$. These are the
``addition'' products of the Hopf ring.

The only additional axiom to check is the Hopf-ring distributivity
law. Let
\[
a\in A_{i+j},
\qquad
b\in A_i,
\qquad
c\in A_j,
\]
and write
\[
\Delta_{i,j}(a)=\sum a'\otimes a''
\]
for the component of the coproduct landing in $A_i\otimes_K A_j$.
Since transfer in a partition ring satisfies Frobenius reciprocity,
we have
\[
a\ast \nabla(b \otimes c)
=
R_*(\u i\sqcup\u j\leq \u{i+j})
\left(
R^*(\u i\sqcup\u j\leq \u{i+j})(a)
\ast
m_{i,j}(b\otimes c)
\right).
\]
Using the symmetric monoidal identification
\[
R(\u i\sqcup\u j)\cong R(i)\otimes_K R(j),
\]
the restriction of $a$ is $\sum a'\otimes a''$. Thus the preceding
display becomes
\[
a\ast \nabla(b \otimes c)
=
\sum
\nabla ( (a'\ast b) \otimes (a''\ast c)),
\]
which is precisely the Hopf-ring distributivity law.

The functor the other direction from the proof of \cref{prop:algequiv} lands in symmetric monoidal partition rings in this case because $R(m) = A_m$ has the structure of a commutative ring.

\end{proof}


\begin{example} \label{ex:constantring}
Let $K$ be a commutative ring. The dual constant partition functor $C^K$ is a symmetric monoidal partition functor with values in $K$-modules. The symmetric monoidal structure comes from the isomorphism $K \otimes_K K \to K$. 

Using the formulas of \cref{divconstant}, we see that, since transfer maps are the identity for the dual constant partition functor $C^K$, the transfer multiplication
\[
C^K(i) \otimes C^K(j) \to C^K(i+j)
\]
is the ordinary multiplication on $K$. We learn from this that there is an isomorphism of graded commutative rings
\[
C^K[\Sigma] \cong K[t],
\]
where $t$ is in degree $1$. Since restriction maps are given by multiplication by the index, in the bialgebra structure on $C^K[\Sigma] \cong K[t]$, we have $\Delta(t) = t \otimes 1 + 1 \otimes t \in K[t] \otimes_K K[t]$ and thus
\[
\Delta(t^m) = \sum_{i+j = m} {m \choose i} t^i \otimes t^j.
\]

Now consider the constant partition ring $C_K$. Let $1_m \in C_K(m) = K$ be the multiplicative identity. These form an additive basis for $C_K[\Sigma]$ as a $K$-module. Since the transfer maps in $C_K$ are given by multiplication by the index, the transfer multiplication gives
\[
1_i \cdot 1_j = {i+j \choose i} 1_{i+j}.
\]
It follows that $C_K[\Sigma] \cong K \divpol{t}$ is the free divided power polynomial ring over $K$ on one variable (\cite[\href{https://stacks.math.columbia.edu/tag/07H4}{Section 07H4}]{stacks-project}). 

The unit map $C^K \to \Div(C^K) \cong C_K$ sends the multiplicative identity in $C^K(m)$ to $m!1_m$. Thus the unit map can be identified with the map of $K$-algebras (not necessarily an inclusion) 
\[
K[t] \to K\divpol{t}
\]
sending $t$ to $t$.
\end{example}

\section{Partition power functors} \label{sec:partpower}

In this section we define partition power functors, which are the partition functor analogue of a global power functor and show that $\Div$ is a monad on the category of partition power functors.

\begin{definition} \label{def:partpower}
A partition power functor is a partition ring $R$ equipped with a family of multiplicative maps, called power operations, 
\[
P_m \colon R(1) \to R(m)
\]
for each $m \in \N$, such that
\begin{enumerate}
    \item[(a)] $P_0$ is the constant function at $1$, $P_m(1) = 1$, and $P_1$ is the identity,
    \item[(b)] $R^*(\u{i} \sqcup \u{j} \leq \u{m})(P_m) = m_{i,j}(P_i \otimes P_j)$, whenever $i+j = m$,
    \item[(c)] $P_m(a+b) = \sum_{i+j = m} R_*(\u i \sqcup \u j \leq \u m)(m_{i,j}(P_i(a) \otimes P_j(b)))$.
\end{enumerate}
A map of partition power functors is a map of partition rings that is compatible with the families of power operations.
\end{definition}

If $R$ is a partition power functor, then we will refer to $P_m$ as the $m$th power operation on $R$. The third axiom ensures that the composite
\[
\bar{P}_m \colon R(1) \xrightarrow{P_m} R(m) \xrightarrow{q_{m}} \bar{R}(m)
\]
is additive when $m > 0$. We will reserve the notation $P_m/I_m$ for the composite of $\bar{P}_m$ with the inclusion $\bar{R}(m) \subseteq \Div(R)(m)$.

\begin{prop}
The functor $\Div$ extends to a monad on the category of partition power functors. 
\end{prop}
\begin{proof}
If $R$ is a partition power functor, then we must equip $\Div(R)$ with a family of power operations. This is easy to do since the unit map
\[
\eta \colon R \to \Div(R)
\]
is a map of partition rings and $\Div(R)(1) = R(1)$. It follows that, if $P_m$ is the $m$th power operation on $R$, then $\eta_{m} \circ P_m$ is an $m$th power operation for $\Div(R)$. The axioms that need to be satisfied by the family of power operations follow from the fact that $\eta$ is a map of partition rings. The naturality of $\eta$ implies that $\Div$ is a functor on partition power functors.

Now consider the map of partition rings
\[
\mu \colon \Div(\Div(R)) \to \Div(R).
\]
To see that this is a map of partition power functors, note that the power operations on $\Div(\Div(R))$ come from the power operations on $\Div(R)$ using $\eta_{\Div(R)}$. Since $\mu$ is a map of partition rings, the power operations on $\Div(\Div(R))$ induce power operations on $\Div(R)$ through $\mu$. We want these to be the same as those on $\Div(R)$ through $\eta$. But this follows from the fact that $\id_{\Div(R)} = \mu \circ \eta_{\Div(R)}$.
\end{proof}

Recall that a global power functor $S$ is a global Green functor equipped with power operations $\P_m \colon S(G) \to S(G \wr \Sigma_m)$ satisfying several axioms \cite[Section 5.1]{Schwede}. Restricting along the diagonal map $G \times \Sigma_m \to G \wr \Sigma_m$ gives operations $P_m \colon S(G) \to S(G \times \Sigma_m)$. Since $G \wr \Sigma_0 = e$, when $m = 0$, the target of $\P_0$ is $S(e)$ and $P_0$ is $\P_0$ composed with the restriction map $S(e) \to S(G)$ along the group homomorphism $G \to e$.

\begin{prop} \label{globalpowertopartition}
If $S$ is a global power functor, then $S \circ (G \wr \Sigma_{(-)})$ and $S \circ (G \times \Sigma_{(-)})$ are both partition power functors.
\end{prop}
\begin{proof}
The definition of a global power functor ensures that $\P_0$ is the constant function at $1$ and that
\[
\res_{G \wr (\Sigma_i \times \Sigma_j)}^{G \wr \Sigma_m} (\P_m) = \P_i \boxtimes \P_j,
\]
where $\boxtimes$ is the external multiplication
\[
S(G \wr \Sigma_i) \otimes_{S(e)} S(G \wr \Sigma_j) \to S(G \wr \Sigma_i \times G \wr \Sigma_j).
\]
This implies that $S \circ (G \wr \Sigma_{(-)})$ is a partition power functor.

The fact that $\P_0$ is the constant function at $1$ implies that $P_0$ is also the constant function at $1$. Now consider the restriction
\[
\res_{G \times \Sigma_i \times \Sigma_j}^{G \times \Sigma_m} P_m.
\]
This is equal to $P_i \boxtimes P_j$, where $\boxtimes$ is the external product
\[
S(G \times \Sigma_i) \otimes_{S(G)} S(G \times \Sigma_j) \to S(G \times \Sigma_i \times \Sigma_j)
\]
as there is a commutative diagram
\[
\xymatrix{S(G \wr \Sigma_i) \otimes_{S(e)} S(G \wr \Sigma_j) \ar[dd] \ar[r] & S(G \times \Sigma_i) \otimes_{S(e)} S(G \times \Sigma_j) \ar[d] \\ & S(G \times \Sigma_i) \otimes_{S(G)} S(G \times \Sigma_j) \ar[d]^{\boxtimes} \\ S(G \wr (\Sigma_i \times \Sigma_j)) \ar[r] & S(G \times \Sigma_i \times \Sigma_j),}
\]
in which both the top and bottom arrow are induced by the diagonal maps.
\end{proof}

\cref{cor:diagonalGreen} implies the following:

\begin{cor}
If $S$ is a global power functor, then restriction along the diagonal induces a map of partition power functors
\[
S \circ (G \wr \Sigma_{(-)}) \to S \circ (G \times \Sigma_{(-)}).
\]
\end{cor}

\begin{example}
Let $K$ be a commutative ring and let $C_K$ be the constant partition ring. There is a canonical partition power functor structure on $C_K$ given by $P_m(k) = k^m \in C_K(m) = K$.
\end{example}

\begin{example} \label{ex:powerrestrict}
Let $S$ be a global power functor. The axioms for the power operations imply that $\res_{G}^{G \times \Sigma_m} P_m = (-)^m$, the $m$th power map on $S(G)$. Let $K= S(G)$. The map of partition rings $S \circ (G \times \Sigma_{(-)}) \to C_K$ of \cref{ex:Greentriv} is a map of partition power functors. 
\end{example}

\section{Universal exponential relations} \label{sec:expreln}
In this section we introduce a completed version of the commutative ring $R[\Sigma]$, called $R\powser{\Sigma}$. We prove that every exponential element in $R\powser{\Sigma}$ admits a canonical logarithm in $\Div(R)\powser{\Sigma}$.

Given a lax symmetric monoidal partition functor $R$, let 
\[
R\powser{\Sigma} = \prod_{m \geq 0} R(m)
\]
and let
\[
R\langle \Sigma \rangle = \Div(R)[\Sigma] \, \, \text{ and } \, \, R\divpowser{\Sigma} = \Div(R)\powser{\Sigma}.
\]
The commutative ring $R \powser{\Sigma}$ can be obtained from $R[\Sigma]$ by completing at the system of irrelevant ideals 
\[
\bigoplus_{m \geq n} R(m) \subseteq R[\Sigma]
\]
for $n > 0$. That is, $R \powser{\Sigma}$ is the image of $R[\Sigma]$ under a ``completion" functor from commutative $\N$-graded rings to commutative rings and $(-)\powser{\Sigma}$ is a functor from lax symmetric monoidal partition functors to commutative rings. Because we are interested in power series, commutative rings of the form $R\powser{\Sigma}$ will play a more important role in our story than $R[\Sigma]$.

Now assume that $R$ is a partition ring. Then the homogeneous factors $R(m)$ of $R \powser{\Sigma}$ also have a ring structure. For $r,s \in R(m) \subset R\powser{\Sigma}$, we will write $r *s$ for the product of $r$ and $s$ using the commutative ring structure on $R(m)$. Extending this, for $x \in R(m)$ and $y \in \Div(R)(m)$, we will write
\[
x * y = \eta_{m}(x) * y \in \Div(R)(m).
\] 
We will reserve $*$ for the multiplication internal to $R(\Lambda)$, when $R$ is a partition ring and $\Lambda$ is a partition. 

For $r \in R(i)$ and $s \in R(j)$, we will reserve ordinary multiplication notation for the transfer multiplication in $R \powser{\Sigma}$, so that
\[
rs = R_*(\u i \sqcup \u j \leq \u {i+j})(m_{i,j}(r \otimes s)).
\]

Given a partition $\Lambda \leq \Omega$, let $1_{\Lambda \leq \Omega} \in \bar{R}(\Lambda)$ be the multiplicative identity. Using the notation of \cref{eq:weyl}, we will write $\1_{\Lambda \leq \Omega}$ for the element of $\Div(R)(\Omega)$ given by $\Sigma_{\Omega} 1_{\Lambda \leq \Omega}$. When $\Omega$ is clear from context, we will just write $\1_{\Lambda}$. Note that $\1_{\Lambda \leq \Omega}$ is an idempotent for the multiplication $*$ on $\Div(R)(\Omega)$. Also, note that if $1_{\Lambda \leq \Omega} = 0$ in $\bar{R}(\Lambda)$, which occurs only when $\bar{R}(\Lambda)$ is the zero ring, then $\1_{\Lambda \leq \Omega} = 0$ in $\Div(R)(\Omega)$.

Recall that the orbits for the $\Sigma_m$ action on the set of subpartitions of $\u m$ are in bijective correspondence with integer partitions of $m$. As in \cref{sec:symmon}, we will often use lower case Greek letters to denote integer partitions. Recall that an integer partition $\lambda$ is a sequence of natural numbers $(\lambda_1, \lambda_2, \ldots)$ with all but finitely many $\lambda_i$ equal to zero and that we say that $\lambda$ is an integer partition of $m$, written $\lambda \vdash m$, if 
\[
m = \sum_{i \geq 1} i \cdot \lambda_i.
\]

Let $m,n > 0$ and let $\lambda \vdash m$ and $\gamma \vdash n$ be two integer partitions. The sum $\lambda + \gamma$ is the integer partition of $m+n$ with
\[
(\lambda + \gamma)_i = \lambda_i + \gamma_i.
\]
Thus, for $k \in \N$, $k\lambda \vdash km$ is the partition with $(k\lambda)_i = k\lambda_i$. We will also make use of the factorial of a partition
\[
\lambda ! = \prod_{i \geq 1} (\lambda_i!),
\]
the length of a partition
\[
\ell(\lambda) = \sum_{i \geq 1} \lambda_i,
\]
the size of a partition
\[
\| \lambda \| =\sum_{i \geq 1} i \cdot \lambda_i = m,
\]
and the extension of the choose function to partitions
\[
{m \choose \lambda} = \frac{m!}{\lambda!}.
\]


We will write $\1_{\lambda \vdash m}$, or just $\1_{\lambda}$ for $m = \|\lambda\|$, for the element $\1_{\Lambda \leq \u m} \in \Div(R)(m)$, where $\Lambda$ is any choice of set partition of $\u m$ with underlying integer partition $\lambda$. This does not depend on the choice of $\Lambda$ as 
\begin{equation*}
(\Sigma_{m} 1_{\Lambda \leq \u m})_{\Phi} = \left\{
\begin{array}{rl}
\bar{R}_*(\beta) 1_{\Lambda \leq \u m} = 1_{\Phi \leq \u m} & \text{for } \beta \in \Sigma_{m} \text{ with } \beta \Lambda = \Phi,\\
0 & \text{otherwise,}
\end{array} \right.
\end{equation*}
making use of the fact that $\bar{R}_*(\beta)$ is a ring map. Further, it follows from \cref{weylgroupfixed} that
\begin{equation} \label{eq:idempotentbasis}
\sum_{\lambda \vdash m} \1_{\lambda \vdash m} = 1 \in \Div(R)(m),
\end{equation}
the multiplicative identity for the product denoted $*$. Following \cref{not:indexing}, we have $\1_m = \1_{\u m \leq \u m} = \1_{m \parts m}$.

\begin{prop} \label{prop:partitionmultiply}
Let $R$ be a partition ring and let $\lambda \vdash n$ and $\gamma \vdash m$ be two integer partitions so that $\1_{\lambda} \in \Div(R)(n)$ and $\1_{\gamma} \in \Div(R)(m)$. Then 
\[
\1_\lambda \1_\gamma = \frac{(\lambda+\gamma)!}{\lambda ! \gamma !} \1_{\lambda+\gamma}.
\]
\end{prop}
\begin{proof}
Choose set partitions
\[
\Lambda\vdash [n]
\qquad\text{and}\qquad
\Gamma\vdash [m]
\]
with underlying integer partitions $\lambda$ and $\gamma$, respectively,
and put
\[
\Theta=\Lambda\sqcup\Gamma
\leq
\u n\sqcup\u m.
\]
By inspecting the coordinates in the definition of the lax symmetric
monoidal structure on $\Div(R)$, we have
\[
\Div(m)_{n,m}(\1_\lambda\otimes\1_\gamma)
=
\1_{\Theta\leq \u n\sqcup\u m}.
\]
Thus
\[
\1_\lambda\1_\gamma
=
\Div(R)_*(\u n \sqcup \u m \leq \u{n+m})
\left(
\1_{\Theta\leq\u n\sqcup\u m}
\right).
\]

Let $\theta=\lambda+\gamma$. We evaluate this transfer at a partition $\Phi \leq \u{n+m}$. The $\Phi$-coordinate is zero unless the underlying partition of $\Phi$ is $\theta$. Indeed, a nonzero summand in the transfer formula is indexed by a permutation $[\sigma]$ such that $\sigma^{-1}\Phi$ lies in the $\Sigma_n \times \Sigma_m$-orbit of $\Theta$. This forces $\Phi$ to have $\theta_i = \lambda_i+\gamma_i$ blocks of size $i$ for $i \geq 1$.


Assume now that the underlying partition of $\Phi$ is $\theta$. The double cosets contributing to the $\Phi$-coordinate are in bijection with choices, for each $i \geq 1$ of $\lambda_i$ blocks from the $\theta_i$ blocks of $\Phi$ of size $i$. The chosen blocks are sent to the $n$-component and the remaining $\gamma_i$ blocks are sent to the $m$-component.

Thus the number of contributing terms is
\[
\prod_{i\geq1}
\binom{\lambda_i+\gamma_i}{\lambda_i}.
\]
Each contributing term is the multiplicative identity in
$\bar R(\Phi)$, since it is obtained from the multiplicative identity
by an isomorphism of partitions. Therefore
\[
\left(
\1_\lambda\1_\gamma
\right)_\Phi
=
\Bigg(
\prod_{i\geq1}
\binom{\lambda_i+\gamma_i}{\lambda_i}
\Bigg)
1_{\Phi\leq\u{n+m}}.
\]
Since this holds for every partition $\Phi$ whose underlying integer partition is $\theta$, we obtain
\[
\1_\lambda\1_\gamma
=
\Bigg(
\prod_{i\geq1}
\binom{\lambda_i+\gamma_i}{\lambda_i}
\Bigg)
\1_{\lambda+\gamma}.
\]
\end{proof}

It follows that
\begin{equation*}
\1_{n} \1_{m} = \left\{
\begin{array}{rl}
\1_{\u{n} \sqcup \u{m} \leq \u{n+m}} & \text{if } n \neq m,\\
2 \1_{\u{n} \sqcup \u{m} \leq \u{n+m}} & \text{if } n = m,
\end{array} \right.
\text{ and } (\1_{\lambda})^k = \frac{(k \lambda)!}{(\lambda !)^k} \1_{k \lambda}.
\end{equation*}




From now on, we will denote an element of $R\powser{\Sigma}$ as a power series $\sum_{m \geq 0} r_m t^m$, where $r_m \in R(m)$ and $t$ is a dummy variable tracking the grading.


\begin{definition}
For $R$ a lax symmetric monoidal partition functor, we say that $\sum_{m \geq 0} r_m t^m \in R \powser{\Sigma}$ is exponential if $r_0 =1$ and, whenever $m=i+j$, we have
\[
R^*(\u{i} \sqcup \u{j} \leq \u{m})(r_m) = m_{i, j}(r_i \otimes r_j).
\]
\end{definition}

Note that if $R \to R'$ is a map of lax symmetric monoidal partition functors then the induced map of commutative rings $R\powser{\Sigma} \to R'\powser{\Sigma}$ sends exponential elements to exponential elements. \cref{def:partpower} ensures that if $R$ has the structure of a partition power functor, then for $r \in R(1)$, we have $\sum_{m \geq 0} P_m(r)t^m$ is an exponential element in $R \powser{\Sigma}$. 

\begin{theorem} \label{exponentialformula}
If $R$ is a partition ring and $\sum_{m \geq 0} r_mt^m \in R\divpowser{\Sigma} = \Div(R)\powser{\Sigma}$ is exponential, then we have the identity
\[
\sum_{m \geq 0} r_mt^m = \exp \Big (\sum_{i \geq 1} (r_i * \1_{i}) t^i \Big ).
\]
\end{theorem}
\begin{proof}
The exponential is interpreted coefficientwise using the usual power series for $\exp$; \cref{prop:partitionmultiply} will show that the factorials arising from products of the elements $\1_i$ cancel the denominators, so the right hand side is defined integrally in $R\divpowser{\Sigma}$.

First notice that
\begin{equation*}
\begin{split}
(r_i * \1_{i})(r_j * \1_{j}) &= \Div(R)_*(\u i \sqcup \u j \leq \u {i+j})(\Div(m)_{i,j}((r_i * \1_{i}) \otimes (r_j * \1_{j}))) \\
&= \Div(R)_*(\u i \sqcup \u j \leq \u {i+j})(\Div(m)_{i,j}(r_i \otimes r_j) * \1_{\u i \sqcup \u j}) \\
&= r_{i+j}*\Div(R)_*(\u i \sqcup \u j \leq \u {i+j})(\1_{\u i \sqcup \u j}) \\
&= r_{i+j} * \1_{i} \1_{j}.
\end{split}
\end{equation*}
The second equality follows from the fact that $\Div(m)_{i, j}$ is a ring map for the $*$ product. The third equality follows from Frobenius reciprocity, making use of the fact that $\sum_{m \geq 0} r_mt^m$ is exponential.

This implies that if $\ell(\lambda) = m$, then
\begin{equation} \label{eq:exppowers}
\prod_{i\geq 1} \big ( (r_i * \1_{i})t^i \big )^{\lambda_i} = \prod_{i \geq 1} (r_{i\lambda_i} * \lambda_i! \1_{\lambda_i \u i \leq \u{i \lambda_i}}) t^{i \lambda_i} = (r_{\| \lambda \|} * \lambda! \1_{\lambda}) t^{\| \lambda \|},
\end{equation}
where $\lambda_i \u i$ is the $\lambda_i$-fold coproduct of $\u i$ with itself.

We calculate that
\begin{equation*}
\begin{split}
\exp \Big (\sum_{i \geq 1} (r_i * \1_{i}) t^i \Big ) &= \sum_{m \geq 0} \frac{1}{m!} \Big (\sum_{i \geq 1} (r_i * \1_{i}) t^i \Big )^m \\
&= \sum_{m \geq 0} \frac{1}{m!} \sum_{\{\lambda \mid \ell(\lambda) = m\}} {m \choose \lambda} \prod_{i \geq 1} (r_i * \1_{i} t^i)^{\lambda_i} \\
&= \sum_{m \geq 0} \frac{1}{m!} \sum_{\{\lambda \mid \ell(\lambda) = m\}} {m \choose \lambda} r_{\| \lambda \|} * \lambda ! \1_{\lambda} t^{\| \lambda \|} \\
&= \sum_{\lambda} r_{\| \lambda \|} * \1_{\lambda} t^{\| \lambda \|} \\
&= \sum_{m \geq 0} r_m t^m.
\end{split}
\end{equation*}
The second equality is the multinomial expansion. The third equality makes use of \eqref{eq:exppowers}. The fourth equality indexes the sum over the set of all integer partitions. The final equality makes use of \eqref{eq:idempotentbasis}.
\end{proof}


\begin{remark}
\cref{exponentialformula} will primarily be applied to exponential elements in $R \powser{\Sigma}$, viewing them in $R \divpowser{\Sigma}$ through $\eta$.
\end{remark}

\begin{remark} \label{rem:genexp}
\cref{exponentialformula} applies somewhat more generally. It is only required that $\Div(R)$ is a partition ring, not $R$ itself. This more general situation actually does sometimes appear in practice as seen in \cref{constantexp} below.
\end{remark}


\begin{remark} \label{rem:grouplike}
When $R$ is a symmetric monoidal partition ring, \cref{cor:symmonhopf} applies. In this case, exponential elements correspond to the group-like elements in the completed component Hopf ring $R\powser{\Sigma}$, using the completed tensor product for the structure maps. The elements $r_m \ast \1_{m}$ satisfy
\[
\Div(R)^*(\u i \sqcup \u j \leq \u m)(r_m \ast \1_{m}) = 0
\]
for $i+j = m$ and $i,j>0$. Thus these elements are primitive whenever $\Div(R)$ is a symmetric monoidal partition functor (e.g., under the conditions of \cref{prop:flatdivsymmon}).
\end{remark}

The following proposition provides a sense in which the exponential relation associated to an exponential element over $R\divpowser{\Sigma}$ is universal:
\begin{prop}
Assume that $R$ is a partition ring and $R \to R'$ is a map of lax symmetric monoidal partition functors in which $R'$ is a $\Div$-algebra. If $\sum_{m \geq 0} r_mt^m$ is an exponential element in $R\powser{\Sigma}$, then the exponential relation 
\[
\sum_{m \geq 0} r_mt^m = \exp \Big (\sum_{i \geq 1} (r_i * \1_{i}) t^i \Big )
\]
in $R\divpowser{\Sigma}$ gives rise to an exponential relation in $R'\powser{\Sigma}$. 
\end{prop}
\begin{proof}
Since $\Div(R)$ is the free $\Div$-algebra on $R$ and $R'$ is a $\Div$-algebra, we obtain a universal map $\Div(R) \to R'$. This further induces a map of commutative rings $R\divpowser{\Sigma} \to R'\powser{\Sigma}$ which we may apply to the exponential relation to obtain an exponential relation in $R'\powser{\Sigma}$.
\end{proof}

\begin{example} \label{constantexp}
Let $K$ be a commutative ring and let $k \in K$. Consider $\sum_{m \geq 0} k^{*m}t^m \in C_K \powser{\Sigma}$. This is an exponential element as restriction maps are the identity and the multiplication for $C_K$ is given by multiplication in $K$. Under the isomorphism of commutative rings (see \cref{ex:constantring})
\[
C_K \powser{\Sigma} \cong K \divpowser{t},
\]
the exponential element $\sum_{m \geq 0} k^{*m}t^m$ is sent to $\sum_{m \geq 0} k^m \frac{t^m}{m!}$.

Since $C_K \cong \Div(C^K)$, \cref{exponentialformula} guarantees that 
\begin{equation*}
\begin{split}
\sum_{m \geq 0} k^{*m}t^m &= \exp \Big (\sum_{i \geq 1} (k^{* i} * \1_{i}) t^i \Big ) \\
&= \exp \Big (k*\1_{1}t \Big )
\end{split}
\end{equation*}
in $C_K\powser{\Sigma}$, where the second equality follows from the fact that $\1_{i} = 0 \in \Div(C^K)(i)$ for all $i> 1$. Applying the isomorphism to $K\divpowser{t}$, we just have
\[
\sum_{m \geq 0} k^m \frac{t^m}{m!} = \exp(kt).
\]
\end{example}

\begin{example} \label{ex:exppower}
Assume that $R$ is a partition power functor and let $P_k/I_k$ be the composite
\[
R(1) \to R(k) \to \bar{R}(k) \subseteq \Div(R)(k),
\]
as opposed to $\bar{P}_k$ which lands in $\bar{R}(k)$. Since $\sum_{m \geq 0} P_m t^m$ takes values in exponential elements, \cref{exponentialformula} provides us with an exponential relation
\[
\sum_{m \geq 0} P_m t^m = \exp\Big(\sum_{k \geq 1} P_k/I_k t^k\Big)
\]
between the total multiplicative and total additive power operations. Applying \cref{globalpowertopartition}, if $S$ is a global power functor, then the family of power operations on the partition power functors $R = S \circ (G \wr \Sigma_{(-)})$ or $R = S \circ (G \times \Sigma_{(-)})$ satisfy this exponential relation over $\Div(R)$. This provides us with an exponential relation between additive and multiplicative power operations that applies very generally. 


We find it interesting that the third axiom in the definition of a partition power functor is not needed for this exponential relation and thus we could drop the requirement that $\bar{P}_k$ is additive.
\end{example}

\begin{example}
Assume $S$ is a global power functor, $G$ is a fixed finite group, $K = S(G)$, and $R = S \circ (G \times \Sigma_{(-)})$ is the associated partition power functor. By \cref{ex:powerrestrict}, restriction along $G \to G \times \Sigma_{\Lambda}$ gives a map of partition power functors $R \to C_K$ and the induced map
\[
R\powser{\Sigma} \to C_K\powser{\Sigma} \cong K\divpowser{t}
\]
sends $\sum_{m \geq 0} P_m t^m$ to $\sum_{m \geq 0} (-)^{*m} t^m \mapsto \sum_{m \geq 0} (-)^{m} \frac{t^m}{m!}$.

Since $C_K$ is a $\Div$-algebra, we have an induced map $\Div(R) \to C_K$. \cref{constantexp} implies that the exponential relation
\[
\sum_{m \geq 0} P_m t^m = \exp\Big(\sum_{k \geq 1} P_k/I_k t^k\Big)
\]
is sent to the exponential relation
\[
\sum_{m \geq 0} (-)^m \frac{t^m}{m!} = \exp\Big((-)t\Big)
\]
in $K\divpowser{t}$.
\end{example}

\section{The divided power envelope} \label{sec:divpow}
In this section we consider the relationship between $\Div(R)[\Sigma]$ and the algebra of symmetric invariant tensors in the case that $R$ is a symmetric monoidal partition functor.

We begin with a review of the algebra of symmetric invariant tensors. Let $K$ be a commutative ring and let $M$ be a $K$-module and let 
\[
TS(M) = \bigoplus_{m \geq 0} (M^{\otimes m})^{\Sigma_m}
\]
and
\[
T(M) = \bigoplus_{m \geq 0} M^{\otimes m},
\]
where the tensor products are over $K$. Then $TS(M)$ is a submodule of $T(M)$, but more is true. There is a graded $K$-algebra structure on $T(M)$ with multiplication given by the shuffle product. If $x \in M^{\otimes i}$ and $y \in M^{\otimes j}$ are simple tensors, then the shuffle product of $x$ and $y$ is
\[
xy = \sum_{[\sigma_{i,j}] \in \Sigma_{i+j}/(\Sigma_i \times \Sigma_j)} \sigma_{i,j}(x\otimes y),
\]
where the $\sigma_{i,j}$'s are the $(i,j)$-shuffles. This multiplication induces a $K$-algebra structure on $TS(M)$ as well and we call the resulting algebra the symmetric invariant tensor algebra. This applies to $M = M_*$, an $\N$-graded $K$-module, as well. Note that the degree $1$ part of $TS(M_*)$ is $M_0$, the degree $2$ part is $M_1 \oplus ((M_{0})^{\otimes 2})^{\Sigma_2}$, and so on. 

In \cite[Chapter III and Chapter IV]{Roby-divided} (see also \cite[Section 4]{Cartan}), it was shown that if $M_*$ is an $\N$-graded free $K$-module, then $TS(M_*)$ is canonically isomorphic to the free graded divided polynomial algebra on a choice of homogeneous basis for $M_*$. This has consequences for partition functors for the following reason:

\begin{prop} \label{prop:symmetricinvariant}
Assume that $R$ is a symmetric monoidal partition functor taking values in $K$-modules and let $M_*$ be the $\N$-graded $K$-module with $M_i = \bar{R}(i+1)$. There is an isomorphism of commutative graded $K$-algebras
\[
TS(M_*) \cong \Div(R)[\Sigma].
\]
\end{prop}
\begin{proof}
Let $\lambda \parts m$ be an integer partition and choose a partition
$\Lambda \leq \u m$ with underlying integer partition $\lambda$. Since $R$ is symmetric monoidal, \cref{lem:symmon} gives an isomorphism
\[
\bar{R}(\Lambda)
\cong
\bigotimes_i \bar{R}(i)^{\otimes \lambda_i},
\]
where the tensor product is over $K$. Under this isomorphism, \cref{eq:permutetensors} identifies the action of
the Weyl group $W_{\Sigma_m}(\Sigma_{\Lambda}) \cong \prod_i \Sigma_{\lambda_i}$
with the action that permutes the tensor factors of equal degree. Thus
\cref{weylgroupfixed} gives an isomorphism
\[
\Div(R)(m)
\cong
\bigoplus_{\lambda \parts m}
\bigg(
\bigotimes_i \bar{R}(i)^{\otimes \lambda_i}
\bigg)^{\prod_i \Sigma_{\lambda_i}}.
\]
Consequently,
\[
\Div(R)[\Sigma]
\cong
\bigoplus_{m \geq 0}
\bigoplus_{\lambda \parts m}
\bigg(
\bigotimes_i \bar{R}(i)^{\otimes \lambda_i}
\bigg)^{\prod_i \Sigma_{\lambda_i}}
\]
as graded $K$-modules.

We compare this with $TS(M_*)$. Recall that $M_{i-1}=\bar{R}(i)$. For $r \geq 0$, the contribution of $M_*^{\otimes r}$ to degree $m$ is
\[
\bigoplus_{\substack{i_1+\cdots+i_r=m\\ i_j>0}}
\bar{R}(i_1)\otimes\cdots\otimes\bar{R}(i_r).
\]
The group $\Sigma_r$ acts by permuting these summands and the tensor factors.
The orbits of the summands are indexed by the integer partitions
$\lambda \parts m$ of length $\ell(\lambda)=r$. For the orbit associated to
$\lambda$, choose the summand
\[
\bigotimes_i \bar{R}(i)^{\otimes\lambda_i}.
\]
Its stabilizer in $\Sigma_r$ is $\prod_i\Sigma_{\lambda_i}$. An invariant
element in the direct sum over the orbit is uniquely determined by its
component in this chosen summand, and this component must be fixed by
$\prod_i\Sigma_{\lambda_i}$. It follows that
\[
\bigg (
\bigoplus_{\substack{i_1+\cdots+i_r=m\\ i_j>0}}
\bar{R}(i_1)\otimes\cdots\otimes\bar{R}(i_r)
\bigg)^{\Sigma_r}
\cong
\bigoplus_{\substack{\lambda \parts m\\ \ell(\lambda)=r}}
\bigg(
\bigotimes_i \bar{R}(i)^{\otimes\lambda_i}
\bigg)^{\prod_i\Sigma_{\lambda_i}}.
\]
Summing over $r$ gives an isomorphism of graded $K$-modules
\[
TS(M_*) \cong \Div(R)[\Sigma].
\]

It remains to identify the multiplication under this isomorphism. Let $\lambda\vdash k$ and $\omega\vdash\ell$, and choose
partitions $\Lambda\leq \u k$ and $\Omega\leq \u \ell$ with underlying
integer partitions $\lambda$ and $\omega$, respectively. Let
\[
s \in
\bigg(
\bigotimes_i \bar{R}(i)^{\otimes\lambda_i}
\bigg)^{\prod_i\Sigma_{\lambda_i}}
\]
and
\[
t \in
\bigg(
\bigotimes_i \bar{R}(i)^{\otimes \omega_i}
\bigg)^{\prod_i\Sigma_{\omega_i}}.
\]
Using the notation preceding \cref{weylgroupfixed}, these determine elements
\[
\Sigma_{k}s \in \Div(R)(k)
\qquad\text{and}\qquad
\Sigma_{\ell}t \in \Div(R)(\ell).
\]
The lax symmetric monoidal multiplication gives
\[
m_{k,\ell}
\left(
\Sigma_{k}s\otimes\Sigma_{\ell}t
\right)
=
\Sigma_{\u k\sqcup\u\ell}(s\otimes t)
\in
\Div(R)(\u k\sqcup\u\ell).
\]
This element is supported on the orbit of the partition
$\Lambda \sqcup \Omega$.

Let $m=k+\ell$. The product of $\Sigma_{k}s$ and
$\Sigma_{\ell}t$ in $\Div(R)[\Sigma]$ is obtained by applying the transfer
along $\u k\sqcup\u\ell\leq\u m$. Fix a partition
$\Gamma\leq\u m$ having $\lambda_i+\omega_i$ blocks of cardinality $i$.
The $\Gamma$-component of the product is
\[
\sum_{\substack{
[\sigma]\in
\Sigma_\Gamma\backslash\Sigma_m/
\Sigma_{\u k\sqcup\u\ell}\\
\sigma^{-1}\Gamma\leq\u k\sqcup\u\ell}}
\bar{R}_*(\sigma)
\left(
\Sigma_{\u k\sqcup\u\ell}(s\otimes t)_{\sigma^{-1}\Gamma}
\right).
\]
A summand is nonzero precisely when $\sigma^{-1}\Gamma$ has $\lambda_i$
blocks of cardinality $i$ in $\u k$ and $\omega_i$ blocks of cardinality $i$
in $\u\ell$, for every $i$. Thus the double cosets which contribute are in
bijection with
\[
\prod_i
\Sigma_{\lambda_i+\omega_i}/
(\Sigma_{\lambda_i}\times\Sigma_{\omega_i}).
\]
Indeed, such a double coset records, for every $i$, which $\lambda_i$ of the
$\lambda_i+\omega_i$ blocks of cardinality $i$ are sent to $\u k$.

For each contributing double coset we may choose a representative $\sigma$
for which
\[
\Sigma_{\u k\sqcup\u\ell}(s\otimes t)_{\sigma^{-1}\Gamma}
=
s\otimes t.
\]
By \cref{eq:permutetensors}, the map $\bar{R}_*(\sigma)$ simply permutes
the corresponding tensor factors. It follows that the $\Gamma$-component of
the product is
\[
\sum_{[\sigma]\in
\prod_i
\Sigma_{\lambda_i+\omega_i}/
(\Sigma_{\lambda_i}\times\Sigma_{\omega_i})}
\sigma(s\otimes t).
\]
Under the module isomorphism above, this is exactly the component of the
shuffle product of $s$ and $t$ in $TS(M_*)$.

Since an invariant element is determined by its component at the chosen
representative of each orbit, the multiplication on $\Div(R)[\Sigma]$
agrees with the shuffle multiplication on $TS(M_*)$. Thus the isomorphism
above is an isomorphism of commutative graded $K$-algebras.
\end{proof}

\begin{cor} \label{cor:divfree}
In the context of the previous proposition, if $\bar{R}(i)$ is a free $R(0)$-module for each $i$, then $\Div(R)[\Sigma]$ is the free graded divided polynomial algebra on a homogeneous basis for $\bigoplus_{i>0} \bar{R}(i)$.
\end{cor}

With this in hand, we can give a universal property of $\Div(R)[\Sigma]$ as an $R[\Sigma]$-algebra when $R$ is a symmetric monoidal partition functor with $\bar{R}(i)$ free as an $R(0)$-module.

\begin{theorem} \label{thm:freediv}
Assume that $R$ is a symmetric monoidal partition functor such that $\bar{R}(m)$ is a free $R(0)$-module for all $m \geq 0$. The divided power algebra $\Div(R)[\Sigma]$ is the divided power envelope of $R[\Sigma]$ with respect to the irrelevant ideal $\oplus_{m \geq 1} R(m)$.
\end{theorem}
\begin{proof}
Let $f \colon R[\Sigma] \to D$ be a ring map to a divided power algebra $(D,J)$ with $f(\oplus_{m \geq 1} R(m)) \subseteq J$. We wish to show that $f$ extends uniquely to a map of divided power algebras $\Div(R)[\Sigma] \to D$. Note that $D$ is not assumed to be graded.

Let $R[\Sigma]^{\leq m}$ be the subalgebra of $R[\Sigma]$ generated by $R[\Sigma]_{\leq m}$, the elements of degree less than or equal to $m$. We will use induction along the subalgebras $R[\Sigma]^{\leq m}$ to extend $f$ in a unique way to a divided power algebra map from $\Div(R)[\Sigma]$ to $D$. 


Let $f_m \colon R[\Sigma]^{\leq m} \to D$ be the restriction of $f$ to $R[\Sigma]^{\leq m}$. Let $\Div(R)[\Sigma]^{\leq m}$ be the \emph{divided power} subalgebra of $\Div(R)[\Sigma]$ generated by $\Div(R)[\Sigma]_{\leq m}$. 

Since $\Div(R)(0) = R(0)$, $\Div(R)(1) = R(1)$, and $\Div(R)[\Sigma]^{\leq 1} = TS(R(1))$ is the free graded divided power algebra on a basis of $R(1)$, there is a unique way to extend $f_1$ to $\Div(R)[\Sigma]^{\leq 1}$ using the divided power algebra structure of the target $(D,J)$.


Now assume that we have constructed a unique extension of $f_{m-1}$ to $\Div(R)[\Sigma]^{< m}$.

Note that the inclusion $\Div(R)[\Sigma]^{< m} \to \Div(R)[\Sigma]^{\leq m}$ is an isomorphism in degrees less than $m$. In degree $m$, we have 
\[
(\Div(R)[\Sigma]^{< m})_m = \bigoplus_{\substack{\lambda \parts m \\ \lambda_m \neq 1}} \bigotimes_i (\bar{R}(i)^{\otimes \lambda_i})^{\Sigma_{\lambda_i}} \subseteq \Div(R)(m),
\]
using the fact that $\Div(R)[\Sigma]^{<m}$ is the divided power subalgebra of $\Div(R)[\Sigma]$ generated by $\Div(R)[\Sigma]_{<m}$. It follows that there is a short exact sequence of $R(0)$-modules
\[
(\Div(R)[\Sigma]^{< m})_m \to (\Div(R)[\Sigma]^{\leq m})_m \to \bar{R}(m).
\]

On the other hand, since $R$ is symmetric monoidal, we have
\[
(R[\Sigma]^{<m})_m = I_{m} \subseteq R(m).
\]

Further, note that the commutative square of $R(0)$-modules
\[
\xymatrix{I_{m} \ar@{^{(}->}[r] \ar[d] & R(m) \ar[d] \\ (\Div(R)[\Sigma]^{< m})_m \ar@{^{(}->}[r] & \Div(R)(m)}
\]
is a pushout square as the horizontal arrows are injections and the induced map of cokernels is an isomorphism (they are isomorphic to $\bar{R}(m)$). It follows that the restriction of $f_m$ to $R[\Sigma]_{\leq m}$ extends in a unique way to $\Div(R)[\Sigma]_{\leq m}$.

Finally, we have a commutative diagram
\[
\xymatrix{R[\Sigma]_{\leq m} \ar[r] \ar[d] & R[\Sigma]^{\leq m} \ar[d] \\ \Div(R)[\Sigma]^{< m} \oplus \bar{R}(m) \ar[r] & \Div(R)[\Sigma]^{\leq m}.}
\]
In the lower left corner, $\bar{R}(m)$ should be viewed as a graded $R(0)$-module concentrated in degree $m$. We have already extended $(f_m)|_{R[\Sigma]_{\leq m}}$ to $\Div(R)[\Sigma]^{< m} \oplus \bar{R}(m)$, the restriction to $\Div(R)[\Sigma]^{< m}$ is the extension of $f_{m-1}$ from our induction hypothesis. This extension extends further to $\Div(R)[\Sigma]^{\leq m}$ using the fact that $\Div(R)[\Sigma]^{\leq m}$ is the free divided polynomial algebra over $\Div(R)[\Sigma]^{< m}$ on a basis of $\bar{R}(m)$ by \cref{cor:divfree}. Since $f_m$ agrees with this extension on a set of generators of $R[\Sigma]^{\leq m}$, we obtain a unique extension of $f_m$ to $\Div(R)[\Sigma]^{\leq m}$ compatible with the map $R[\Sigma]^{\leq m} \to \Div(R)[\Sigma]^{\leq m}$.

Finally taking the colimit along $m$ gives the desired result.
\end{proof}

\begin{example} \label{ex:symmetriccoinvariants}
Let $K$ be a commutative ring and let $M$ be a $K$-module. We describe a symmetric monoidal partition functor generalizing $C^K$ whose associated ring is the symmetric algebra on $M$.

For a finite set $X$, let
\[
M^{\otimes X} = \bigotimes_{x \in X} M,
\]
where the tensor product is taken over $K$. If $\Lambda = (X,\sim_{\Lambda})$ is a partition, then $\Sigma_{\Lambda}$ acts on $M^{\otimes X}$ by permuting the tensor factors. Define
\[
S_M(\Lambda) = (M^{\otimes X})_{\Sigma_{\Lambda}},
\]
where the subscript denotes coinvariants.

We describe restriction and transfer maps making $S_M$ into a partition functor. First suppose that $\Gamma \leq \Lambda$. The transfer
\[
(S_M)_*(\Gamma \leq \Lambda) \colon
(M^{\otimes X})_{\Sigma_{\Gamma}}
\longrightarrow
(M^{\otimes X})_{\Sigma_{\Lambda}}
\]
is the canonical quotient map. The restriction
\[
S_M^*(\Gamma \leq \Lambda) \colon
(M^{\otimes X})_{\Sigma_{\Lambda}}
\longrightarrow
(M^{\otimes X})_{\Sigma_{\Gamma}}
\]
is the norm map
\[
[x]
\longmapsto
\sum_{[\sigma] \in \Sigma_{\Gamma}\backslash \Sigma_{\Lambda}}
[\sigma x].
\]
This is independent of the choice of representatives for the left cosets and is well-defined on $\Sigma_{\Lambda}$-coinvariants. For an isomorphism of partitions, the restriction and transfer maps are induced by the corresponding permutation of the tensor factors. The maps associated to arbitrary maps of partitions are obtained using the factorization of \cref{lemma:factorization}. The usual double coset decomposition shows that these restriction and transfer maps satisfy the double coset formula. Thus $S_M$ is a partition functor. Note that, if $M$ is free of rank $>1$, then $S_M$ is not restricted from a global functor as the values of $S_M$ on the discrete partitions have distinct ranks. In the case that $M=K$, we will see below that $S_M = C^K$.   

In fact, $S_M$ is a symmetric monoidal partition functor with values in $K$-modules. The symmetric monoidal structure is given by the canonical isomorphism
\[
S_M(\Lambda) \otimes_K S_M(\Omega)
\longrightarrow
S_M(\Lambda \sqcup \Omega)
\]
sending
\[
[x] \otimes [y]
\longmapsto
[x \otimes y].
\]
This is an isomorphism since
\[
\Sigma_{\Lambda \sqcup \Omega}
\cong
\Sigma_{\Lambda} \times \Sigma_{\Omega}.
\]

The associated graded ring has a familiar description:
\[
S_M[\Sigma]
=
\bigoplus_{m \geq 0}(M^{\otimes m})_{\Sigma_m}
\cong
\Sym_K(M).
\]
Under this isomorphism, the transfer product is the usual multiplication in the symmetric algebra.

The effect of $\Div$ on this partition functor is particularly simple. If $\Lambda$ is not discrete, then there is a proper refinement $\Gamma < \Lambda$, and the transfer $S_M(\Gamma) \longrightarrow S_M(\Lambda)$ is surjective. It follows that $\bar{S}_M(\Lambda)=0$ unless $\Lambda$ is discrete. If $\bar{X}$ denotes the discrete partition of $X$, then $\bar{S}_M(\bar{X}) = M^{\otimes X}$. Consequently, for any partition $\Lambda$ of $X$, $\Div(S_M)(\Lambda) \cong (M^{\otimes X})^{\Sigma_{\Lambda}}$. In particular,
\[
\Div(S_M)[\Sigma]
\cong
\bigoplus_{m \geq 0}(M^{\otimes m})^{\Sigma_m}
=
TS(M),
\]
where the multiplication on the right is the shuffle product described above.

Under these identifications, the unit
\[
\eta \colon S_M \longrightarrow \Div(S_M)
\]
is the norm from coinvariants to invariants. At the indiscrete partition $\u m$, it is the map
\[
(M^{\otimes m})_{\Sigma_m}
\longrightarrow
(M^{\otimes m})^{\Sigma_m}
\]
given by
\[
[x]
\longmapsto
\sum_{\sigma \in \Sigma_m} \sigma x.
\]
Thus after applying $(-)[\Sigma]$, the unit of the $\Div$ monad is the canonical symmetrization map
\[
\Sym_K(M) \longrightarrow TS(M).
\]

When $M=K$, the symmetric groups act trivially on $M^{\otimes X} \cong K$. The transfer maps are the identity and the restriction along $\Gamma \leq \Lambda$ is multiplication by $|\Sigma_{\Lambda}/\Sigma_{\Gamma}|$. Thus
\[
S_K \cong C^K,
\]
so this construction generalizes the dual constant partition functor and discussion of \cref{divconstant}.

Finally, if $M$ is a free $K$-module, then \cref{thm:freediv} identifies
\[
TS(M) \cong \Div(S_M)[\Sigma]
\]
as the divided power envelope of
\[
\Sym_K(M) \cong S_M[\Sigma]
\]
with respect to its ideal generated by elements of positive degree.
\end{example}

\begin{example} \label{ex:PDnotsymmetric}
The functor $\Div$ does not preserve symmetric monoidal partition functors in
general. We use a counterexample of Lundkvist
\cite[Example 5.3]{Lundkvist}.

Let $k$ be a field of characteristic $2$, let
\[
K=k[s,t],
\qquad
M=K^2/\langle se_1+te_2\rangle,
\]
and let
\[
K'=K[z]/(z(s+t)).
\]
Lundkvist shows that the base change map
\[
TS_2(M)\otimes_KK'
\longrightarrow
((M\otimes_KK')^{\otimes_{K'} 2})^{\Sigma_2}
\]
is not injective.

Set $N=M\oplus K'$, viewed as a $K$-module, and consider the symmetric monoidal partition functor
$S_N$ of \cref{ex:symmetriccoinvariants}. The lax symmetric monoidal structure gives a map
\[
\Div(S_N)(2)\otimes_K\Div(S_N)(1)
\longrightarrow
\Div(S_N)(\u 2\sqcup\u 1).
\]
Under the identification above, this is
\[
TS_2(N)\otimes_KN
\longrightarrow
(N^{\otimes 3})^{\Sigma_2}.
\]
Using the decomposition $N=M\oplus K'$, this map contains Lundkvist's
base change map as a direct summand. It is therefore not an isomorphism. Thus $S_N$ is a
symmetric monoidal partition functor for which $\Div(S_N)$ is not symmetric
monoidal.
\end{example}


\section{The initial partition ring} \label{sec:init}

Recall that the initial global Green functor is the Burnside ring global Green functor. In this section we will show that the initial partition ring is the representation ring functor. This has some concrete advantages: the representation ring partition functor is much smaller than the Burnside ring partition functor and it is a symmetric monoidal partition functor -- unlike the Burnside ring partition functor.

Let $\AA(\Lambda) \subseteq A(\Sigma_{\Lambda})$ be the subgroup of the Burnside ring of $\Sigma_{\Lambda}$ on the basis elements of the form $[\Sigma_{\Lambda}/\Sigma_{\Gamma}]$, for $\Gamma \leq \Lambda$. Note that $\Sigma_{\Lambda} / \Sigma_{\Gamma} \cong \Sigma_{\Lambda} / \Sigma_{\Gamma '}$ as $\Sigma_{\Lambda}$-sets if and only if there exists $\sigma \in \Sigma_{\Lambda}$ such that $\sigma \Gamma = \Gamma '$. The subgroup $\AA(\Lambda)$ is a subring of $A(\Sigma_{\Lambda})$ as Young subgroups are closed under conjugation and intersection. 

This construction extends to a partition ring. It admits restriction maps as, again, the collection of Young subgroups is closed under conjugation and intersection. It admits transfer maps as, given $\Lambda \leq \Omega$, we have an isomorphism of $\Sigma_{\Omega}$-sets $\Sigma_{\Omega} \times_{\Sigma_{\Lambda}} \Sigma_{\Lambda}/\Sigma_{\Gamma} \cong \Sigma_{\Omega}/\Sigma_{\Gamma}$. We see that the transfer sends basis elements to basis elements. The transfer maps satisfy Frobenius reciprocity because they are compatible with the transfer maps in the Burnside ring -- which satisfy Frobenius reciprocity. 

Further, $\AA$ is a symmetric monoidal partition ring. This follows from the fact that $\AA(\Lambda) \otimes \AA(\Omega)$ and $\AA(\Lambda \sqcup \Omega)$ are free of the same rank and the canonical map $\AA(\Lambda) \otimes \AA(\Omega) \to \AA(\Lambda \sqcup \Omega)$ is surjective.

\begin{prop}
The partition ring $\AA$ is the initial partition ring.
\end{prop}
\begin{proof}
Let $R$ be a partition ring. Our goal is to produce a unique map of partition rings $\iota \colon \AA \to R$. Let $\Lambda = (X, \sim_{\Lambda})$ be a partition. We will first use induction to produce a unique map of abelian groups $\AA(\Lambda) \to R(\Lambda)$ natural in transfer maps. 

Let $\bar{X}$ be the discrete partition on $X$. Note that $\AA(\bar{X}) = \Z$ is the integers so there is a unique ring map $\AA(\bar{X}) \to R(\bar{X})$.

Now assume that we have constructed a unique map of abelian groups $\AA(\Theta) \to R(\Theta)$ for all $\Theta < \Lambda$ compatible with transfers along inclusions. Since $\iota$ must commute with transfer maps, the value of every basis element except for $[\Sigma_{\Lambda}/\Sigma_{\Lambda}]$ is determined by the transfer maps. Since $[\Sigma_{\Lambda}/\Sigma_{\Lambda}]$ is the multiplicative unit in $\AA(\Lambda)$, it must go to $1$ in $R(\Lambda)$.

To see that these maps commute with restrictions along inclusions we make use of the double coset formula. Assume that $\iota$ commutes with restrictions for all proper subpartitions of $\Lambda$. Let $x$ be a basis element in $\AA(\Lambda)$ and let $\Gamma \leq \Lambda$. If $x \neq 1$, then $x = \AA_*(\Theta \leq \Lambda)(y)$ for some $\Theta$ and $y$ a basis element in $\AA(\Theta)$. Now $\AA^*(\Gamma \leq \Lambda)(x)$ can be expressed in terms of transfers of restrictions of $y$. Thus the restriction of $x$ is determined by our induction hypothesis. Now if $x = 1$ in $\AA(\Lambda)$ then the restriction of $1$ is $1$ and $\iota(\Lambda)(1)=1$, taking care of that case. We conclude that $\iota$ commutes with restrictions.

To see that $\iota(\Lambda)$ is a ring map, assume that $\iota(\Gamma)$ is a ring map for all $\Gamma < \Lambda$. The base case is $\iota(\bar{X})$, which is a ring map. It suffices to prove that $\iota(\Lambda)$ is multiplicative on basis elements. Given basis elements $x,y \in \AA(\Lambda)$, we may assume without loss of generality that $x \neq 1$. Then $x = \AA_*(\Theta < \Lambda)(x')$ for some $\Theta < \Lambda$ and $x'$ a basis element in $\AA(\Theta)$. If $y' = \AA^*(\Theta < \Lambda)(y)$, then Frobenius reciprocity implies that
\[
\AA_*(\Theta < \Lambda)(y'x') = y\AA_*(\Theta < \Lambda)(x') = yx.
\]
Now $\iota(\Lambda)$ is multiplicative since $R$ also satisfies Frobenius reciprocity and $\iota(\Theta)$ is a ring map.

We leave any additional details to the reader.
\end{proof}

Let $RU(\Sigma_{\Lambda})$ be the complex representation ring of $\Sigma_{\Lambda}$. It is classical (see \cite[Sections 7.2 and 7.3]{Fulton}, for instance) that the composite
\[
\AA(\Lambda) \to A(\Sigma_{\Lambda}) \to RU(\Sigma_{\Lambda})
\]
is an isomorphism of commutative rings. Since both of these maps extend to maps of partition rings, the composite is an isomorphism of partition rings. Thus we may conclude that the initial partition ring is the representation ring partition ring sending a partition $\Lambda$ to $RU(\Sigma_{\Lambda})$. Since $RU(\Sigma_{(-)})$ is symmetric monoidal, \cref{cor:symmonhopf} implies that the component Hopf ring of symmetric functions $\bigoplus_{m\geq 0} RU(\Sigma_m)$ is the initial component Hopf ring. Two other categories in which symmetric functions are initial arise in the work of Hazewinkel \cite[Theorem 11.15]{HazWitt} on Hopf algebras equipped with a divided power sequence and the work of Baez, Moeller, and Trimble \cite{Baez} on $2$-plethories.

The power operations on Burnside rings $P_m \colon A(e) \to A(\Sigma_m)$ land inside $\AA(m)$. This is due to the fact that, if $X$ is a set, then the stabilizers of elements in $X^{m}$ equipped with the permutation action by $\Sigma_m$ are Young subgroups. This fact and equivariant generalizations of this fact are thoroughly explored in \cite{cornelius2024imagetotalpoweroperation}. It follows that $\AA$ is a partition power functor.

\begin{cor}
The partition power functor $\AA$ is the initial partition power functor.
\end{cor}
\begin{proof}
For any partition power functor $R$, the power operations applied to the image of the ring map $\Z \to R(1)$ are completely determined by transfers for the non-negative integers and determined inductively by expanding the right hand side of the formula $0 = P_m(0) = P_m(k+(-k))$ to find the value on negative integers. This follows from the third axiom in the definition of a partition power functor, which reduces the calculation of $P_m(k)$ for $k \in \N$ to transfers and the value of power operations on $1$.
\end{proof}

Since every basis element except for $1$ in $\AA(\Lambda)$ is in the image of a transfer map from a proper subpartition of $\Lambda$, we have
\[
\AA(\Lambda)/I_{\Lambda} \cong \Z.
\]
By construction, the quotient map
\[
\AA(\Lambda) \to \AA(\Lambda)/I_{\Lambda}
\]
takes a $\Sigma_{\Lambda}$-set $X$ representing a class in $\AA(\Lambda)$ to $|X^{\Sigma_{\Lambda}}|$. Since we have a canonical isomorphism of partition rings $\AA \cong RU(\Sigma_{(-)})$ and the $\Div$ monad is natural, it follows that
\[
\Div(\AA)(\Lambda) \cong \Cl(\Sigma_{\Lambda}, \Z)
\]
and the unit map $\AA(\Lambda) \to \Div(\AA)(\Lambda)$ is the character map
\[
RU(\Sigma_{\Lambda}) \to \Cl(\Sigma_{\Lambda},\Z).
\]
Further, \cref{prop:symmetricinvariant} and \cref{thm:freediv} imply that $\Cl(\Sigma_{(-)},\Z)[\Sigma]$ is a graded divided power polynomial algebra with one generator in each degree and also the divided power envelope of the ring of symmetric functions $RU(\Sigma_{(-)})[\Sigma]$ with respect to the ideal generated by positively graded elements. The divided power structure on integer-valued class functions was already observed in \cite{classfncdivpower}.

Note that, if $R$ is a $\Div$-algebra in partition rings, then there exists a canonical map of partition rings $\Cl(\Sigma_{(-)},\Z) \to R$.


\section{The universal exponential element} \label{sec:univexprel}

In this section we describe a symmetric monoidal partition ring $\cE$ that carries the universal exponential element. It follows that $\Div(\cE)$ carries the universal exponential relation in the category of $\Div$-algebras in partition rings. We also describe two automorphisms of $\cE$.

We will make use of Oda and Yoshida's construction of crossed Burnside rings \cite{Oda1,Oda2}, but applied to infinite (rather than finite) commutative monoids equipped with an action of $G$. Let $\Lambda = (X, \sim_{\Lambda})$ be a partition and let $M_{\Lambda}$ be the commutative monoid $\N^{X}$ equipped with the canonical $\Sigma_{\Lambda}$-action. A $\Sigma_{\Lambda}$-set over $M_{\Lambda}$ with Young stabilizers is a pair $(W, \alpha \colon W \to M_{\Lambda})$, where $W$ is a finite $\Sigma_{\Lambda}$-set with stabilizers all Young subgroups of $\Sigma_{\Lambda}$ and $\alpha$ is a  $\Sigma_{\Lambda}$-equivariant function. Two pairs $(W, \alpha)$ and $(V, \beta)$ are isomorphic if there is an isomorphism of $\Sigma_{\Lambda}$-sets $\sigma \colon W \to V$ such that $\alpha = \beta \sigma$. The disjoint union of $(W,\alpha)$ and $(V,\beta)$ is $(W \coprod V, \alpha \coprod \beta)$. The product of $(W,\alpha)$ and $(V,\beta)$ is $(W \times V, \alpha+\beta)$, where $(\alpha+\beta)(w,v) = \alpha(w)+\beta(v)$. Let $[W, \alpha]$ be the isomorphism class of $(W,\alpha)$. The disjoint union and product induce a well-defined sum and product of isomorphism classes. Let $\cE(\Lambda)$ be the Grothendieck ring of isomorphism classes of $\Sigma_{\Lambda}$-set over $M_{\Lambda}$ with Young stabilizers. Also note that $\cE(0) \cong \Z$ as $\N^{[0]} = *$. Note that $\AA(\Lambda)$ is the subring of $\cE(\Lambda)$ on isomorphism classes of the form $[W, c_0]$, where $c_0$ is the constant function at $0 \in \N^X$.

Next we will describe restriction and transfer maps giving $\cE$ the structure of a partition ring. For $[W,\alpha] \in \cE(\Lambda)$ and $\Gamma \leq \Lambda$, the restriction $\cE^*(\Gamma \leq \Lambda)([W,\alpha])$ is given by restricting all of the actions to $\Sigma_{\Gamma}$ (i.e., restricting $W$ to $\Sigma_{\Gamma}$ and viewing $\alpha$ as $\Sigma_{\Gamma}$-equivariant). For $[V,\beta] \in \cE(\Gamma)$ and $\Gamma \leq \Lambda$, the transfer $\cE_*(\Gamma \leq \Lambda)([V,\beta])$ is given by $[\Sigma_{\Lambda} \times_{\Sigma_{\Gamma}} V, [\sigma, v] \mapsto \sigma \beta(v)]$. Now assume that $\Omega = (Y, \sim_{\Omega})$. Given an isomorphism of partitions $f \colon \Omega \to \Lambda$ and a pair $(W,\alpha)$ as above, we may restrict $W$ along the induced map $\Sigma_{\Omega} \to \Sigma_{\Lambda}$ to obtain a $\Sigma_{\Omega}$-set and the underlying bijection $Y \xrightarrow{\cong} X$ gives an isomorphism $M_{\Lambda} \xrightarrow{\cong} M_{\Omega}$. These maps induce an isomorphism $\cE(\Lambda) \to \cE(\Omega)$. Restriction and transfer along arbitrary maps of partitions are obtained using the factorization of \cref{lemma:factorization}.

We define a symmetric monoidal structure on $\cE$ generalizing the symmetric monoidal structure on $\AA$. Given $[W,\alpha] \in \cE(\Lambda)$, where $\Lambda = (X, \sim_{\Lambda})$ and $[V,\beta] \in \cE(\Omega)$, where $\Omega = (Y,\sim_{\Omega})$, we set
\[
[W,\alpha] \boxtimes [V,\beta] = [W \times V, \alpha \times \beta \colon W \times V \to \N^X \times \N^Y \cong \N^{X \coprod Y}],
\]
where $W \times V$ is viewed as a $\Sigma_{\Lambda} \times \Sigma_{\Omega}$-set.

\begin{prop}
With the structure maps described above, $\cE$ has the structure of a symmetric monoidal partition ring.
\end{prop}
\begin{proof}
The partition ring structure follows from \cite[Sections 2 and 3]{Oda2}. Their proofs apply unchanged in this setting. Indeed, restriction and transfer preserve the condition on stabilizers since conjugates and intersections of Young subgroups are Young subgroups, and none of the constructions uses finiteness of the monoid.

It remains to consider the symmetric monoidal structure. Every refinement of
$\Lambda\sqcup\Omega$ is uniquely of the form $\Gamma\sqcup\Theta$, with
$\Gamma\leq\Lambda$ and $\Theta\leq\Omega$, and
\[
M_{\Lambda\sqcup\Omega}\cong M_\Lambda\times M_\Omega.
\]
It follows, by decomposing crossed sets into transitive crossed sets, that the
external product induces an isomorphism
\[
\cE(\Lambda)\otimes \cE(\Omega)\xrightarrow{\cong} \cE(\Lambda\sqcup\Omega).
\]
The compatibility of this product with restrictions and transfers is immediate
from the corresponding properties for Cartesian product and transfer.
\end{proof}

Let $e_{\Lambda} = [\Sigma_{\Lambda}/\Sigma_{\Lambda}, c_1 \colon \Sigma_{\Lambda}/\Sigma_{\Lambda} \to \N^X] \in \cE(\Lambda)$, where $c_1$ is the constant function at $1$ on $X$. Following \cref{not:indexing}, let $e_m = e_{\u m}$. We set $e_{0} = 1$. Note that 
\[
e_{\Lambda} \boxtimes e_{\Omega} = e_{\Lambda \sqcup \Omega}.
\] 
For $\Gamma \leq \Lambda$, consider the pair $(\Sigma_\Lambda/\Sigma_{\Gamma}, \alpha)$. The function $\alpha(e\Sigma_{\Gamma}) \in \N^{X}$ is constant on blocks. Let $a_B = \alpha(e\Sigma_{\Gamma})(B)$ for $B \in X/{\sim_\Gamma}$.

\begin{lemma} \label{lem:univelement}
Assume $\Lambda = (X,\sim_{\Lambda})$, $\Gamma \leq \Lambda$, and $[\Sigma_{\Lambda}/\Sigma_{\Gamma}, \alpha] \in \cE(\Lambda)$, then
\[
[\Sigma_{\Lambda}/\Sigma_{\Gamma}, \alpha] = \cE_*(\Gamma \leq \Lambda) \Big (\boxtimes_{B \in X/{\sim_\Gamma}} e_{\u B}^{a_B} \Big ).
\]
\end{lemma}
\begin{proof}
The element
\[
\boxtimes_{B\in X/{\sim_\Gamma}}e_{\u B}^{a_B}\in \cE(\Gamma)
\]
is represented by the one-point $\Sigma_\Gamma$-set $*$ and function $a \in \N^X$ with value $a_B$ on the block $B$.
The transfer of this crossed set from $\Sigma_\Gamma$ to $\Sigma_\Lambda$ gives
$\Sigma_\Lambda/\Sigma_\Gamma$ with function
\[
\sigma\Sigma_\Gamma\longmapsto \sigma a.
\]
Since $\alpha$ is equivariant and is determined by its value at
$e\Sigma_\Gamma$, this is precisely $\alpha$.
\end{proof}

Now consider the commutative ring $\cE\powser{\Sigma}$. For $i+j = m$, we have 
\[
\cE^*(\u i \sqcup \u j \leq \u m)(e_{m}) = e_{i} \boxtimes e_{j}
\]
so $\sum_{m \geq 0} e_{m} t^m$ is an exponential element in $\cE\powser{\Sigma}$. In fact, this is the universal exponential element:

\begin{theorem} \label{thm:univexpelt}
Let $R$ be a partition ring. There is a natural bijection between the set of maps of partition rings from $\cE$ to $R$ and the set of exponential elements in $R\powser{\Sigma}$ given by sending a map of partition rings
\[
f \colon \cE \to R
\]
to the exponential element $\sum_{m \geq 0} f_{m} (e_m) t^m$.
\end{theorem}
\begin{proof}
A map of partition rings $f\colon\cE\to R$ sends the elements $e_m$
to elements $r_m=f_m(e_m)$ satisfying $r_0=1$ and
\[
R^*(\u i\sqcup\u j\leq\u{i+j})(r_{i+j})
=r_i\boxtimes r_j.
\]
Thus $\sum_{m\geq0}r_mt^m$ is exponential.

Conversely, suppose that $\sum_{m\geq0}r_mt^m\in R\powser{\Sigma}$
is exponential. For a nonempty finite set $B$, let
$r_{\u B}\in R(\u B)$ correspond to $r_{|B|}$ under the canonical
identification of \cref{lem:earlyiso}. If $\Gamma$ is a partition
of $X$ and $a\in\N^X$ is constant on its blocks, write
\[
z_\Gamma(a)=\boxtimes_{B\in X/{\sim_\Gamma}}r_{\u B}^{a_B}
\qquad\text{and}\qquad
\alpha_a(g\Sigma_\Gamma)=ga.
\]
By \cref{lem:univelement}, a map sending $e_m$ to $r_m$ must satisfy
\[
f_\Lambda([\Sigma_\Lambda/\Sigma_\Gamma,\alpha_a])
=R_*(\Gamma\leq\Lambda)z_\Gamma(a).
\]
We use this formula to define $f_\Lambda$ on the transitive crossed
sets and extend additively.

The construction of $z_\Gamma(a)$ is compatible with relabelling:
for an isomorphism $\sigma\colon\Gamma\to\sigma\Gamma$, we have
\[
R_*(\sigma)z_\Gamma(a)=z_{\sigma\Gamma}(\sigma a).
\]
Changing the representative of a transitive crossed set conjugates
$\Gamma$ and transports $a$ by an element of $\Sigma_\Lambda$.
Since such an element induces the identity on $R(\Lambda)$,
the formula for $f_\Lambda$ is independent of this choice. It also
shows that the maps $f_\Lambda$ commute with isomorphisms.

The identities satisfied by the $r_m$'s describe what happens when a block is
refined. If $\Delta\leq\Gamma$, repeated application of these
identities, using that restrictions preserve internal and external
multiplication, gives
\[
R^*(\Delta\leq\Gamma)z_\Gamma(a)=z_\Delta(a).
\]
Here each block of $\Delta$ inherits the label of the block of
$\Gamma$ containing it. Together with the double coset formula,
this identity shows that $f$ commutes with restrictions.
Compatibility with transfers follows from transitivity of transfers,
and compatibility with the external product follows from the
definition of $z_\Gamma(a)$ and the compatibility of transfers
with external products.

It remains to check that $f_\Lambda$ is multiplicative. We will use $*$ for the multiplication in $\cE(\Lambda)$ and $R(\Lambda)$. Let
\[
x=[\Sigma_\Lambda/\Sigma_\Gamma,\alpha_a],
\qquad
y=[\Sigma_\Lambda/\Sigma_\Omega,\alpha_b].
\]
Choose representatives $\sigma$ for
$\Sigma_\Gamma\backslash\Sigma_\Lambda/\Sigma_\Omega$ and set
$\Delta_\sigma=\Gamma\cap\sigma\Omega$.
The orbit of $(\Sigma_\Gamma,\sigma\Sigma_\Omega)$ in the product
of the two coset spaces has stabilizer $\Sigma_{\Delta_\sigma}$
and label $a+\sigma b$. Thus
\[
x*y=\sum_{[\sigma]}
[\Sigma_\Lambda/\Sigma_{\Delta_\sigma},\alpha_{a+\sigma b}].
\]
Frobenius reciprocity and the double coset formula express
$f_\Lambda(x)*f_\Lambda(y)$ as a sum over the same double cosets,
by restricting the two factors to $\Delta_\sigma$ and transferring
their product to $\Lambda$. The restriction and relabelling
identities above identify these two factors with
$z_{\Delta_\sigma}(a)$ and $z_{\Delta_\sigma}(\sigma b)$.
Since external multiplication is a ring map, their product is
$z_{\Delta_\sigma}(a+\sigma b)$. Consequently,
\[
\begin{aligned}
f_\Lambda(x)*f_\Lambda(y)
&=\sum_{[\sigma]}R_*(\Delta_\sigma\leq\Lambda)
  \bigl(z_{\Delta_\sigma}(a)*z_{\Delta_\sigma}(\sigma b)\bigr)\\
&=\sum_{[\sigma]}R_*(\Delta_\sigma\leq\Lambda)
  z_{\Delta_\sigma}(a+\sigma b)\\
&=f_\Lambda(x*y).
\end{aligned}
\]
The multiplicative identity is
$[\Sigma_\Lambda/\Sigma_\Lambda,\alpha_0]$, whose image is
$z_\Lambda(0)=1$. Hence $f\colon\cE\to R$ is a map of partition
rings.

By construction, $f_m(e_m)=r_m$, and \cref{lem:univelement} shows
that these values determine $f$ uniquely. The two constructions
are therefore inverse and natural in $R$.
\end{proof}

It follows that $\Div(\cE)$ carries the universal exponential relation in the category of $\Div$-algebras in partition rings. That is, if the partition ring $R$ is a $\Div$-algebra with structure map $\iota \colon \Div(R) \to R$ and $\sum_{m \geq 0} r_mt^m$ is an exponential element in $R\powser{\Sigma}$ with exponential relation
\[
\sum_{m \geq 0} r_mt^m = \exp\Big(\sum_{k \geq 1} \iota_k(r_k * \1_k) t^k\Big)
\]
then there exists a unique map of $\Div$-algebras $\Div(\cE) \to R$ sending the exponential relation carried by $\Div(\cE)\powser{\Sigma}$ to the exponential relation over $R$. 

We can also describe the universal exponential relation. Note that $\bar{\cE}(m) \cong \Z[x_m]$ for $m > 0$, where $x_m = q_m(e_m) = q_m([\Sigma_m/\Sigma_m, c_1])$. Viewing $\bar{\cE}(m) \subseteq \Div(\cE)(m)$, the universal exponential relation in $\cE\divpowser{\Sigma}$ is 
\[
\sum_{m \geq 0} e_mt^m = \exp\Big(\sum_{k \geq 1} x_k t^k\Big).
\]

The partition ring $\cE$ admits nontrivial automorphisms. We will describe two of them here and explore the full automorphism group as well as further algebraic structure possessed by the partition ring $\cE$ and the commutative ring $\cE[\Sigma]$ in future work. Let $u(t) = \sum_{m \geq 0} e_m t^m$ be the universal exponential element in $\cE\powser{\Sigma}$. Note that $u(-t)$ is also an exponential element. Since $u(t)$ is a group-like element of $\cE\powser{\Sigma}$ (see \cref{rem:grouplike}), so is $u(t)^{-1}$ (defined in the expected way) and this implies that $u(t)^{-1}$ is exponential as well. The exponential elements $u(-t)$ and $u(t)^{-1}$ correspond to involutions of $\cE$.

There is a useful description of the composite of these two involutions. By the preceding section, $RU(\Sigma_{(-)})$ is the initial partition ring. Thus, for every partition ring $R$, there are canonical elements
\[
s_m \in R(m)
\]
given by the images of the sign representations $[\mathrm{sgn}_m] \in RU(\Sigma_m)$ under the unique map of partition rings $RU(\Sigma_{(-)}) \to R$. We refer to $s_m$ as the sign element in degree $m$.

\begin{prop} \label{prop:signinvolution}
Let $R$ be a partition ring and let $r(t)=\sum_{m \geq 0} r_m t^m \in R\powser{\Sigma}$ be an exponential element. Then
\[
r(-t)^{-1}=\sum_{m \geq 0}(s_m*r_m)t^m.
\]
\end{prop}

\begin{proof}
Let $1_m \in RU(\Sigma_m)$ and $1_{m}^R \in R(m)$ be the multiplicative identities. We have the classical relation (see \cite[Section 6.1]{Fulton})
$$
\bigg(\sum_{m\geq0}(-1)^m1_mt^m\bigg)
\bigg(\sum_{m\geq0}[\mathrm{sgn}_m]t^m\bigg)=1
$$
in the completed ring of symmetric functions $RU(\Sigma_{(-)})\powser{\Sigma}$. Applying the unique map $RU(\Sigma_{(-)})\to R$ gives
$$
\bigg(\sum_{m\geq0}(-1)^m1_m^Rt^m\bigg)
\bigg(\sum_{m\geq0}s_mt^m\bigg)=1.
$$

The degree $m$ coefficient of
$$
r(-t)\bigg(\sum_{j\geq0}(s_j*r_j)t^j\bigg)
$$
is zero for $m>0$ and is $1$ for $m=0$.
Indeed, the degree $m$ coefficient is
\[
\sum_{i+j=m}(-1)^i
R_*(\u i\sqcup \u j\leq \u m)
\Bigl(m_{i,j}\bigl(r_i\otimes(s_j*r_j)\bigr)\Bigr).
\]
Since $m_{i,j}$ is multiplicative and $r(t)$ is exponential, we have
\[
m_{i,j}\bigl(r_i\otimes(s_j*r_j)\bigr)
=
R^*(\u i\sqcup\u j\leq\u m)(r_m)
*
m_{i,j}(1_i^R\otimes s_j).
\]
Thus Frobenius reciprocity gives
\begin{align*}
\sum_{i+j=m}(-1)^i
R_*(\u i\sqcup\u j\leq\u m)  \Bigl(
R^*(\u i\sqcup\u j\leq\u m)&(r_m)
*
m_{i,j}(1_i^R\otimes s_j)
\Bigr) \\
= \,
&r_m*
\Big ( \sum_{i+j=m}(-1)^i
R_*(\u i\sqcup\u j\leq\u m)
m_{i,j}(1_i^R\otimes s_j) \Big ).
\end{align*}
The final sum is the degree $m$ coefficient of
\[
\bigg(\sum_{i\geq0}(-1)^i1_i^Rt^i\bigg)
\bigg(\sum_{j\geq0}s_jt^j\bigg).
\]
By the preceding identity, this coefficient is $0$ for $m>0$ and $1$ for $m=0$.
\end{proof}

\begin{example}
Assume that $R$ is a partition power functor. In this case, the sign elements have an intrinsic description in terms of the power operations. Indeed, the canonical map from the initial partition power functor to $R$ preserves power operations, while in the initial partition power functor
\[
\sum_{m\geq0}P_m(-1)t^m
=
\Big (\sum_{m\geq0}P_m(1)t^m \Big )^{-1}
=
\sum_{m\geq0}(-1)^m[\mathrm{sgn}_m]t^m.
\]
Consequently,
\[
s_m=(-1)^mP_m(-1),
\]
and multiplicativity of $P_m$ gives
\[
s_m*P_m(x)=(-1)^mP_m(-x).
\]
Thus
\begin{equation}
\sum_{m\geq0}(s_m*P_m(x))t^m
=
\Big (\sum_{m\geq0}(-1)^mP_m(x)t^m\Big )^{-1}.
\end{equation}
\end{example}


\section{Obtaining classical exponential relations} \label{sec:classicalexp}

To obtain the exponential relations that we are interested in, we would liked weaker conditions on the map $R \to R'$ that allow us to push exponential relations from $R$ to $R'$. In this section, we provide weaker conditions on a collection of maps $R(\Lambda) \to C^K(\Lambda) = K$ that allow us to construct a map of commutative $\N$-graded rings $R\langle \Sigma \rangle \to K\langle t \rangle$. Completion then provides us with a map of commutative rings $R\divpowser{\Sigma} \to K\divpowser{t}$. Any exponential relation holding in the source then gives a classical exponential relation in the target.

First we notice that the conditions involving restriction maps are not necessary on a map between lax symmetric monoidal partition functors in order to get an induced map of commutative $\N$-graded rings.

\begin{prop} \label{prop:transferring}
Assume $R$ and $R'$ are lax symmetric monoidal partition functors and that, for each partition $\Lambda$, we have a map of abelian groups $f_{\Lambda} \colon R(\Lambda) \to R'(\Lambda)$ commuting with transfer maps and the lax symmetric monoidal structure maps, then there is an induced map of commutative $\N$-graded rings
\[
f[\Sigma] \colon R[\Sigma] \to R'[\Sigma].
\]
\end{prop}

We are particularly interested in obtaining classical exponential relations. Thus we are interested in conditions on maps $R(\Lambda) \to C^K(\Lambda) = K$, for each partition $\Lambda$, that give rise to a map of commutative $\N$-graded rings
\[
R\langle \Sigma \rangle \to K\langle t \rangle.
\]

\begin{prop}\label{prop:transfer-augmentation}
Let $R$ be a lax symmetric monoidal partition functor and let $K$ be a
commutative ring. Suppose that, for every partition
$\Lambda$, there is a map of abelian groups
\[
\epsilon_\Lambda \colon R(\Lambda)\longrightarrow C^K(\Lambda) = K
\]
such that the maps $\epsilon_\Lambda$ commute with transfers and with
the lax symmetric monoidal structure maps. In this case, we construct canonical maps of abelian groups
\[
\widetilde\epsilon_\Lambda \colon
\Div(R)(\Lambda)\longrightarrow \Div(C^K)(\Lambda) \cong C_K(\Lambda) = K
\]
compatible with transfer maps, the lax symmetric monoidal structure, and the unit for the monad $\Div$ so that $\tilde{\epsilon}_{\Lambda} \eta_{\Lambda} = |\Sigma_{\Lambda}|\epsilon_{\Lambda}$.



\end{prop}

\begin{proof}
We construct $\tilde \epsilon_{\Lambda}$ out of maps $\ell_{\Lambda} \colon \bar{R}(\Lambda) \to K$ characterized by the identities
\[
|\Sigma_\Lambda|\epsilon_\Lambda(x)
=
\sum_{\Gamma\leq\Lambda}
\ell_\Gamma
\left(
q_\Gamma R^*(\Gamma\leq\Lambda)(x)
\right)
\]
for $x\in R(\Lambda)$. The maps $\ell_\Lambda$ will be built inductively over the poset of
partitions. 

If $\Lambda$ is discrete, then
$\bar{R}(\Lambda)=R(\Lambda)$ and we set
\[
\ell_\Lambda=\epsilon_\Lambda.
\]
Suppose that $\ell_\Gamma$ has been constructed for every proper
subpartition $\Gamma<\Lambda$. For $x\in R(\Lambda)$, set
\[
\ell_\Lambda(x)
=
|\Sigma_\Lambda|\epsilon_\Lambda(x)
-
\sum_{\Gamma<\Lambda}
\ell_\Gamma
\left(
q_\Gamma R^*(\Gamma < \Lambda)(x)
\right).
\]
Note that $\ell_\Lambda$ is compatible with isomorphisms. This is immediate for the discrete partition as $\epsilon_{\Lambda}$ is compatible with all transfers. Transporting the defining formula above along an isomorphism sends proper subpartitions to proper subpartitions, so the naturality of $\epsilon_{\Lambda}$ again gives that $\ell_{\Lambda}$ is compatible with isomorphisms.

We claim that $\ell_\Lambda$ vanishes on the transfer subgroup
$I_\Lambda$ and thus induces a map
\[
\ell_\Lambda \colon \bar{R}(\Lambda)\longrightarrow K.
\]
To check this, let $\Omega<\Lambda$ and $y\in R(\Omega)$. As in the proof of \cref{lemma:unitMapForDivMonad}, after applying the
double-coset formula to
\[
R^*(\Gamma\leq\Lambda)R_*(\Omega<\Lambda)(y)
\]
and then applying $q_\Gamma$, every term vanishes unless
$\Gamma\leq\sigma\Omega$. Consequently,
\begin{align} \label{eq:double}
q_\Gamma R^*(\Gamma\leq\Lambda)
R_*(\Omega<\Lambda)(y)
&=
\sum_{\substack{
[\sigma]\in
\Sigma_\Gamma\backslash\Sigma_\Lambda/\Sigma_\Omega\\
\Gamma\leq\sigma\Omega
}}
\bar{R}_*(\sigma)
q_{\sigma^{-1}\Gamma}
R^*(\sigma^{-1}\Gamma\leq\Omega)(y).
\end{align}
Using this, we have the equalities explained below:
\begin{align*}
\sum_{\Gamma<\Lambda}
\ell_\Gamma\Bigg(
q_\Gamma R^*(\Gamma\leq\Lambda)
R_*(\Omega<\Lambda)(y)
\Bigg) &= \sum_{\Gamma<\Lambda}
\ell_\Gamma\Bigg(\sum_{\substack{
[\sigma]\in
\Sigma_\Gamma\backslash\Sigma_\Lambda/\Sigma_\Omega\\
\Gamma\leq\sigma\Omega
}}
\bar{R}_*(\sigma)
q_{\sigma^{-1}\Gamma}
R^*(\sigma^{-1}\Gamma\leq\Omega)(y)
\Bigg) \\
&= \sum_{\Gamma<\Lambda} \sum_{\substack{
[\sigma]\in
\Sigma_\Gamma\backslash\Sigma_\Lambda/\Sigma_\Omega\\
\Gamma\leq\sigma\Omega
}} \ell_{\sigma^{-1}\Gamma} q_{\sigma^{-1}\Gamma}
R^*(\sigma^{-1}\Gamma\leq\Omega)(y) \\
&= \sum_{[\sigma] \in \Sigma_\Lambda/\Sigma_\Omega} \sum_{\Gamma \leq \sigma \Omega} \ell_{\sigma^{-1}\Gamma} q_{\sigma^{-1}\Gamma}
R^*(\sigma^{-1}\Gamma\leq\Omega)(y) \\
&= \sum_{[\sigma] \in \Sigma_\Lambda/\Sigma_\Omega} \sum_{\Theta \leq \Omega} \ell_{\Theta} q_{\Theta}
R^*(\Theta \leq\Omega)(y) \\
&= |\Sigma_\Lambda/\Sigma_\Omega| \sum_{\Theta \leq \Omega} \ell_{\Theta} q_{\Theta}
R^*(\Theta \leq\Omega)(y).
\end{align*}
The first equality follows from \cref{eq:double}. The second equality follows from the compatibility of $\ell_{\Gamma}$ with isomorphisms. The third equality follows from the fact that a double coset $[\sigma] \in \Sigma_\Gamma\backslash\Sigma_\Lambda/\Sigma_\Omega$ such that $\Gamma \leq \sigma \Omega$ has the property that $\Sigma_{\Gamma} \subseteq \Sigma_{\Omega}^{\sigma}$ and thus $\Sigma_{\Gamma} \sigma \Sigma_{\Omega} = \sigma \Sigma_{\Omega}$, providing a bijection with the set of cosets $\Sigma_\Lambda / \Sigma_\Omega$ such that $\Gamma \leq \sigma \Omega$. The fourth equality comes from setting $\Theta = \sigma^{-1} \Gamma$. The fifth equality follows from the fact that the internal sum is independent of $[\sigma]$.

The induction hypothesis identifies the right hand side with
\[
|\Sigma_\Lambda/\Sigma_\Omega|
|\Sigma_\Omega|\epsilon_\Omega(y)
=
|\Sigma_\Lambda|\epsilon_\Omega(y).
\]
Since the maps $\epsilon_\Lambda$ commute with transfers,
\[
\epsilon_\Lambda
R_*(\Omega\leq\Lambda)(y)
=
\epsilon_\Omega(y).
\]
It follows that
\[
\ell_\Lambda
R_*(\Omega<\Lambda)(y)=0,
\]
and hence $\ell_\Lambda$ descends to $\bar{R}(\Lambda)$.

M\"obius inversion gives the explicit formula
\[
\ell_\Lambda(q_\Lambda x)
=
\sum_{\Gamma\leq\Lambda}
\mu(\Gamma,\Lambda)
|\Sigma_\Gamma|
\epsilon_\Gamma
\left(
R^*(\Gamma\leq\Lambda)(x)
\right),
\]
where $\mu$ is the M\"obius function of the lattice of subpartitions of
$\Lambda$. Note that, if $\Gamma = ([m], \sim_{\Gamma})$ and $n = |[m]/{\sim_{\Gamma}}|$ and $m>0$, then $\mu(\Gamma, \u m) = (-1)^{n-1}(n-1)!$.

Now define
\[
\widetilde\epsilon_\Lambda
\bigl((x_\Gamma)_{\Gamma\leq\Lambda}\bigr)
=
\sum_{\Gamma\leq\Lambda}\ell_\Gamma(x_\Gamma).
\]
The compatibility of the maps $\ell_\Gamma$ with isomorphisms shows
that this is well-defined on the $\Sigma_\Lambda$-invariants. The same
double-coset reindexing used above gives, for
$\Omega\leq\Lambda$,
\[
\widetilde\epsilon_\Lambda
\Div(R)_*(\Omega\leq\Lambda)
=
|\Sigma_\Lambda/\Sigma_\Omega|
\widetilde\epsilon_\Omega.
\]
Thus $\widetilde{\epsilon}$ is compatible with the transfers in $C_K$.

Next we check the lax symmetric monoidal multiplication. We wish to show that
\[
\widetilde{\epsilon}_{\Lambda \sqcup \Omega} \Div(m)_{\Lambda, \Omega}((u_{\Gamma})_{\Gamma \leq \Lambda} \otimes (v_{\Theta})_{\Theta \leq \Omega}) = \widetilde{\epsilon}_{\Lambda}((u_{\Gamma})_{\Gamma \leq \Lambda}) \widetilde{\epsilon}_{\Omega}((v_{\Theta})_{\Theta \leq \Omega}).
\]
Recall that every
subpartition of $\Lambda\sqcup\Omega$ is uniquely of the form
$\Gamma\sqcup\Theta$, with $\Gamma\leq\Lambda$ and
$\Theta\leq\Omega$. Using that $\widetilde{\epsilon}_{\Lambda}((u_{\Gamma})) = \sum_{\Gamma \leq \Lambda} \ell_{\Gamma}(u_\Gamma)$ (and similarly for $\Omega$ and $(v_{\Theta})$), it suffices to show that
\[
\ell_{\Gamma\sqcup\Theta}
\bigl(\bar m_{\Gamma,\Theta}(u\otimes v)\bigr)
=
\ell_\Gamma(u)\ell_\Theta(v),
\]
for $u \in \bar{R}(\Gamma)$ and $v \in \bar{R}(\Theta)$.

Indeed, choose lifts $x\in R(\Gamma)$ and $y\in R(\Theta)$ of $u$ and
$v$. Applying the defining identity for $\ell_{\Gamma\sqcup\Theta}$ to
$m_{\Gamma,\Theta}(x\otimes y)$, and using the fact that every
subpartition of $\Gamma\sqcup\Theta$ is uniquely of the form
$\Phi\sqcup\Psi$, gives a sum indexed by $\Phi\leq\Gamma$ and
$\Psi\leq\Theta$. On the other hand, multiplying the defining identities for
$\ell_\Gamma$ and $\ell_\Theta$ gives a second sum indexed by the same
pairs $(\Phi,\Psi)$. The two sums have the same value, using
$|\Sigma_{\Gamma\sqcup\Theta}|=|\Sigma_\Gamma||\Sigma_\Theta|$ and the
compatibility of $\epsilon$ with the lax symmetric monoidal structure.
By induction, the corresponding terms agree whenever
$(\Phi,\Psi)\neq(\Gamma,\Theta)$, so the remaining terms agree as well.
Hence
\[
\ell_{\Gamma\sqcup\Theta}
\bigl(\bar m_{\Gamma,\Theta}(u\otimes v)\bigr)
=
\ell_\Gamma(u)\ell_\Theta(v).
\]


Finally, for $x\in R(\Lambda)$, the defining identity gives
\[
\widetilde\epsilon_\Lambda\eta_\Lambda(x)
=
\sum_{\Gamma\leq\Lambda}
\ell_\Gamma
\left(
q_\Gamma R^*(\Gamma\leq\Lambda)(x)
\right)
=
|\Sigma_\Lambda|\epsilon_\Lambda(x).
\]
By \cref{ex:monadconst}, the unit
\[
C^K(\Lambda)\longrightarrow\Div(C^K)(\Lambda) \cong C_K(\Lambda)
\]
is multiplication by $|\Sigma_\Lambda|$, so $\tilde \epsilon_{\Lambda}$ is compatible with the monadic units. 
\end{proof}

It follows that, under the conditions of \cref{prop:transfer-augmentation}, we get a map of commutative $\N$-graded rings
\[
\tilde{\epsilon} \colon R\langle \Sigma \rangle \to C_K[\Sigma] \cong K\langle t \rangle
\]
under $R[\Sigma]$. The isomorphism on the right is described in \cref{ex:constantring}; it takes $k \in C_K(m)=K$ to $k\frac{t^m}{m!}$. After completion, we get a map of commutative rings
\[
\tilde{\epsilon} \colon R\divpowser{\Sigma} \to K\divpowser{t}.
\]
If $\sum_{m \geq 0} r_m t^m \in R\powser{\Sigma}$, then
\[
\sum_{m \geq 0} \tilde{\epsilon}_{m} \eta_{m} r_m \frac{t^m}{m!} = \sum_{m \geq 0} m!\epsilon_{m}r_m \frac{t^m}{m!} = \sum_{m \geq 0} \epsilon_{m}r_m t^m
\]
in $K\divpowser{t}$.

Since the maps $\epsilon_{\Lambda}$ were not assumed to commute with restrictions, the map $\tilde{\epsilon}$ need not preserve exponential elements, however it will preserve exponential relations. That is, if $\sum_{m \geq 0} r_m t^m$ is an exponential element of $R\powser{\Sigma}$, then we get the exponential relation
\begin{equation} \label{eq:transferexp}
\sum_{m \geq 0} \epsilon_{m}r_m t^m = \exp \bigg (\sum_{k \geq 1} \ell_{k}(q_{k} r_k) \frac{t^k}{k!} \bigg )
\end{equation}
in $K\divpowser{t}$. Note that this holds assuming that $R$ is just lax symmetric monoidal: the multiplicativity of $\ell_{\Lambda}$ implies that 
\[
m! \epsilon_{m}r_m = \sum_{\Gamma \leq \u m} \prod_{B \in [m]/\sim_{\Gamma}} \ell_{|B|}q_{|B|} r_{|B|}
\]
and this is the coefficient of $\frac{t^m}{m!}$ on the right hand side of \cref{eq:transferexp}.

\begin{example} \label{ex:epsilonpower}
Assume that $R$ is a partition power functor. \cref{ex:exppower} furnishes us with the exponential relation
\[
\sum_{m \geq 0} P_m t^m = \exp \bigg (\sum_{k \geq 1} P_k/I_k t^k \bigg )
\]
in $R\divpowser{\Sigma}$. Now assume that the maps $\epsilon_{\Lambda} \colon R(\Lambda) \to C^K(\Lambda)$ satisfy the hypothesis of \cref{prop:transfer-augmentation}. Since $q_{k}P_k = \bar{P}_k$, we have $\ell_{k}q_{k}P_k = \tilde{\epsilon}_{k} P_k/I_k$. Thus, in this case, \cref{eq:transferexp} gives the exponential relation
\[
\sum_{m \geq 0} \epsilon_{m} P_m t^m = \exp \bigg ( \sum_{k \geq 1} \tilde{\epsilon}_{k}P_k/I_k\frac{t^k}{k!} \bigg )
\]
in $K\divpowser{t}$.
\end{example}


\section{Partition functors and ambidexterity} \label{sec:globGreen}

In this section we apply \cref{prop:transfer-augmentation} to partition rings restricted from global Green functors that admit transfers along surjections of finite groups as well as injections.



We say that a global Green functor is ambidextrous if it admits transfer maps along arbitrary maps of finite groups, rather than just injections. These transfers are required (as usual) to satisfy a double coset formula and Frobenius reciprocity (see \cite{Ganter-global} for more details).

Assume that $S$ is an ambidextrous global Green functor. Let $R = S \circ (G \times \Sigma_{(-)})$ for a fixed group $G$, then $R$ is a partition ring with extra structure. In particular there are ``transfer'' maps $\epsilon_{\Lambda}^2 \colon R(\Lambda) \to R(2)$ and $\epsilon_\Lambda \colon R(\Lambda) \to R(1)$ induced by the transfers in $S$ along the map $G \times \Sigma_{\Lambda} \to G \times \Sigma_2$, induced by the sign homomorphism, and $G \times \Sigma_{\Lambda} \to G \times \Sigma_1$, the projection onto $G$. 

Since $\Sigma_2$ is an abelian group, the multiplication map 
\[
\Sigma_2 \times \Sigma_2 \to \Sigma_2
\]
is a group homomorphism. This induces a group homomorphism $G \times \Sigma_2 \times \Sigma_2 \to G \times \Sigma_2$. Transfer along this homomorphism in turn induces a map
\[
R(2) \otimes_{R(0)} R(2) \to R(2)
\]
giving $R(2)$ an exotic ring structure that we will denote $R^{\tr}(2)$. We call the multiplication on this ring the convolution product. This ring is an $R(0)$-algebra via the transfer map in $S$ along $G \to G \times \Sigma_2$. Thus, the multiplicative unit in $R^{\tr}(2)$ is the image of $1 \in R(0)$ along this transfer.

\begin{remark}
Note that the same construction does not work for $G \wr \Sigma_{2}$ since the multiplication map on $\Sigma_2$ does not induce a group homomorphism $G \wr \Sigma_{2} \times G \wr \Sigma_{2} \to G \wr \Sigma_{2}$.
\end{remark}




\begin{prop} \label{prop:epsiloncompatible}
Assume $S$ is an ambidextrous global Green functor and fix a group $G$. Let $R = S \circ (G \times \Sigma_{(-)})$ and let $C^{R^{\tr}(2)}$ be the dual constant partition functor with value the commutative ring $R^{\tr}(2)$. The maps
\[
\epsilon^{2}_\Lambda \colon R(\Lambda) \to C^{R^{\tr}(2)}(\Lambda)
\]
commute with transfers and the lax symmetric monoidal multiplication.
\end{prop}
\begin{proof}
Let $\Omega$ be a partition and let $\Sigma_{\Omega} \to \Sigma_2$ be the sign map. Given a map of partitions $\Lambda \to \Omega$, the induced map $\Sigma_{\Lambda} \to \Sigma_{\Omega} \to \Sigma_2$ is the sign map for $\Sigma_{\Lambda}$. The maps $\epsilon^{2}_\Lambda$ commute with transfers because transfer maps compose to give transfer maps. That is, there is a commutative diagram
\[
\xymatrix{S(G \times \Sigma_{\Lambda}) \ar[r] \ar[d] & S(G \times \Sigma_2) \ar[d]^{=} \\ S(G \times \Sigma_{\Omega}) \ar[r] & S(G \times \Sigma_2),}
\]
where all of the maps are transfer maps.

To see that $\epsilon^2$ respects the lax symmetric monoidal structure, note that we have a commutative diagram
\[
\xymatrix{S(G \times \Sigma_{\Lambda}) \otimes_{S(G)} S(G \times \Sigma_{\Omega}) \ar[r] \ar[d] & S(G \times \Sigma_2) \otimes_{S(G)} S(G \times \Sigma_2) \ar[d] \\ S(G \times \Sigma_{\Lambda} \times \Sigma_{\Omega}) \ar[r] \ar[rd] & S(G \times \Sigma_2\times \Sigma_2)  \ar[d] \\ & S(G \times \Sigma_2).}
\]
The top square commutes by the naturality of the lax symmetric monoidal structure in pairs of transfer maps. The commutativity of the bottom triangle, which consists of transfer maps, follows from the fact that the composite
\[
\Sigma_{\Lambda} \times \Sigma_{\Omega} \to \Sigma_2 \times \Sigma_2 \to \Sigma_2,
\]
where the second arrow is the multiplication map, is the sign map for $\Sigma_{\Lambda} \times \Sigma_{\Omega}$.

\end{proof}

The compatibilities satisfied by $\epsilon^2$ in \cref{prop:epsiloncompatible} allow us to apply \cref{prop:transfer-augmentation} to obtain:



\begin{prop} \label{prop:exptransferSigma2}
Assume that $S$ is an ambidextrous global Green functor and that $R$ is the partition ring $R = S(G \times \Sigma_{(-)})$. Let $\epsilon^2 \colon R[\Sigma] \to R^{\tr}(2)[t]$ be as described above. The composite
\[
R[\Sigma] \xrightarrow{\epsilon^2} R^{\tr}(2)[t] \to R^{\tr}(2)\langle t \rangle 
\]
extends canonically to a map of commutative $\N$-graded $R(0)$-algebras
\[
\tilde{\epsilon}^2 \colon R \langle \Sigma \rangle \to R^{\tr}(2) \langle t \rangle.
\]
\end{prop}


Completing this map with respect to the system of irrelevant ideals, gives a map of commutative $R(0)$-algebras
\[
\tilde{\epsilon}^2 \colon R \divpowser{\Sigma} \to R^{\tr}(2)\divpowser{t}. 
\]
We will continue to call this map $\tilde{\epsilon}^2$.

With the same setup, let 
\[
\epsilon_{\Lambda} \colon R(\Lambda) \to R(1)
\]
be induced by the transfer map $S(G \times \Sigma_{\Lambda}) \to S(G)$ induced by the map to the trivial group $\Sigma_\Lambda \to e$. Since transfer maps compose to give transfer maps, $\epsilon_{\Lambda} = \epsilon_{2} \circ \epsilon^{2}_\Lambda$. It follows that the maps $\epsilon_\Lambda$ commute with transfers as a map $R(\Lambda) \to C^{R(1)}(\Lambda)$.

\begin{cor} \label{cor:ambitriv}
Assume that $S$ is an ambidextrous global Green functor and that $R$ is the partition ring $R = S(G \times \Sigma_{(-)})$. The composite
\[
R[\Sigma] \xrightarrow{\epsilon} R(1)[t] \to R(1)\langle t \rangle 
\]
extends canonically to a map of commutative $\N$-graded $R(0)$-algebras
\[
\tilde{\epsilon} \colon R \langle \Sigma \rangle \to R(1) \langle t \rangle
\]
that induces a map of $R(0)$-algebras
\[
\tilde{\epsilon} \colon R \divpowser{\Sigma} \to R(1)\divpowser{t}.
\]
\end{cor}

\begin{remark}
There is no reason to expect the transfer maps $\epsilon_\Lambda$ to induce a map of partition functors $R \to C^{R(1)}$. They do not interact appropriately with restriction maps. This already goes wrong for $RU(\Sigma_{(-)})$ and $\Lambda = \u 2$ as the transfer map
\[
\epsilon_{2} \colon RU(\Sigma_2) \cong \Z[x]/(x^2-1) \to RU(e) \cong \Z = C^{\Z}(2)
\]
sends $x$ to $0$, where $x$ is the isomorphism class of the sign representation of $\Sigma_2$. On the other hand the restriction map $RU(\Sigma_2) \to RU(\Sigma_{1})$ sends $x$ to $1$, so the required square, induced by $\u 1 \sqcup \u 1 \to \u 2$, 
\[
\xymatrix{RU(\Sigma_{2}) \ar[r]^-{\epsilon_{2}} \ar[d]_-{RU^*(\u 1 \sqcup \u 1 \leq \u 2)} & \Z \ar[d]^{\times 2} \\ \Z \ar[r]_{\id} & \Z}
\]
does not commute.
\end{remark}





If $S$ is an ambidextrous global power functor and $R = S \circ (G \times \Sigma_{(-)})$, then we have the exponential relation of \cref{ex:exppower}
\[
\sum_{m \geq 0} P_m t^m = \exp \Big(\sum_{k \geq 1} P_k/I_k t^k\Big)
\]
in $R \divpowser{\Sigma}$. By \cref{ex:epsilonpower}, applying $\tilde \epsilon^2$ to this relation gives the exponential relation
\[
\sum_{m \geq 0} \epsilon_{m}^2 P_m t^m = \exp\Big(\sum_{k \geq 1}  \tilde \epsilon_{k}^2 P_k/I_k \frac{t^k}{k!}\Big)
\]
in $R^{\tr}(2)\divpowser{t}$. Recall that we are using the identity $\tilde \epsilon_{m}^2 \eta_{m} P_m = m! \epsilon_{m}^2 P_m$, where the unit $\eta$ is suppressed in the exponential relation. Using the ring map $R^{\tr}(2) \to R(1)$ (induced by the transfer), we obtain the exponential relation 
\begin{equation} \label{eq:epsilonpower}
\sum_{m \geq 0} \epsilon_{m} P_m t^m = \exp\Big(\sum_{k \geq 1} \tilde \epsilon_{k} P_k/I_k \frac{t^k}{k!}\Big)
\end{equation}
in $R(1)\divpowser{t}$.




\addtocounter{secnumdepth}{1}
\counterwithin{theorem}{subsection}
\section{Examples} \label{sec:examples}

In this section we illustrate the preceding theory through several examples that exhibit differing features of partition functors and the monad $\Div$. Structured ring spectra provide a large natural source of partition power functors: the power operations of an $H_\infty$-ring spectrum satisfy the universal exponential relation between total power operations and their additive quotients after applying $\Div$, while in the $K(n)$-local setting ambidexterity supplies additional transfers that turn this into a classical exponential relation over the coefficient ring. Representation rings give the basic algebraic example, in which the transfer quotients are all rank one and $\Div$ is a ring of integer-valued class functions; here the general construction recovers the classical exponential relations relating symmetric and exterior powers to Adams operations. Morava $E$-theory provides a richer chromatic analogue: Strickland's description of the $E$-cohomology of symmetric groups in terms of finite subgroup schemes of the universal deformation gives an algebro-geometric interpretation of $\Div$, and the $K(n)$-local transfer maps recover Ganter's Hecke operators and their exponential relation with symmetric powers. For mod $2$ singular cohomology, the transfer quotients vanish away from powers of $2$ and are Dickson algebras at powers of $2$, so that $\Div$ organizes the cohomology of symmetric groups into coordinates indexed by dyadic partitions and makes its connection with elementary abelian detection particularly explicit. Combinatorial species show that the same structures are not peculiar to representation theory or cohomology: assembling connected structures along the blocks of a partition produces a partition ring in which the exponential relation is precisely the classical exponential formula expressing arbitrary structures in terms of their connected components. Finally, we produce an example of an exponential element in the Burnside ring partition ring that does not live inside the initial partition ring. This shows that the collection of exponential elements in the Burnside ring partition ring is richer than one might think and allows us to apply \cref{prop:transfer-augmentation} to a lax symmetric monoidal partition functor.

\subsection{Structured ring spectra}

Let $E$ be a spectrum and let $X$ be a space. For a partition $\Lambda = (Z, \sim_{\Lambda})$, let $X \wr \Sigma_{\Lambda} = (X^Z)_{h\Sigma_{\Lambda}}$, the homotopy orbits for the $\Sigma_{\Lambda}$-action on $X^Z$. For each $X$, we get two partition functors
\[
E^0(X \wr \Sigma_{(-)}) \text{ and } E^0(X \times B\Sigma_{(-)}).
\]
Of course, the second of these is simply $(E^X)^0(B\Sigma_{(-)})$. The restriction and transfer maps in these partition functors are induced by the restriction and transfer maps for the cohomology theory $E^0(-)$. Further, if $E$ is a homotopy commutative ring spectrum then these partition functors are both partition rings.

Now assume that $E$ is an $H_{\infty}$-ring spectrum in the sense of \cite{BMMS}. Recall that every $E_{\infty}$-ring spectrum has an underlying $H_{\infty}$-ring spectrum. The $H_\infty$-ring structure on $E$ gives power operations
\[
\P_m \colon E^0(X) \to E^0(X \wr \Sigma_m)
\]
and
\[
P_m \colon E^0(X) \to E^0(X \times B\Sigma_m).
\]
By \cite[Section 1, Chapter VIII]{BMMS}, these operations give the associated partition rings the structure of partition power functors. 

Let $R = E^0(X \wr \Sigma_{(-)})$ or $E^0(X \times B\Sigma_{(-)})$ for an $H_{\infty}$-ring spectrum $E$ and let $P_m \colon R(1) \to R(m)$ be the associated power operation. In most cases, $\Div(R)$ has not been previously studied. \cref{ex:exppower} gives the relation
\[
\sum_{m \geq 0} P_m t^m = \exp\Big(\sum_{k > 0} P_k/I_{k} t^k\Big)
\]
in $\Div(R)\powser{\Sigma}$.

If $E$ is a $K(n)$-local homotopy commutative ring spectrum, then the partition ring $R = E^0(B\Sigma_{(-)})$ is restricted from an ambidextrous global Green functor. If $E$ is a $K(n)$-local $E_{\infty}$-ring spectrum, then $R$ is restricted from an ambidextrous global power functor. Thus the results of \cref{sec:globGreen} apply in these situations and the extra transfer maps can be used to produce exponential relations in $E^0\divpowser{t}$.

The constructions generalize to global equivariant spectra in the sense of \cite{Schwede} and $G_{\infty}$-ring global spectra, which are the global analogues of $H_{\infty}$-ring spectra.

\subsection{Representation rings} \label{sec:reprings}

Let $RU$ be the representation ring functor, let $G$ be a finite group, and let $R_G(\Lambda) = RU(G \times \Sigma_{\Lambda})$. Recall that $RU(G \times H) \cong RU(G) \otimes RU(H)$ for any finite $G$ and $H$. Since $RU(G \times \Sigma_{\Lambda}) \cong RU(G) \otimes RU(\Sigma_{\Lambda})$, we have
\[
R_G(\Lambda)/I_{\Lambda} \cong RU(G) \otimes R_e(\Lambda)/I_{\Lambda} \cong RU(G).
\]
Thus $R_G(\Lambda)/I_{\Lambda}$ is a free rank $1$ $R_G(0)$-module so $R_G$ is a symmetric monoidal partition functor satisfying the hypotheses of \cref{thm:freediv}. In this case we have
\[
\Div(R_G)(\Lambda) \cong \Cl(\Sigma_{\Lambda},RU(G)) 
\cong \Cl(\Sigma_{\Lambda},\Z) \otimes RU(G)
\]
and 
\[
\Div(R_G)[\Sigma] = \bigoplus_{m \geq 0} \Cl(\Sigma_{m},RU(G)) \cong RU(G) \otimes \bigoplus_{m \geq 0} \Cl(\Sigma_m,\Z) \cong RU(G)\langle x_1, x_2, \ldots \rangle,
\]
where $x_i$ is in degree $i$, is the free graded divided polynomial algebra over $RU(G)$ on a countable number of variables with one in each positive degree. Further, \cref{thm:freediv} implies that this is the divided power envelope of $R_G[\Sigma]$, the ring of symmetric functions with coefficients in the representation ring $RU(G)$. 

The power operations
\[
P_m \colon R_G(1) = RU(G) \to R_G(m) = RU(G \times \Sigma_m)
\]
are induced by the map sending a complex $G$-representation $V$ to the $G \times \Sigma_m$-representation $V^{\otimes m}$ with $G$ acting diagonally and $\Sigma_m$ permuting the tensor factors. In this case, $\bar{P}_m = \psi^m$ is the $m$th Adams operation. Thus the operation $P_m/I_m$ is $\psi^m$ concentrated in the summand of $\Div(R_G)(m) = \Cl(\Sigma_m,RU(G))$ corresponding to the conjugacy class of the long cycle $(1\ldots m)$.

Note that $RU$ has transfers along arbitrary maps of finite groups. Given $f \colon H \to G$ and an $H$-representation $V$, we set $RU_*(f)([V]) = [\C[G] \otimes_{\C[H]} V]$. This makes $R_G$ into an ambidextrous global power functor. Let $\pi \colon G \times \Sigma_m \to G$ be the projection. Since the composite
\[
RU(G) \xrightarrow{P_m} RU(G \times \Sigma_m) \xrightarrow{\tr_\pi} RU(G)
\]
is the $m$th symmetric power $\beta^m$, the composite
\[
R_G(1) \to R_G\powser{\Sigma} \xrightarrow{\epsilon} R_G(1)\powser{t} = RU(G)\powser{t}
\]
is $\sum_{m \geq 0} \beta^mt^m$. Applying \cref{cor:ambitriv} and using the fact that $RU(G)$ is torsion-free (so logarithms are unique), this gives the well known exponential relation over $RU(G) \divpowser{t}$:
\begin{equation} \label{eqn:RUexp}
\sum_{m \geq 0} \beta^mt^m = \exp\Big(\sum_{k>0} \frac{\psi^k}{k}t^k\Big).
\end{equation}
One could also apply the formula for $\tilde{\epsilon}_k$ in \cref{prop:transfer-augmentation} to see that $\tilde{\epsilon}_k P_k/I_k = (k-1)!\psi^k$ by following the proof in \cref{prop:ellHecke}.

However, \cref{prop:epsiloncompatible} implies that something more refined actually occurs. The convolution product on $RU(\Sigma_2)$ is induced by the transfer along the surjective multiplication map $\Sigma_2 \times \Sigma_2 \to \Sigma_2$. As in \cref{sec:globGreen}, we will write $RU^{\tr}(\Sigma_2)$ for $RU(\Sigma_2)$ with the convolution product, denoted $\star$. The multiplicative identity in $RU^{\tr}(\Sigma_2)$ is the isomorphism class of the regular representation $[\C\{\Sigma_2\}]$. There is an isomorphism of commutative rings
\[
RU^{\tr}(\Sigma_2) \cong \Z[\sigma]/(\sigma^2-\sigma),
\]
where the isomorphism class of the sign representation maps to $\sigma$. This can be verified using the formula for the convolution product on class functions:
\[
(f \star g)(a) = \frac{1}{2} \sum_{b+c =a} f(b)g(c)
\]
for $a,b,c \in \Sigma_2$ and $f,g \in \Cl(\Sigma_2, \C)$. 

The composite
\[
RU(G) \xrightarrow{P_m} RU(G \times \Sigma_m) \xrightarrow{\tr_{G\times \Sigma_m}^{G \times \Sigma_2}} RU(G \times \Sigma_2)
\]
sends $[V]$ to $[\C[\Sigma_2] \otimes_{\C[\Sigma_m]} V^{\otimes m}]$ and where $\tr_{G \times \Sigma_m}^{G \times \Sigma_2}$ is the transfer along the surjection $G \times \Sigma_m \to G \times \Sigma_2$ induced by the sign homomorphism. In the notation of \cref{prop:epsiloncompatible}, this is $\epsilon^2 P_m$. The exponential relation
\begin{equation} \label{eq:repthy}
\sum_{m \geq 0} \epsilon^2 P_m t^m = \exp\Big(\sum_{k \geq 1}  \tilde \epsilon^2 P_k/I_k \frac{t^k}{k!}\Big).
\end{equation}
holds in $R_{G}^{\tr}(2)\divpowser{t}$.

There are two ring maps $RU^{\tr}(\Sigma_2) \to \Z$ given by $\sigma \mapsto 1$ and $\sigma \mapsto 0$ and these induce ring maps $R_G^{\tr}(2) \to R_G(1)$. Since the ring map given by $\sigma \mapsto 0$ is transfer along $\Sigma_2 \to e$, applying this to \cref{eq:repthy} recovers the exponential relation of \cref{eqn:RUexp}. The ring map given by $\sigma \mapsto 1$ gives rise to another classical exponential relation. The composite
\[
RU(G) \xrightarrow{P_m} RU(G \times \Sigma_m) \cong RU(G) \otimes RU(\Sigma_m) \xrightarrow{\id \otimes \tr_{\Sigma_m}^{\Sigma_2}} RU(G) \otimes RU^{\tr}(\Sigma_2) \xrightarrow{\id \otimes (\sigma \mapsto 1)} RU(G)
\]
is induced by the map sending a $G$-representation $V$ to $\mathrm{sgn}_m \otimes_{\C[\Sigma_m]} V^{\otimes m}$, where $\mathrm{sgn}_m$ is the $1$-dimensional sign representation of $\Sigma_m$. This is the $m$th exterior power of $V$. Thus the exponential relation that we recover with the ring map $\sigma \mapsto 1$ is the classical relation between exterior powers and Adams operations
\[
\sum_{m \geq 0} \Lambda^mt^m = \exp\Big(\sum_{k \geq 1}  \frac{(-1)^{k+1}\psi^k}{k} t^k\Big).
\]
We see that if $1_\star$ denotes the multiplicative identity of $RU^{\tr}(\Sigma_2)$, then the ordinary trivial representation is $1_\star-\sigma$, and
\[
\epsilon^2P_m([V])
=
\beta^m([V])(1_\star-\sigma)+\Lambda^m([V])\sigma.
\]

This is also a concrete realization of the composite involution of \cref{prop:signinvolution}. Under the canonical map of partition rings $RU(\Sigma_{(-)})\to R_G$, the sign element $s_m$ is represented by $1_G \boxtimes [\mathrm{sgn}_m]$, where $1_G$ is the trivial representation of $G$. Hence the composite involution $u(t)\mapsto u(-t)^{-1}$ sends the total power operation to
\[
\sum_{m\geq0}\bigl(s_m*P_m(x)\bigr)t^m
=
\bigg(\sum_{m\geq0}(-1)^mP_m(x)t^m\bigg)^{-1}.
\]
For $x=[V]$, the coefficient of $t^m$ is represented by $\mathrm{sgn}_m\otimes V^{\otimes m}$. Transfer along $G\times\Sigma_m\to G$ therefore sends the original series to $\sum_m\beta^m(x)t^m$ and the transformed series to $\sum_m\Lambda^m(x)t^m$, recovering the classical identity
\[
\sum_{m\geq0}\Lambda^m(x)t^m
=
\bigg(\sum_{m\geq0}(-1)^m\beta^m(x)t^m\bigg)^{-1}.
\]

\subsection{Morava $E$-theory}

Let $\kappa$ be a perfect field of characteristic $p$ and let $\Gamma$ be a height $n$ formal group law over $\kappa$. Let $E = E(\Gamma/\kappa)$ be the Morava $E$-theory associated to $\Gamma / \kappa$. This is a $K(n)$-local $E_{\infty}$-ring spectrum. Recall that $\pi_0E$ is the Lubin--Tate ring and that there is a noncanonical isomorphism $\pi_0E \cong W_p(\kappa)\powser{u_1, \ldots, u_{n-1}}$, where $W_p(\kappa)$ is the ring of $p$-typical Witt vectors. Let $\G / \Spf(\pi_0E)$ be the universal deformation of $\Gamma$. See \cite{StapletonHandbook} for more details.

Let $I_m \subseteq E^0(B\Sigma_m)$ be the transfer ideal. In \cite{etheorysym}, Strickland proves that $E^0(B\Sigma_m)/I_m$ is a finitely generated free $E^0$-module and that there is a canonical isomorphism of formal schemes over Lubin--Tate space
\[
\Spf(E^0(B\Sigma_m)/I_m) \cong \Sub_m(\G),
\]
where $\Sub_m(\G)$ is the formal scheme classifying subgroup schemes of $\G$ of order $m$. Of course, $\Sub_m(\G) = \emptyset$ unless $m$ is a power of $p$.

It follows from \cite[Theorem 7.3]{hkr} that $E^0(B\Sigma_{\Lambda})$ is a free $E^0$-module, thus the partition ring $E^0(B\Sigma_{(-)})$ is symmetric monoidal. Our first goal is to understand $\Div(E^0(B\Sigma_{(-)}))(m)$. Let $\lambda \parts m$ be an integer partition of $m$. Recall that the length $\ell(\lambda) = \sum_{i=1}^{\infty} \lambda_i$ is the number of summands in the integer partition. Because an integer partition is an unordered list of positive natural numbers, an integer partition $\lambda$ of length $k$ is an element of the $k$th symmetric power
\[
((\N_{>0})^{\times k})/\Sigma_{k}.
\]
Let $\Sub(\G) = \coprod_{m > 0} \Sub_m(\G)$. Applying the $k$th symmetric power functor to the map
\[
\Sub(\G) \to \N_{>0}
\]
gives 
\[
(\Sub(\G)^{\times k})/\Sigma_{k} \to ((\N_{>0})^{\times k})/\Sigma_{k}.
\]
Let $\Sum_{\lambda}(\G)$ be the fiber over $\lambda$. If $\lambda \parts m$ and $\lambda_m =1$ (i.e. $\ell(\lambda) =1$), then $\Sum_{\lambda}(\G) = \Sub_m(\G)$. The scheme $\Sum_{\lambda}(\G)$ classifies ``formal sums" of subgroups of $\G$ with $\lambda_i$ summands that are subgroups of size $i$. There is an isomorphism of formal schemes 
\begin{equation*}
\Sum_{\lambda}(\G) \cong \Spf \Big (\bigotimes_{i > 0} ((E^0(B\Sigma_i)/I_i)^{\otimes \lambda_i})^{\Sigma_{\lambda_i}} \Big ).
\end{equation*}
Applying \cref{eq:symmondecomp} and the proof of \cref{prop:flatdivsymmon} (using Strickland's freeness result), we learn that
\begin{equation} \label{eq:sumG}
\Spf(\Div(E^0(B\Sigma_{(-)}))(m)) \cong \coprod_{\lambda \parts m} \Sum_{\lambda}(\G).
\end{equation}

To better understand this object, the best discrete approximation to this formal scheme is in terms of what we call height $n$ integer partitions. Let $\mathbb{T} = (S^1)^{\times n}$ be the $n$-fold product of the circle group and let $\mathbb{T}_p = (\Q_p/\Z_p)^{\times n} = \mathbb{T}[p^{\infty}]$. Let $\Sub(\mathbb{T})$ be the set of finite subgroups of $\mathbb{T}$. A height $n$ integer partition of $m$ is a function
\[
\tau \colon \Sub(\mathbb{T}) \to \N
\]
such that $\sum_{A \in \Sub(\mathbb{T})} \tau_{A} \cdot |A| = m$. Note that this specializes to the definition of an ordinary integer partition when $n=1$. Similarly, a height $n$ $p$-power integer partition of $m$ is a height $n$ integer partition of $m$ supported on $\Sub(\mathbb{T}_p) \subseteq \Sub(\mathbb{T})$ (i.e. for which $\tau_A = 0$ whenever $A$ is not a $p$-group). A height $n$ integer partition $\tau$ of $m$ has an underlying ordinary integer partition $u(\tau)$ of $m$ given by the formula
\[
u(\tau)_i = \sum_{\substack{A \in \Sub(\mathbb{T}) \\ |A| = i}} \tau_A.
\]

Let $C_0$ be the rationalization of the Drinfeld ring of infinite level structures on $\G$ and let $\lambda$ be an integer partition of $m$, then 
\[
\Spec(C_0) \times_{\Spec(E^0)} \Sum_{\lambda}(\G)
\]
is canonically isomorphic to (the constant scheme associated to) the set of height $n$ $p$-power integer partitions $\tau$ of $m$ with underlying partition $u(\tau) = \lambda$. Note that the base change makes sense as written since the ring of functions on $\Sum_{\lambda}(\G)$ is finitely generated and free as an $E^0$-module.

Next we turn our attention to exponential relations. Write $E_X$ for the partition functor with 
\[
E_X(\Lambda) = E^0(X \times B\Sigma_{\Lambda}).
\]
When $X$ is a point, we will denote this by $E$. Since Morava $E$-theory is $K(n)$-local and an $E_{\infty}$-ring spectrum, the partition functor $E_X$ is restricted from the ambidextrous global power functor with value at $G$ given by $E^0(X \times BG)$. Recall that $E^0(B\Sigma_\Lambda)$ is a finitely generated free $E^0$-module. This implies that 
\[
E^0(X \times B\Sigma_{\Lambda}) \cong E^0(X) \otimes_{E^0} E^0(B\Sigma_{\Lambda}),
\]
which implies that $E_X$ is a symmetric monoidal partition power functor.


Since $E^0(B\Sigma_m)/I_{m}$ is a free $E^0$-module, we have an isomorphism of $E^0$-algebras
\[
E_X(m)/I_{m} \cong E^0(X) \otimes_{E^0} E^0(B\Sigma_m)/I_{m}.
\]
Thus $E_X(m)/I_{m}$ is a free module over $E_X(0) = E^0(X)$ so \cref{thm:freediv} can be applied to $E_X$.


It also follows that there are isomorphisms of graded $E^0(X)$-algebras $E_X[\Sigma] \cong E^0(X) \otimes_{E^0} E[\Sigma]$ and $E_X\langle \Sigma \rangle \cong E^0(X) \otimes_{E^0} E\langle \Sigma \rangle$. Combining \cref{prop:symmetricinvariant} and \cref{eq:sumG} tells us that the $E^0$-algebra $E \langle \Sigma \rangle_m$ is the ring of functions on the ``formal'' scheme $\coprod_{\lambda \vdash m} \Sum_{\lambda}(\G)$. The base change of $E[\Sigma]$ to $C_0$ was studied by Strickland and Turner in \cite{StricklandTurner}.


In \cite{Ganter-orbifold}, Ganter produced an exponential relation between symmetric powers and Hecke operators on Morava $E$-theory. We describe these operations now. Let $\beta^m$ be the composite
\[
E^0(X) \xrightarrow{P_m} E^0(X \times B\Sigma_m) \cong E^0(X) \otimes_{E^0} E^0(B\Sigma_m) \xrightarrow{\id \otimes \tr_{\Sigma_m}^{e}} E^0(X) \otimes_{E^0} E^0(Be) \cong E^0(X),
\]
where $\tr_{\Sigma_m}^{e}$ is the transfer along $\Sigma_m \to e$. We call $\beta^m$ the $m$th symmetric power. 

Let $\Level((\Z/p^k)^n,\G)$ be the formal scheme of $(\Z/p^k)^n$-level structures on $\G$ and let $D_k$ be the ring of functions on $\Level((\Z/p^k)^n,\G)$ as described in \cite[Section 4]{Drinfeld}. For $H \subseteq (\Z/p^k)^n$ of order $p^k$, there is an $E^0$-algebra map
\[
\alpha_H \colon E^0(B\Sigma_{p^k})/I_{p^k} \to D_k
\]
classifying the subgroup of order $p^k$ that is the image of $H$ under the universal level structure. Ando's Adams operation \cite{Ando} associated to $H$ is the ring map
\[
\psi_H \colon E^0(X) \xrightarrow{\bar{P}_{p^k}} E^0(X) \otimes_{E^0} E^0(B\Sigma_{p^k})/I_{p^k} \xrightarrow{\id \otimes \alpha_H} E^0(X) \otimes_{E^0} D_k.
\]
The sum $\sum_{\substack{H \subseteq (\Z/p^k)^n\\|H|=p^k}} \alpha_H$ is $\GL_n(\Z/p^k)$-invariant, and thus lands in $(D_{k})^{\GL_n(\Z/p^k)} = E^0$. Thus the Hecke operator 
\[
T_{p^k} = \sum_{\substack{H \subseteq (\Z/p^k)^n\\|H|=p^k}} \psi_H = \bigg (\sum_{\substack{H \subseteq (\Z/p^k)^n\\|H|=p^k}} \id \otimes \alpha_H \bigg ) \circ \bar{P}_{p^k}
\]
lands in $E^0(X)$ and therefore gives rise to an additive operation
\[
T_{p^k} \colon E^0(X) \to E^0(X).
\]

With this normalization of $T_{p^k}$ (Ganter's notation includes the factor $\frac{1}{p^k}$), \cite[Proposition 9.1]{Ganter-orbifold} gives
\[
\sum_{m \geq 0} \beta^m t^m = \exp\Big(\sum_{k \geq 0} \frac{T_{p^k}}{p^k} t^{p^k}\Big).
\]
However, it is a bit unclear in \cite{Ganter-orbifold} where this relation holds, but it certainly holds in $(\Q \otimes E^0(X))\powser{t}$. 

\cref{ex:exppower} provides us with the exponential relation
\[
\sum_{m \geq 0} P_m t^m = \exp\Big(\sum_{k > 0} P_k/I_{k} t^k\Big)
\]
in $E_X\divpowser{\Sigma}$. Since $E^0(B\Sigma_m)/I_m = 0$ when $m$ is not a power of $p$, this simplifies to
\begin{equation} \label{eq:Ethyexp}
\sum_{m \geq 0} P_m t^m = \exp\Big(\sum_{k \geq 0} P_{p^k}/I_{p^k} t^{p^k}\Big).
\end{equation}
As in \cref{sec:globGreen}, the map $\tilde \epsilon \colon E_X\divpowser{\Sigma} \to E_X(1)\divpowser{t}$ can be applied to the exponential relation to get an exponential relation over 
\[
E_X(1)\divpowser{t} = E^0(X)\divpowser{t}.
\]
By definition, we have $\epsilon_m P_m = \beta^m$. To obtain Ganter's exponential relation, \cref{eq:epsilonpower} implies that we must show that $\tilde \epsilon_{p^k} P_{p^k}/I_{p^k} = (p^k-1)!T_{p^k}$. Since $\ell_{p^k} \bar{P}_{p^k} = \tilde \epsilon_{p^k} P_{p^k}/I_{p^k}$ and $T_{p^k} = (\id \otimes \sum_{H} \alpha_H) \circ \bar{P}_{p^k}$, it suffices to compare $\ell_{p^k}$ and $(p^k-1)!(\id \otimes \sum_{H} \alpha_H)$. Since 
\[
\epsilon_{p^k} = \id \otimes \tr_{\Sigma_{p^k}}^{e} \colon E^0(X) \otimes_{E^0} E^0(B\Sigma_{p^k}) \to E^0(X),
\]
we have $\ell_{p^k, \bar{E}_X(p^k)} = \id \otimes \ell_{p^k, \bar{E}(p^k)}$, where $\ell_{p^k, \bar{E}_X(p^k)}$ is associated to the partition power functor $E_X$ and $\ell_{p^k, \bar{E}(p^k)}$ is associated to the partition power functor $E$. Thus it suffices to prove the following:

\begin{prop}\label{prop:ellHecke}
Let
\[
\ell_{p^k}\colon E^0(B\Sigma_{p^k})/I_{p^k} \longrightarrow E^0
\]
be the map constructed in \cref{prop:transfer-augmentation}.
We have
\[
\ell_{p^k} = (p^k-1)! \Bigg ( \sum_{\substack{H \subseteq (\Z/p^k)^n\\|H|=p^k}} \alpha_H \Bigg).
\]
\end{prop}

\begin{proof}
Since $E^0 \subseteq C_0$, it suffices to prove this viewing the target as $C_0$ and this allows us to use Hopkins--Kuhn--Ravenel character theory \cite{hkr}.

Let $m = p^k$. We will make use of several sets: Let $\hom(\Z_{p}^{n}, \Sigma_{m})$ be the set of continuous group homomorphisms from $\Z_{p}^{n}$ to $\Sigma_{m}$. Let $\hom^{\trans}(\Z_{p}^{n}, \Sigma_{m})$ denote the subset of continuous homomorphisms with transitive image. Let $\hom(\Z_{p}^{n}, \Sigma_{m})/{\conj}$ denote the set of conjugacy classes and $\hom^{\trans}(\Z_{p}^{n}, \Sigma_{m})/{\conj}$ the set of conjugacy classes of transitive homomorphisms. There is a canonical bijection of sets
\[
\Phi \colon \hom^{\trans}(\Z_{p}^{n}, \Sigma_{m})/{\conj} \cong \Sub_{m}((\Z/m)^{n}),
\]
where $\Sub_{m}$ denotes the set of subgroups of order $m$. A transitive map $\gamma \colon \Z_p^n\to\Sigma_{m}$ determines a transitive
$\Z_p^n$-set $\Z_p^n/L$ of order $m$. Its Pontryagin dual determines a subgroup
\[
H\leq(\Q_p/\Z_p)^n[m]\cong(\Z/m)^n
\]
of order $m = p^k$. This is independent of choice of map in $[\gamma]$. Also, note that the orbit-stabilizer lemma and the fact that $\im(\gamma) \subset \Sigma_m$ is abelian transitive, implies that the conjugacy class of $\gamma$ has
cardinality
\[
\frac{|\Sigma_{m}|}
{|C_{\Sigma_{m}}(\operatorname{im}(\gamma))|}
=
\frac{{m}!}{{m}}
=
({m}-1)!.
\]

The character map for $E^0(B\Sigma_{m})$ has the form
\[
\chi \colon E^0(B\Sigma_{m}) \to \Fun(\hom(\Z_{p}^{n}, \Sigma_{m})/{\conj}, C_0)
\]
and, for a conjugacy class $[\gamma]$, we write $\chi_{[\gamma]}$ for the map to the factor corresponding to $[\gamma]$. We will also write $\chi_{\gamma} = \chi_{[\gamma]}$ for any $\gamma \in \hom(\Z_{p}^{n}, \Sigma_{m})$.

Let $x\in E^0(B\Sigma_{m})/I_{m}$ and choose a lift
$\widetilde{x}\in E^0(B\Sigma_{m})$. The construction of the character map $\chi$ implies that $\alpha_{H} q_{m} = \chi_{\Phi^{-1}H}$. For $\Gamma \leq \u m$, we view $\hom(\Z_p^n, \Sigma_{\Gamma})$ as a subset $\hom(\Z_p^n, \Sigma_{m})$. The character formula for
the transfer along $\Sigma_\Gamma \to e$ \cite[Proposition 7.9]{Ganter-orbifold} together with the formula for restriction of characters, imply that
\[
|\Sigma_\Gamma|\epsilon_\Gamma(\res_{\Sigma_{\Gamma}}^{\Sigma_m}\tilde x)
=
\sum_{\gamma \in
\hom(\Z_p^n,\Sigma_\Gamma)}
\chi_\gamma(\tilde x).
\]
Note that the sum is over all maps (not conjugacy classes). Using the formula for $\ell_m$ from
Proposition~\ref{prop:transfer-augmentation}, we obtain
\[
\ell_{m}(x)
=
\sum_{\Gamma\leq\u {m}}
\mu(\Gamma,\u {m})
\Big ( \sum_{\gamma \in
\hom(\Z_p^n,\Sigma_\Gamma)}
\chi_\gamma(\widetilde{x}) \Big ).
\]
Given $\gamma \colon \Z_p^n\longrightarrow\Sigma_m$, let $\Omega_\gamma \leq \u m$ be the partition of $[m]$ according to the orbits of $\gamma$. The image of
$\gamma$ is contained in $\Sigma_\Gamma$ precisely when
$\Omega_\gamma\leq\Gamma$. Reordering the sum above gives
\[
\ell_{m}(x)
=
\sum_{\gamma\in
\hom(\Z_p^n,\Sigma_{m})}
\Big (
\sum_{\Omega_\gamma \leq\Gamma\leq\u {m}}
\mu(\Gamma,\u {m})
\Big )
\chi_\gamma(\widetilde{x}).
\]
By the defining identity for the Mobius function,
\[
\sum_{\Omega_\gamma \leq\Gamma\leq\u {m}}
\mu(\Gamma,\u {m})
=
\begin{cases}
1,&\Omega_\gamma=\u {m},\\
0,&\Omega_\gamma<\u {m}.
\end{cases}
\]
Consequently,
\[
\ell_{m}(x)
=
\sum_{\gamma \in \hom^{\trans}(\Z_{p}^{n}, \Sigma_{m})}
\chi_\gamma(\widetilde{x}).
\]
Grouping the transitive homomorphisms by conjugacy classes and applying $\Phi$ now gives
\[
\ell_{m}(x)
=
({m}-1)!
\Big(\sum_{\substack{H \in \Sub_{m}((\Z/m)^n)}}
\alpha_H(x)\Big),
\]
as desired.
\end{proof}

We learn that
\[
\frac{\tilde \epsilon P_{p^k}/I_{p^k}}{(p^k)!} = \frac{T_{p^k}}{p^k}.
\]
Thus applying $\tilde \epsilon$ to the exponential relation of \cref{eq:Ethyexp} gives Ganter's exponential relation. \cref{eq:Ethyexp} offers the possibility of finding more exponential relations related to Morava $E$-theory.

At the prime $2$ (i.e., if $\kappa$ has characteristic $2$) there is a refinement of this formula due to the fact that $E^0(B\Sigma_2)$ is not isomorphic to $E^0$. In the case of $RU$, described in \cref{sec:reprings}, this had to do with exterior powers. By \cref{prop:exptransferSigma2}, there is an $E^0(X)$-algebra map
\[
E_X\divpowser{\Sigma} \to E_{X}^{\tr}(2)\divpowser{t}.
\]
Let $V = \hom((\Z_{2})^{\times n}, \Sigma_2)$, the set of continuous group homomorphisms from $(\Z_{2})^{\times n}$ to $\Sigma_2$. Since $\Sigma_2$ is abelian, $V$ is an abelian group isomorphic to $(\Z/2)^{\times n}$. Recall that $C_0$ is the rationalized Drinfeld ring of infinite level structures. Hopkins--Kuhn--Ravenel character theory \cite[Theorem C]{hkr} provides an isomorphism of $E^0$-algebras
\[
C_0 \otimes_{E^0} E^0(B\Sigma_2) \cong \prod_{V} C_0.
\]

Recall that there is an isomorphism of $E^0$-algebras $E^0(B\Sigma_2) \cong E^0\powser{x}/([2](x))$ for a choice of coordinate $x$ on $\G$ and where $[2](x)$ is the $2$-series for the associated formal group law. The convolution product on $E^0(B\Sigma_2)$ has multiplicative unit $[2](x)/x$, which is the image of $1 \in E^0$ under the transfer map $E^0(Be) \to E^0(B\Sigma_2)$. 

By \cite{Stricklandtransfer} (see also \cite{Ganter-orbifold}), the convolution product on $\prod_{V} C_0$ is given by
\[
(f \star g)(v) = \frac{1}{2} \sum_{u+w = v} f(u)g(w)
\]
and the multiplicative unit is given by the function that is $2$ on $0 \in V$ and $0$ elsewhere. This is isomorphic to the group algebra $C_0[V]$ and this is isomorphic to $\prod_{V^{\vee}} C_0$, since $C_0$ is a $\Q$-algebra and where $V^{\vee}$ is the dual. Thus, each $v \in V^{\vee}$ determines a map of $E^0$-algebras
\[
E_{X}^{\tr}(2) \to C_0 \otimes_{E^0} E_X(1).
\]
In fact, the image of each of these maps lands in $D_1 \otimes_{E^0} E_X(1)$. These maps give rise to exotic exponential relations over
\[
(D_1 \otimes_{E^0} E_X(1))\divpowser{t} \cong (D_1 \otimes_{E^0} E^0(X))\divpowser{t}.
\]
We will explore this further in future work.



\subsection{Mod $2$ singular cohomology}

We now consider the example of singular cohomology with coefficients in
\(\F_2\).  This example is slightly different in flavor from the previous
one because we want to retain the graded structure. Since we are working with $\F_2$-coefficients, we do not need to worry about graded commutativity. Thus we will ignore the grading and set
\[
        \cH(\Lambda)=H^*(B\Sigma_{\Lambda};\F_2),
\]
viewing this as a commutative ring. The Kunneth
isomorphism for mod $2$ singular cohomology gives
\[
H^*(B\Sigma_{\Lambda};\F_2)\otimes H^*(B\Sigma_{\Omega};\F_2)
        \cong
H^*(B(\Sigma_{\Lambda}\times\Sigma_{\Omega});\F_2)
        \cong
H^*(B\Sigma_{\Lambda\sqcup\Omega};\F_2).
\]
Thus $\cH$ is a symmetric monoidal partition ring. 

Let
\[
        I_{\Lambda}\subseteq \cH(\Lambda) =  H^*(B\Sigma_{\Lambda};\F_2)
\]
be the transfer ideal and recall that $\bar{\cH}(\Lambda)=\cH(\Lambda)/I_{\Lambda}$.


We will use the following classical calculation, in the form appearing in
the work of Cohen--Lada--May \cite{CohenLadaMay} and in the Hopf ring calculation of
Giusti--Salvatore--Sinha \cite{sinhamod2sym}.  Let \(\cD_k\) denote the \(k\)th Dickson algebra
over \(\F_2\):
\[
        \cD_k
        =
        H^*(B(\Z/2)^k;\F_2)^{GL_k(\F_2)}
        \cong
        \F_2[d_{k,0},d_{k,1},\ldots,d_{k,k-1}],
\]
with cohomological degree $|d_{k,i}|=2^k-2^i$. Note that \(\cD_0=\F_2\).  Then
\[
\bar{\cH}(m)\cong
\begin{cases}
        \cD_k, & m=2^k,\\
        0,     & m\text{ is not a power of }2.
\end{cases}
\]

For $\Lambda = (X, \sim_{\Lambda})$, we have
\[
        \cH(\Lambda)\cong
        \bigotimes_{B \in X/{\sim_{\Lambda}}} H^*(B\Sigma_{|B|};\F_2)
\]
and, from \cref{lem:symmon},
\[
        \bar{\cH}(\Lambda)
        \cong
        \bigotimes_{B \in X/{\sim_{\Lambda}}} \bar{\cH}(|B|).
\]
In particular, \(\bar{\cH}(\Lambda)\) vanishes unless every block of
\(\Lambda\) has cardinality a power of two.  If \(\Lambda\) has \(a_k\)
blocks of cardinality \(2^k\), then
\[
        \bar{\cH}(\Lambda)
        \cong
        \bigotimes_{k\geq 0}\cD_k^{\otimes a_k}.
\]

Applying \cref{eq:symmondecomp} and the proof of \cref{prop:flatdivsymmon}, using the fact that $\cD_i$ is a free $\F_2$-modules, we obtain the following description of
\(\Div(\cH)(m)\).  Let \(\mathrm{Part}_2(m)\) denote the set of $2$-power integer partitions
of $m$. For $\lambda \in \mathrm{Part}_2(m)$, $\lambda_i$ counts the number of times $2^i$ appears in the integer partition, so that we have $m = \sum_{i \geq 0} \lambda_i 2^i$. 
Then $\Div(\cH)(0) = \F_2$ and, for $m > 0$, we have
\[
        \Div(\cH)(m)
        \cong
        \bigoplus_{\lambda\in \mathrm{Part}_2(m)}
        \bigg(
        \bigotimes_{i \geq 0} \Big ( \cD_i^{\otimes \lambda_i} \Big )^{\Sigma_{\lambda_i}}
        \bigg)
\]
and, using the notation of \cref{sec:divpow},
\[
        \cH \langle \Sigma \rangle
        =
        \Div(\cH)[\Sigma]
        \cong
        \bigoplus_{m\geq 0}
        \Div(\cH)(m) \cong TS(\bigoplus_{i \geq 0} D_i),
\]
where $\bigoplus_{i \geq 0} D_i$ is viewed as a graded $\F_2$-module with $D_i$ in degree $2^i-1$. Thus \(\Div(\cH)(m)\) is a form of class functions taking values in a dyadic Dickson algebra depending on the partition and \cref{thm:freediv} implies that $\cH\langle \Sigma \rangle$ is the divided power envelope of $\cH[\Sigma]$ with respect to the ideal generated by positive-degree homogeneous elements.


The unit of the monad $\Div$ gives a ring map
\[
        \eta_{m} \colon
        \cH(m)=H^*(B\Sigma_m;\F_2)
        \longrightarrow
        \Div(\cH)(m).
\]
The $\lambda$-coordinate of $\eta_{m}$ is induced by the composite
\[
H^*(B\Sigma_m;\F_2)
        \xrightarrow{\mathrm{res}_{\Sigma_{\Lambda}}^{\Sigma_{m}}}
H^*(B\Sigma_{\Lambda};\F_2)
        \xrightarrow{q_{\Lambda}}
\bar{\cH}(\Lambda) \cong \bigotimes_{i \geq 0}\cD_i^{\otimes \lambda_i},
\]
for any choice of partition $\Lambda \parts \u m$ with underlying integer partition $\lambda$. This map lands in the indicated fixed points as $W_{\Sigma_m}(\Sigma_\Lambda) \cong \prod_{i \geq 1} \Sigma_{\lambda_i}$.

Let $E_{2^k} \subseteq \Sigma_{2^k}$ be a transitive elementary abelian $2$-subgroup of $\Sigma_{2^k}$ (unique up to conjugacy). The restriction map
\[
\res_{E_{2^k}}^{\Sigma_{2^k}} \colon H^*(B\Sigma_{2^k};\F_2) \to H^*(BE_{2^k}; \F_2)
\]
lands in $D_k$, the $\Aut(E_{2^k}) \cong GL_k(\F_2)$-invariants, because each automorphism of $E_{2^k}$ extends to an inner automorphism of $\Sigma_{2^k}$. There is a commutative diagram of commutative rings
\[
\xymatrix{\cH(2^k) \ar[r]^{q_{2^k}} \ar[d]_{=} & \bar{\cH}(2^k) \ar[d]^{\cong} \\ H^*(B\Sigma_{2^k};\F_2) \ar[r]^-{\res_{E_{2^k}}^{\Sigma_{2^k}}} & D_k.}
\]
It follows that $q_{\Lambda}$, since it is a tensor product of quotient maps, admits a similar interpretation.

Returning to the ring homomorphism $\eta_{m} \colon \cH(m) \to \Div(\cH)(m)$, note that this map is not generally an isomorphism. Already for \(m=2\),
\[
        H^*(B\Sigma_2;\F_2)\cong \F_2[u],
\]
where the cohomological degree of $u$ is  $1$, whereas
\[
        \Div(\cH)(2)
        \cong
        \bar{\cH}(2)\oplus \bar{\cH}(\u 1 \sqcup \u 1)
        \cong
        \F_2[u]\oplus \F_2.
\]
Under this identification,
\[
        \eta_{2}(f(u))=(f(u),f(0)).
\]

However, work of \cite{GunawardenaLannesZarati} proves that the Quillen map
\[
        H^*(B\Sigma_m;\F_2)
        \longrightarrow
        \lim_{E\in \mathcal C(\Sigma_m)} H^*(BE;\F_2)
\]
is an isomorphism. Here $C(\Sigma_m)$ is the category with objects elementary abelian $2$-subgroups of $\Sigma_m$ and morphisms generated by inclusions and conjugations in $\Sigma_m$. Since each elementary abelian $2$-subgroup of $\Sigma_m$ is contained in the product of transitive elementary abelian $2$-subgroups of $\Sigma_{\Lambda}$ for some $\Lambda \leq \u m$, this implies that $\eta_{m}$ is injective for all $m$. This also can be used to give a description of the image of $\eta_{m}$ -- it consists of appropriately compatible families of elements in $\Div(\cH)(m)$.

In \cite{sinhamod2sym}, Giusti, Salvatore, and Sinha describe the component Hopf ring structure on $\cH[\Sigma]$ in terms of classes
\[
\gamma_{\ell,n}\in H^{n(2^\ell-1)}(B\Sigma_{n2^\ell};\F_2),
\]
with $\ell > 0$. Let $m = n2^{\ell}$. It is not hard to calculate 
\[
\eta_{m}(\gamma_{\ell,n})
\]
in terms of the Dickson classes. For $\lambda \in \mathrm{Part}_2(m)$, we have
\[
(\eta(\gamma_{\ell,n}))_\lambda
=
\begin{cases}
\displaystyle
\bigotimes_{i\geq \ell}
d_{i,i-\ell}^{\otimes \lambda_i},
&
\lambda_0=\cdots=\lambda_{\ell-1}=0, \\[1.2em]
0,
&
\text{otherwise.}
\end{cases}
\]
This follows from \cite[Corollary 7.6]{sinhamod2sym}, which describes the restriction of classes of the form $\gamma_{\ell, 2^k}$ to the Dickson algebra, and \cite[proof of Theorem 8.3]{sinhamod2sym}, which describes $\res_{\Sigma_{\Lambda}}^{\Sigma_m}(\gamma_{\ell, n})$ for $\Lambda \leq \u m$ in terms of external tensor products of lower $\gamma$'s.

For $X$ a space, let $\cH_X(\Lambda) = H^*(X \times B\Sigma_{\Lambda}; \F_2)$. Then $\cH_X$ is also a symmetric monoidal partition ring and we have an isomorphism of commutative rings
\[
\Div(\cH_X)(m) \cong H^*(X; \F_2) \otimes \Div(\cH)(m).
\]
These facts give rise to isomorphisms of commutative graded rings
\begin{equation*}
\cH_X[\Sigma] \cong H^*(X; \F_2) \otimes \cH[\Sigma] \, \text{ and } \, \cH_X\langle \Sigma \rangle \cong H^*(X; \F_2) \otimes \cH \langle \Sigma \rangle.
\end{equation*}
\cref{thm:freediv} implies that $\cH_X\langle \Sigma \rangle$ is the divided power envelope of $\cH_X[\Sigma]$ with respect to the irrelevant ideal. The unit $\cH_X \to \Div(\cH_X)$ is $H^*(X; \F_2) \otimes \eta_{m}$ and $H^*(X; \F_2)$ is a free $\F_2$-module, so the unit $\cH_X(m) \to \Div(\cH_X)(m)$ and thus the map of commutative graded rings $\cH_X[\Sigma] \to \cH_X\langle \Sigma \rangle$ are injective for any space $X$.

\subsection{Combinatorial species}
A common source of exponential relations is the theory of combinatorial species \cite{Combinatorialspecies}. In this section, we discuss the relation between combinatorial species and partition functors. This provides us with examples of partition functors without obvious connections to cohomology theories.

Let $\Set^{\mathrm{fin}}$ be the category of finite sets and let $\Set^{\mathrm{fin},\cong}$ be the groupoid of finite sets and bijections. Let
\[
F \colon \Set^{\mathrm{fin},\cong}\longrightarrow \Set^{\mathrm{fin}}
\]
be a combinatorial species with $F(\emptyset) = \emptyset$. We will regard the elements of $F(X)$ as connected structures on $X$ and refer to the elements of $F(X)$ as $F$-structures.

For a finite set $X$, define the set of $F$-assemblies on $X$ by
\[
\Asm_F(X)
=
\coprod_{\Gamma \parts X}
\prod_{B \in X/{\sim_\Gamma}}F(B).
\]
Thus an element of $\Asm_F(X)$ is a partition $\Gamma$ of $X$,
together with an $F$-structure on each block of $\Gamma$. The
partition $\Gamma$ will be called the component partition of the
assembly.

More generally, if $\Lambda$ is a partition of $X$, define
\[
\Asm_F(\Lambda)
=
\prod_{B\in X/{\sim_\Lambda}}\Asm_F(B).
\]
Equivalently,
\[
\Asm_F(\Lambda)
\cong
\coprod_{\Gamma\leq\Lambda}
\prod_{B\in X/{\sim_\Gamma}}F(B).
\]
The group $\Sigma_\Lambda$ acts on $\Asm_F(\Lambda)$ by transporting
the component partition and the $F$-structures on its blocks. That is, given $\sigma \in \Sigma_{\Lambda}$, the induced isomorphism of partitions $\sigma \colon \Gamma \to \sigma \Gamma$ restricts to an isomorphism $\sigma |_{B} \colon B \to \sigma B$ for $B \in X/{\sim_\Gamma}$ that gives an isomorphism $F(\sigma |_{B}) \colon F(B) \to F(\sigma B)$.

Fix a commutative ring $K$. Define
\[
K_F(\Lambda)
=
\Map(\Asm_F(\Lambda),K)^{\Sigma_\Lambda}.
\]
Thus $K_F(\Lambda)$ is the ring of invariant $K$-valued functions on
the set of $F$-assemblies compatible with the partition $\Lambda$.

We describe a symmetric monoidal partition ring structure on $K_F$. Let $\Omega\leq\Lambda$ be partitions of the same finite set. There
is a $\Sigma_\Omega$-equivariant inclusion
\[
j_\Omega^\Lambda \colon \Asm_F(\Omega)\injto\Asm_F(\Lambda).
\]
The restriction map
\[
K_F^*(\Omega\leq\Lambda) \colon K_F(\Lambda)\longrightarrow K_F(\Omega)
\]
is restriction of functions along $j_\Omega^\Lambda$.

The transfer along $\Omega \leq \Lambda$ is defined in the following way: If $\varphi\in K_F(\Omega)$, let $\widetilde{\varphi}$ be its
extension by zero to $\Asm_F(\Lambda)$. Then
\[
(K_F)_*(\Omega\leq\Lambda)(\varphi)
=
\sum_{[\sigma] \in\Sigma_\Lambda/\Sigma_\Omega}
\sigma\widetilde{\varphi},
\]
where $(\sigma\psi)(a)=\psi(\sigma^{-1}a)$. This is independent of
coset representatives and lands in the $\Sigma_\Lambda$-invariants.

These restriction and transfer maps make $K_F$ into a symmetric monoidal partition
ring. The double coset formula essentially follows from the double coset formula for the Mackey functor $H \mapsto M^H$ for $H \subseteq G$ and $M$ a $G$-module. Frobenius reciprocity follows
from the identity
\[
a\cdot
\Big ( \sum_{[\sigma]\in\Sigma_\Lambda/\Sigma_\Omega}
\sigma\widetilde{\varphi} \Big )
=
\sum_{[\sigma]\in\Sigma_\Lambda/\Sigma_\Omega}
\sigma\left(K_F^*(\Omega\leq\Lambda)(a)\cdot\varphi
\right),
\]
for $a\in K_F(\Lambda)$ and $\varphi\in K_F(\Omega)$ and where extension by $0$ is understood on the right.

The symmetric monoidal structure is induced by disjoint union of
assemblies. Since modules of functions on finite sets are free $K$-modules, and invariant functions on a finite $G$-set identify with functions on its orbit set, the flatness argument of \cref{prop:flatdivsymmon} shows that taking invariants commutes with the tensor product in this case. Thus the canonical bijection
\[
\Asm_F(\Lambda)\times\Asm_F(\Lambda')
\cong
\Asm_F(\Lambda\sqcup\Lambda')
\]
gives the symmetric monoidal multiplication
\[
m_{\Lambda,\Lambda'} \colon K_F(\Lambda)\otimes_K K_F(\Lambda')
\xrightarrow{\cong}
K_F(\Lambda\sqcup\Lambda').
\]

For $i \geq 1$, set
\[
M_{i-1} =\Map(F([i]),K)^{\Sigma_i}.
\]
Consider the $\N$-graded $K$-module $M_* = \oplus_{i \geq 0} M_i$ in which $M_i$ is in degree $i$. Note that
\[
K_F(m) = \Map(\coprod_{\Gamma \parts [m]}
\prod_{B \in \u m/{\sim_\Gamma}} F(B), K)^{\Sigma_m} \cong \bigoplus_{\lambda \parts m} \bigg(
\bigotimes_{i\geq 1}
M_{i-1}^{\otimes \lambda_i}
\bigg)^{
\prod_{i\geq 1}\Sigma_{\lambda_i}
}.
\]
This induces an isomorphism of $\N$-graded $K$-modules

\[
K_F[\Sigma]\cong TS(M_*).
\]
Under this identification, the transfer product on $K_F[\Sigma]$
agrees with the shuffle product on $TS(M_*)$. Hence
\[
K_F[\Sigma]\cong TS(M_*)
\]
as $\N$-graded commutative $K$-algebras.

For a partition $\Lambda$, let
\[
\epsilon_\Lambda \colon K_F(\Lambda)\longrightarrow C_K(\Lambda)=K
\]
be the $K$-linear map
\[
\epsilon_\Lambda(\phi)
=
\sum_{a\in\Asm_F(\Lambda)}\phi(a).
\]
These maps are not compatible with restrictions in general, and they
are not maps of rings for the pointwise multiplication on
$K_F(\Lambda)$. They do, however, commute with transfers and with the
symmetric monoidal structure maps.

Indeed, if $\Omega\leq\Lambda$, then
\[
\epsilon_\Lambda\bigl((K_F)_*(\Omega\leq\Lambda)(\phi)\bigr)
=
|\Sigma_\Lambda/\Sigma_\Omega|\,\epsilon_\Omega(\phi),
\]
which is exactly the transfer in the constant partition functor
$C_K$. Moreover, the canonical bijection
\[
\Asm_F(\Lambda)\times\Asm_F(\Lambda')
\cong
\Asm_F(\Lambda\sqcup\Lambda')
\]
implies that
\[
\epsilon_{\Lambda\sqcup\Lambda'}
\bigl(m_{\Lambda,\Lambda'}(\phi\otimes\psi)\bigr)
=
\epsilon_\Lambda(\phi)\epsilon_{\Lambda'}(\psi).
\]
\cref{prop:transferring} therefore gives a map
\[
\epsilon \colon K_F[\Sigma]\longrightarrow C_K[\Sigma]
\cong K\langle t\rangle.
\]
After completion, this gives
\[
\epsilon \colon K_F[[\Sigma]]
\longrightarrow K\langle\!\langle t\rangle\!\rangle,
\]
which sends $x\in K_F(m)$ to
\[
\epsilon_m(x)\frac{t^m}{m!}.
\]

Let $g_m\in K_F(m)$ be the constant function $1$ on $\Asm_F(m)$, and let $c_m \in K_F(m)$ be the characteristic function on
\[
F([m])\subseteq \Asm_F(m).
\]
Thus $\epsilon_{m}(g_m)$ counts all $F$-assemblies of size $m$, while $\epsilon_{m}(c_m)$ counts
the connected $F$-structures of size $m$.




\begin{prop} \label{prop:combexp}
In $K_F\powser{\Sigma}$, we have the exponential relation
\[
\sum_{m\geq 0} g_mt^m
=
\exp
\Big(
\sum_{i\geq 1}c_it^i
\Big).
\]
\end{prop}
\begin{proof}
We give two proof sketches. First, we can verify the relation directly. Let
$g_\lambda\in K_F(m)$ denote the characteristic function of the
assemblies whose component partition has type $\lambda\vdash m$, and
choose a partition $\Lambda\leq \u m$ with underlying integer partition
$\lambda$. The external product of $\lambda_i$ copies of $c_i$, over all
$i$, is the characteristic function of the assemblies whose component
partition is exactly $\Lambda$. Transferring this function to $K_F(m)$
gives
\[
\prod_i c_i^{\lambda_i}
=
\left(\prod_i\lambda_i!\right)g_\lambda,
\]
since, for a fixed component partition of type $\lambda$, there are
$\prod_i\lambda_i!$ ways to match the blocks of equal size with those
of $\Lambda$. The same multinomial calculation as in the proof of
\cref{exponentialformula} now shows that the coefficient of $t^m$ in
$\exp\Big(\sum_{i\geq1}c_it^i \Big)$ is
\[
\sum_{\lambda\vdash m}g_\lambda=g_m,
\]
since every assembly has a unique component partition.

There is also a conceptual proof using $\Div$. The construction of
\cref{ex:symmetriccoinvariants} extends to graded $K$-modules. Apply it
to the graded module $M_*$ above, with $M_{i-1}$ assigned weight $i$,
and denote the resulting symmetric coinvariant partition functor by
$K^F$. Its transfer quotient in weight $i$ is $M_{i-1}$, and the
calculation of \cref{ex:symmetriccoinvariants} gives
\[
\Div(K^F)(m)
\cong
\bigoplus_{\lambda\vdash m}
\left(
\bigotimes_{i\geq1}M_{i-1}^{\otimes\lambda_i}
\right)^{\prod_{i\geq1}\Sigma_{\lambda_i}}
\cong
K_F(m).
\]
These identifications are compatible with restrictions, transfers, and
the symmetric monoidal structure, so that
\[
K_F\cong\Div(K^F).
\]
Under this isomorphism, the summand corresponding to the indiscrete
partition in degree $i$ consists of functions supported on
$F(\u i)\subseteq\Asm_F(\u i)$, and the distinguished element
$\1_i$ corresponds to $c_i$.

Finally, $\sum_{m\geq0}g_mt^m$ is exponential: restriction of the
constant function $g_{i+j}=1_{i+j}$ to
$\Asm_F(\u i\sqcup\u j)\cong\Asm_F(\u i)\times\Asm_F(\u j)$ is
$g_i\boxtimes g_j$. Since $\Div(K^F)\cong K_F$ is a partition ring, we may use the
more general form of \cref{exponentialformula} noted in \cref{rem:genexp}. As $g_i$ is the multiplicative identity of
$K_F(i)$, we obtain
\[
\sum_{m\geq0}g_mt^m
=
\exp\left(\sum_{i\geq1}(g_i*\1_i)t^i\right)
=
\exp\left(\sum_{i\geq1}c_it^i\right).
\]
\end{proof}

Applying the map
\[
\epsilon \colon K_F\powser{\Sigma} \longrightarrow
K\divpowser{t}
\]
to the exponential relation of \cref{prop:combexp} gives
\[
\sum_{m\geq 0}|\Asm_F(m)|\frac{t^m}{m!}
=
\exp\Big(
\sum_{i\geq 1}|F([i])|\frac{t^i}{i!}
\Big).
\]

\begin{example}
Let $F=\Conn$ be the species of connected finite simple graphs so that $\Conn(X)$ is the set of connected simple graphs with vertex set $X$. Let $\Graph(X)$ be the set of simple graphs with vertex set $X$. Then
\[
\Asm_{\Conn}(X)
=
\coprod_{\Gamma\parts X}
\prod_{B\in X/{\sim_\Gamma}}\Conn(B)
\]
is canonically identified with $\Graph(X)$: a simple graph determines its
partition into connected components, and conversely a partition
together with a connected simple graph on each block determines a simple graph.
Thus
\[
K_{\Conn}(m)
=
\Map(\Graph([m]),K)^{\Sigma_m}.
\]
Under the notation above, $g_m$ is the constant function $1$ on
$\Graph([m])$, and $c_m$ is the characteristic function of
\[
\Conn([m])\subseteq \Graph([m]).
\]
The exponential relation is
\[
\sum_{m\geq0}g_mt^m
=
\exp
\Big(
\sum_{i\geq1}c_it^i
\Big).
\]
Applying $\epsilon$ gives
\[
\sum_{m\geq0}|\Graph([m])|\frac{t^m}{m!}
=
\exp\Big(
\sum_{i\geq1}|\Conn([i])|\frac{t^i}{i!}
\Big).
\]
If $K=\Q$, then $K\divpowser{t}$ identifies with
the ordinary power series ring $\Q\powser{t}$, and this is the usual
labeled exponential formula for graphs.
\end{example}

\subsection{An exotic exponential element in Burnside rings}
\label{sec:exotic-burnside}

Let $A_e$ be the Burnside ring partition power functor
\[
A_e(\Lambda)=A(\Sigma_\Lambda).
\]
This partition ring is not symmetric monoidal (e.g., consider $\Sigma_2 \times \Sigma_2$). We construct an exponential element in $A_e\powser{\Sigma}$ which
does not lie in $\AA\powser{\Sigma}$. 

Recall that the marks homomorphism $\chi \colon A(\Sigma_m) \to \mathrm{Marks}(\Sigma_m,\Z)$ is an injective ring map that sends an isomorphism class of $\Sigma_m$-sets $[X]$ to the function $\chi([X])([H]) = |X^H|$ on the set of conjugacy classes of subgroups of $\Sigma_m$. Let
\[
\omega=[\Sigma_2/e]-1\in A(\Sigma_2).
\]
The marks of $\omega$ are
\[
\chi_e(\omega)=1
\qquad\text{and}\qquad
\chi_{\Sigma_2}(\omega)=-1.
\]
In particular, $\omega$ is a nontrivial unit satisfying $\omega^2=1$.

Let
\[
\cG_m=\setof{U\subseteq [m] \mid |U|=2}
\qquad\text{and}\qquad
N_m=|\cG_m|={m\choose 2}
\]
so $\cG_m$ is the set of edges in the complete graph on the set $[m]$. Note that $\Sigma_m$ acts on $\cG_m$ via its action on $[m]$. In fact, the action of $\Sigma_m$ on $\cG_m$, together with its
action on the two vertices of each edge, determines a homomorphism,
well-defined up to conjugacy,
\[
\rho_m\colon \Sigma_m\longrightarrow \Sigma_2\wr\Sigma_{N_m}.
\]
For $m\geq 2$, define
\[
w_m=\res_{\rho_m}\P_{N_m}(\omega)\in A(\Sigma_m),
\]
and set $w_0=w_1=1$. Here $\P_{N_m} \colon A(\Sigma_2) \to A(\Sigma_2\wr\Sigma_{N_m})$ is the total power
operation.

\begin{prop} \label{prop:burnsidecalc}
Let $H \subseteq \Sigma_m$ be a subgroup.  For $m \geq 2$, let $r_m(H)$ be the
number of $H$-orbits in $\cG_m$ with the property that the setwise
stabilizer of an edge contains an element that interchanges its two vertices. Then
\[
\chi_H(w_m)
=
(-1)^{r_m(H)}.
\]
\end{prop}
\begin{proof}
This is an application of \cite[Corollary 13.8]{cornelius2024imagetotalpoweroperation}. By naturality of marks under restriction,
\[
\chi_H(w_m)
=
\chi_{\rho_m(H)}(\P_{N_m}(\omega)).
\]
The character formula for the total power operation
\cite[Corollary 13.8]{cornelius2024imagetotalpoweroperation} gives
\[
\chi_{\rho_m(H)}(\P_{N_m}(\omega))
=
\prod_{O\in H \backslash \cG_m}
\chi_{(\rho_m(H))_{UU}}(\omega),
\]
where $U$ is a representative of the orbit $O$. Let $H_U$ be the stabilizer of the edge $U$ and let $\alpha_U \colon H_U \to \Sigma_U$ be the action of $H_U$ on the vertices of $U$. In the notation of
\cite[Section 9]{cornelius2024imagetotalpoweroperation}, $(\rho_m(H))_{UU}\leq \Sigma_2$ is the projection to the
$U$-coordinate of the stabilizer of $U$. By construction of $\rho_m$, this
subgroup is precisely
\[
(\rho_m(H))_{UU}=\alpha_U(H_U)\subseteq \Sigma_U\cong\Sigma_2.
\]
Consequently,
\[
\chi_H(w_m)
=
\prod_{O\in H\backslash \cG_m}\chi_{\alpha_U(H_U)}(\omega)
=
(-1)^{r_m(H)},
\]
since $\alpha_U(H_U)=\Sigma_2$ precisely when the setwise stabilizer of $U$
interchanges its two vertices.
\end{proof}

We claim that
\[
\sum_{m\geq 0}w_mt^m\in A_e\powser{\Sigma}
\]
is exponential. By \cref{prop:burnsidecalc}, $\chi_H(w_m)=(-1)^{r_m(H)}$. Now consider $\u i\sqcup\u j \leq \u{i+j}$. There is a
$\Sigma_i\times\Sigma_j$-equivariant decomposition
\[
\cG_{i+j}
\cong
\cG_i\sqcup(\u i\times\u j)\sqcup\cG_j.
\]
The stabilizer of a mixed edge in $\u i\times\u j$ cannot interchange
its vertices. It follows that, for every
$H\leq\Sigma_i\times\Sigma_j$,
\[
r_{i+j}(H)
=
r_i(\operatorname{pr}_{\Sigma_i}H)+r_j(\operatorname{pr}_{\Sigma_j}H).
\]
Consequently,
\[
\chi_H\left(
\res^{\Sigma_{i+j}}_{\Sigma_i\times\Sigma_j}(w_{i+j})
\right)
=
\chi_H(w_i\boxtimes w_j) = \chi_{\operatorname{pr}_{\Sigma_i}H}(w_i)\chi_{\operatorname{pr}_{\Sigma_j}H}(w_j).
\]
Since the marks homomorphism is injective, we have
\[
\res^{\Sigma_{i+j}}_{\Sigma_i\times\Sigma_j}(w_{i+j})
=
w_i\boxtimes w_j.
\]
Thus $\sum_{m\geq 0}w_mt^m$ is exponential.

This exponential element is not produced by the usual power operations
on $A(e)$ which land in $\AA(m) \subseteq A(\Sigma_m)$. The first terms are
\[
w_0=w_1=1,
\qquad
w_2=[\Sigma_2/e]-1,
\qquad
w_3=[\Sigma_3/C_3]-1,
\]
and $C_3 \subset \Sigma_3$ is not a Young subgroup. For the last identity, the marks of $w_3$ are
\[
\begin{array}{c|cccc}
H&e&C_2&C_3&\Sigma_3\\ \hline
\chi_H(w_3)&1&-1&1&-1.
\end{array}
\]



Recall that the Burnside ring global functor $A$ is an ambidextrous global power functor. The transfer $\epsilon_\Lambda\colon A(\Sigma_\Lambda)\longrightarrow\Z$ is given by $\epsilon_\Lambda([X])=|X/\Sigma_\Lambda|$. 
Thus
\cref{cor:ambitriv} gives a map
\[
\widetilde\epsilon\colon
A_e\divpowser{\Sigma}
\longrightarrow
\Z\divpowser{t}.
\]

Burnside's orbit-counting formula gives
\[
\epsilon_{m}(w_m)
=
\frac{1}{m!}
\sum_{\sigma\in\Sigma_m}
\chi_{\langle\sigma\rangle}(w_m).
\]
If $c_{\mathrm{ev}}(\sigma)$ denotes the number of even cycles of
$\sigma$, then
\[
\chi_{\langle\sigma\rangle}(w_m)
=
(-1)^{c_{\mathrm{ev}}(\sigma)}.
\]
Indeed, let $U=\{a,b\}$ be an edge. An element of $\langle \sigma\rangle$ can stabilize $U$ setwise while interchanging $a$ and $b$ only if $a$ and $b$ lie in the same cycle of $\sigma$. If that cycle has length $d$, this occurs precisely when $d$ is even and $a$ and $b$ are separated by $d/2$ steps around the cycle (are ``antipodal''). Thus each even cycle contributes exactly one orbit of such edges, namely the orbit of the pairs of opposite points in that cycle.


The cycle-index identity for symmetric groups is the equation
\[
\sum_{m \geq 0} \Big (\sum_{\sigma \in \Sigma_m} \prod_{k \geq 1} x_{k}^{c_k(\sigma)} \Big ) \frac{t^m}{m!} = \exp \Big (\sum_{k \geq 1}x_k \frac{t^k}{k} \Big ),
\]
where $c_k(\sigma)$ is the number of $k$-cycles in $\sigma$. If we set
\[
x_k = 
\begin{cases}
        1, & k \text{ odd}\\
        -1,     & k \text{ even},
\end{cases}
\]
then $\prod_{k \geq 1} x_{k}^{c_k(\sigma)} = (-1)^{c_{\mathrm{ev}}(\sigma)}$ and we find that
\[
\begin{split}
\sum_{m\geq 0}\epsilon_{m}(w_m)t^m
&=
\exp\bigg (
\sum_{\substack{k\geq 1\\ k\text{ odd}}}\frac{t^k}{k}
-
\sum_{\substack{k\geq 1\\ k\text{ even}}}\frac{t^k}{k}
\bigg)\\
&=
\exp\Big(
\sum_{k\geq 1}(-1)^{k-1}\frac{t^k}{k}
\Big)\\
&=1+t.
\end{split}
\]
Thus $\epsilon_{0}(w_0)=\epsilon_{1}(w_1)=1$ and $\epsilon_{m}(w_m)=0$ for $m\geq 2$. Since $\Z\divpowser{t}$ is torsion-free, the logarithm is unique and we learn that $\ell_k(q_k(w_k)) = (-1)^{k-1}(k-1)!$.

\bibliographystyle{amsalpha}
\bibliography{bib}

\end{document}